\documentclass[a4paper,11pt,reqno]{article}
\usepackage[utf8]{inputenc}
\usepackage[body=17cm,top=3cm,bottom=3cm]{geometry}
 \usepackage{amsrefs}

\usepackage{amsmath,amsfonts,amsthm,amscd,amssymb,mathrsfs,bbm,dsfont}
\usepackage[colorlinks=true]{hyperref}
\usepackage{tikz}

\usepackage{lipsum}

\usepackage{cases}
\usepackage{bbm}
\usepackage{graphicx}
\usepackage{subfig}
\usepackage{float}
\usepackage{mathrsfs}
\usepackage{enumerate}
\usepackage{appendix}
\usepackage{times}
\usepackage{microtype}
\usepackage{mathrsfs}
\usepackage{tikz}
\usepackage{geometry}

\DeclareMathOperator{\expect}{{\mathbb E}}

\DeclareMathOperator{\var}{Var}

\newcommand{\eps}{\varepsilon}

\makeatletter
\newcommand{\eqcolon}{\mathrel{\mathord{=}\raise.2\p@\hbox{:}}}
\newcommand{\coloneq}{\mathrel{\raise.2\p@\hbox{:}\mathord{=}}}
\makeatother
\newcommand{\der}{\delta}
\newcommand{\dd}{\mathrm{d}}
\newcommand{\XX}{\mathbb{X}}
\newcommand{\RR}{\mathbb{R}}

\newcommand{\R}{\mathbb R}
\newcommand{\N}{\mathbb N}
\newcommand{\dD}{\mathscr D}

\newcommand{\cF}{\mathscr{F}}
\newcommand{\cD}{\mathscr{D}}

\newcommand{\cC}{\mathscr{C}}

\newcommand{\TT}{\mathbb{T}}

\newtheorem{theorem}{Theorem}[section]

\newtheorem{corollary}[theorem]{Corollary}

\newtheorem{definition}[theorem]{Definition}

\newtheorem{lemma}[theorem]{Lemma}
\newtheorem{notation}[theorem]{Notation}

\newtheorem{proposition}[theorem]{Proposition}

\newtheorem{remark}[theorem]{Remark}

\newtheorem{ansatz}[theorem]{Ansatz}

\newcommand{\mathd}{\mathrm{d}}
\newcommand{\tmscript}[1]{\text{\scriptsize{$#1$}}}
\DeclareMathAlphabet{\mathpzc}{OT1}{pzc}{m}{it}

\begin{document}
\title{Paracontrolled Distributions and the 3-dimensional Stochastic Quantization Equation} 
\author{R. Catellier\\ 
{\small
IRMAR UMR 6625}\\
{\small  Universit\'e de Rennes 1, France
}\\
\bigskip
{\small remi.catellier@gmail.com}\\
K. Chouk\\
{\small
	Technische Universität Berlin, Germany}\\
{\small
khalil.chouk@gmail.com	}
}

\date{\today}
\maketitle
\begin{abstract}
We prove the existence and uniqueness of a local in time solution to the periodic $\Phi^4_3$  model of stochastic quantisation using the method of paracontrolled distributions introduced recently by  M. Gubinelli, P. Imkeller and N. Perkowski in~\cite{gubinelli_paraproducts_2012}.
\end{abstract}
\tableofcontents 

\section{Introduction}
We study in this work the following Cauchy problem:
\begin{equation}\label{eq:cauchy problem}
\left\{
\begin{aligned}
\partial_tu & =\Delta_{\mathbb T^3}u-u^3+\xi, \\
u(0,x) & =u^{0}(x) \quad x\in\mathbb T^{3},
\end{aligned}
\right.
\end{equation}
where $\xi$ is a space-time white noise such that $\int_{\mathbb T^3}\xi(\cdot,x)\dd x=0$, i.e. it is a centered Gaussian space-time distribution with covariance function defined formally by:
$$
\mathbb E[\xi(s,x)\xi(t,y)]=\der(t-s)\der(x-y).
$$
As we will see in the sequel the solution $u:\RR_+\times \TT^3 \to \RR$ is expected to be a Schwartz distribution in space but not a function, which will give us some trouble to understand the nonlinear part of this equation. Actually the most challenging part of this work is to define the term $u^3$ and to control it in a suitable topological space. 
\smallskip

To look further into this question, let us start by writing this equation in its mild formulation:
\begin{equation}
\label{eq:milde-stoch}
u=P_tu^0-\int_0^tP_{t-s}(u_s)^3\dd s+X_t,
\end{equation}
where $P_t=e^{t\Delta}$ is the heat flow and $X_t=\int_0^tP_{t-s}\xi_s\dd s$ is  the solution of  the linear equation: 
\begin{equation}\label{eq:linear-ou}
\partial_t X_t=\Delta_{\mathbb T^3} X_t+\xi, \quad X_0=0.
\end{equation}
Moreover $X$ is a Gaussian process and as we will see below $X\in C([0,T];\mathcal C^{-1/2-\eps}(\mathbb T^3))$ for every $\eps>0$, where $\mathcal C^{\alpha}=B_{\infty,\infty}^{\alpha}$ is  the Besov-H\"older space of regularity $\alpha$. The main difficulty of  Equation~\eqref{eq:cauchy problem} comes from the fact that for any fixed time $t$ the spatial
regularity of the solution $u(t,x)$ cannot be better than the one of $X_t$. If we measure the spatial
regularity in the scale of Besov-H\"older spaces $\mathcal C^{\alpha}$ we should expect that at best $u(t, x)\in\mathcal C^\alpha(\mathbb T^3)$ for any $\alpha<-1/2$. In particular 
the term $u^3$ is not well-defined. To give a meaning to the equation, a natural approach would consist in regularizing the noise in $\xi^\eps=\xi\star\rho^\eps$, where $\rho^\eps=\eps^{-3}\rho(\frac{.}{\eps})$ is an approximation of the identity, and trying to get a uniform bound in $\eps$ on the solution $u^\eps$ of the approximate equation
 \begin{equation}\label{eq:mollifier-eq}
 \partial_tu^\eps=\Delta u^\eps-(u^\eps)^3+\xi^\eps.
 \end{equation}
Since the non-linear term diverges when $\eps$ goes to zero, an a priori estimate for the wanted solution is difficult to find. To overcome this problem we have to focus on the following modified equation:
\begin{equation}\label{eq:mollify equation}
 \partial_tu^\eps=\Delta u^\eps-((u^\eps)^3-C_{\eps}u^\eps)+\xi^\eps,
\end{equation}
where $C_{\eps}>0$ is a renormalization constant which diverges when $\eps$ goes to 0. We will show that we have to take $C_\eps\sim \frac{a}{\eps}+b\log(\eps)+c$  to obtain a non trivial limit for $u^\eps$. 

Therefore the aim of this work is to give a meaning to the Equation~\eqref{eq:milde-stoch} and to obtain a (local in time) solution. The method developed here uses some ideas of~\cite{hairer_solving_2013} where the author deals with the KPZ equation. More precisely, we will expand the solution as the sum of stochastic objects involving the Gaussian field $X$ and derive an equation for the remainder, which can be solved by a fixed point argument using the notion of paracontrolled distributions introduced in~\cite{gubinelli_paraproducts_2012}. A solution to this equation has already been constructed in the remarkable paper of Hairer~\cite{hairer_theory_2013} where the author shows the convergence of the solution of the mollified equation \eqref{eq:mollify equation}.   
\smallskip 

The stochastic quantization problem has been studied since the eighties in theoretical physics (see for examples \cite{jona-lasinio_stochastic_1985} and \cite{ jona-lasinio_large_1990}).

From a mathematical point of view, several articles deal with the 2-dimensional case. Weak probabilistic solutions have been constructed by Jona-Lasinio and Mitter in~\cite{jona-lasinio_stochastic_1985} and  \cite{ jona-lasinio_large_1990}. Some other probabilistic results are obtained thanks to non perturbative methods by Bertini, Jona-Lasinio and Parrinello in~\cite{bertini_stochastic_1993}. In \cite{da_prato_strong_2003} Da Prato and Debussche give a strong (in the  probabilistic sense) formulation for the $2$-dimensional problem.

In a recent work, Hairer~\cite{hairer_theory_2013} solves the 3-dimensional case thanks to his theory of regularity structures.
Hairer's theory of regularity structures is a generalization of rough path theory. Hairer gets his result by giving a generalization of the notion of pointwise H\"older regularity. With this extended notion, it is possible to work on a more abstract space where the solutions are constructed thanks to a fixed point argument, and then to project the abstract solutions into a space of distributions via the so called reconstruction map. Let us point out that the theory of regularity structures is not specific to the $\Phi^4_3$ model and allows to treat a large class of semilinear parabolic stochastic partial differential equation.

\smallskip

In \cite{gubinelli_paraproducts_2012} the authors have introduced a different approach to handle  singular equations, namely the notion of paracontrolled distributions. Even if this notion is less striking and cannot cover, at the moment, all the local-well posedness results obtained via Hairer's approach, it has the advantage to be elementary and more explicit, which can be useful if we want to tackle the problem of non explosion in time  (see~\cite{HM1,HM2} for such matter).  

\smallskip

We will proceed  in two steps. In an analytic part we will extend the flow of the regular equation   
$$
\partial_t u_t=\Delta u_t-u_t^3+3au_t+9bu_t+\xi_t,
$$
where $(a,b)\in\mathbb R^2$ and $\xi\in C([0,T],C^{\infty}(\mathbb T^3))$ to the situation of more irregular driving noise $\xi$. More precisely we will prove that the solution $u$ is a continuous function of $(u^0,R_{a,b}X(\xi))$ with 
\begin{equation}\label{eq:rough-dist}
\begin{split}
R_{a,b}X(\xi)=&(X,X^2-a,I(X^3-3aX),I(X^3-3aX)\circ X,\\&I(X^2-a)\circ(X^2-a)-b,I(X^3-3aX)\circ(X^2-a)-3bX).
\end{split}
\end{equation}
Here $X_t=\int_0^tP_{t-s}\xi \dd s$, and $f\circ g$ denotes the part of the product between $f$ and $g$ where the two functions have the same frequency (see Proposition~\ref{proposition:Bony-estim} for the exact definition) and $I(f)_t=\int
_0^t P_{t-s}f\dd s$. This extension is given in the following theorem.
\begin{remark}
Let us remark that the vector appearing in the right hand side of~\eqref{eq:rough-dist} does not depend on $\xi$ in the sense that it can be defined for every function $X$ in $C([0,T],C^{\infty}(\mathbb T^2))$. In that case we will keep simply the same notation $R_{a,b}X$ for it.
\end{remark}

\begin{theorem}
\label{Th:Flow-const}
Let $F:\mathcal C^{1}(\mathbb T^3)\times C(\mathbb R^+,\mathcal C^0(\mathbb T^3))\times\mathbb R\times \RR\to C(\mathbb R^+,\mathcal C^{1}(\mathbb T^3))$ be the flow of the equation 
\begin{equation*}
\left\{
\begin{aligned}
&\partial_tu_t=\Delta u_t-u_t^3+3au_t+9bu_t+\xi_t,\quad t\in \big[0,T_C(u^0,X,(a,b))\big),
\\&\partial_tu_t=0,\quad t\geq T_C(u^0,X,(a,b)),
\\& u(0,x)=u^{0}(x)\in\mathcal C^{1}(\mathbb T^3),
\end{aligned}
\right.
\end{equation*}
where $\xi\in C(\mathbb R^+,\mathcal C^0(\mathbb T^3))$ and $T_{C}(u^0,\xi,(a,b))$ is a time such that the equation holds for $t\le T_C$. Now let $z\in(1/2,2/3)$, then there exists a Polish space $\mathcal X$, called the space of rough distributions, $\tilde T_C:\mathcal C^{-z}\times\mathcal X\to\mathbb R^{+}$ a lower semi-continuous function and a function $\tilde F:\mathcal C^{-z}\times\mathcal X\to C(\mathbb R^+,C^{-z}(\mathbb T^3))$ which is continuous in $(u^0,\mathbb X)\in C^{-z}(\mathbb T^3)\times \mathcal X$ such that  $(\tilde F,\tilde T)$ extends $(F,T)$ in the following sense:
$$
T_{C}(u^0,\xi,(a,b))\geq\tilde T_{C}(u^0,R_{a,b}X(\xi))>0
$$
and
$$
F(u^0,\xi,a,b)(t)=\tilde F(u^0,R_{a,b}X(\xi))(t), \mathrm{\ for\ all\ } t\leq\tilde T_{C}(u^0,R_{a,b}X(\xi)), 
$$  
for all $(u^0,\xi)\in\mathcal C^{1}(\mathbb T^3)\times C(\mathbb R^+,\mathcal C^0(\mathbb T^3))$, $(a,b)\in\mathbb R^2$ with $X_t=\int_0^t\dd sP_{t-s}\xi$ and 
where $R_{a,b}^\varphi$ is given in~\eqref{eq:rough-dist}.
\end{theorem}  
In a second part we will obtain probabilistic estimates for the stationary Ornstein Uhlenbeck (O.U.) process  which is the solution of the linear equation~\eqref{eq:linear-ou} and this will allow us to construct the rough distribution in that case.  
\begin{theorem}\label{theorem:main-theo-2}
Let $X$ be the stationary O.U. process and $X^\eps$ be a spatial mollification of $X$ defined by 
$$
X_t^\eps=\sum_{k\in\mathbb Z^3}f(\eps k)\hat X_t(k)e_k,\quad t\geq 0,
$$
where $\hat X$ is the Fourier transform of $X$ in the space variable, $(e_k)_{k\in\mathbb Z^3}$ the Fourier basis of $L^2(\mathbb T^3)$ and $f$ is a smooth function with compact support which satisfies $f(0)=1$.  Then there exists two diverging constants (not unique) $C^\eps_1,C_2^\eps\to^{\eps\to0}+\infty$ such that $R_{C_1^\eps,C_2^\eps}X^\eps$ converges in $L^{p}(\Omega,\mathcal X)$ for all $p>1$. Moreover the limit $\mathbb X\in\mathcal X$ does not depend on the choice of the mollification $f$ and the first component of $\mathbb X$ is $X$. 
\end{theorem}
\begin{remark}
The choice of the constants $C^\eps_1,C^\eps_2$ is not unique and depends in general of the choice of the mollification $f$. However, as it is mentioned in \cite{HW} the constant $C_2^\eps$ can be taken being independent of the mollification.  
\end{remark}
In this setting, the corollary below follows immediately.
\begin{corollary}
Let $\xi$ be a space time white noise, and $\xi^\eps$ be a spatial mollification of $\xi$ such that:
$$
\xi^\eps=\sum_{k\ne0}f(\eps k)\hat \xi(k)e_k,
$$ 
where we have adopted the same assumptions and notations as in Theorem~\ref{theorem:main-theo-2}. Let $X$ be the stationary O.U. process associated to $\xi$, $\mathbb X$ the element of $\mathcal X$ given by  Theorem~\ref{theorem:main-theo-2} and $u^0\in\mathcal C^{-z}$ for $z\in(1/2,2/3)$. Then there exists a sequence of positive time $T^\eps$ which converges almost surely to a random time $T>0$ and such that the solution $u^\eps$ of the mollified equation:
\begin{equation*}
\left\{
\begin{aligned}
&\partial_tu^\eps_t=\Delta u^\eps_t-(u^\eps_t)^3+3C_1^\eps u_t+9C_2^\eps u_t+\xi^\eps_t,\quad t\in[0,T^\eps),
\\&\partial_tu^\eps_t=0,\quad t\geq T^\eps,
\\& u(0,x)=(u^{0})^\eps(x),
\end{aligned}
\right.
\end{equation*}
converges to $\tilde F(u^0,\mathbb X)$. Here the limit is understood in the probability sense in the space $C(\mathbb R^+,\mathcal C^{-z})$. 
\end{corollary}

\paragraph{Plan of the paper.}
The aim of Section~\ref{section:expansion} is to introduce the space of paracontrolled distributions where the renormalized equation will be solved.
In Section~\ref{section:fixed point} we prove that for a small time the application associated to the renormalized equation is a contraction, which, by a fixed point argument, gives the existence and uniqueness of the solution, but also the continuity with respect to the rough distribution and the initial condition.
The last Section~\ref{section:renormalization} is devoted to the existence of the rough distribution for the O.U. process.

\paragraph{Acknowledgments.}
This work was written when the two authors were PhD student at the CEREMADE (Dauphine University, CEREMADE UMR 7534). Both authors are grateful to M. Gubinelli for the extended discussions they had together about this work. We also thank the anonymous referees for their help in greatly improving this paper. 

\section{Paracontrolled distributions}\label{section:expansion}

\subsection{Besov spaces and paradifferential calculus}

The results given in this Subsection can be found in~\cite {BCD-bk} and~\cite{gubinelli_paraproducts_2012}. Let us start by recalling the definition of Besov spaces via the Littelwood-Paley projectors.   

Let $\chi, \theta\in \mathcal D$ be two non-negative radial functions such that
\begin{enumerate}
	\item The support of $\chi$ is contained in a ball and the support of $\theta$ is contained in an annulus;
	\item $\chi(\xi)+\sum_{j\ge0}\theta(2^{-j}\xi)=1 \mathrm{\ for\ all}\ \xi\in\RR^d$;
	\item $\mathrm{supp}(\chi) \cap \mathrm{supp}(\theta(2^{-j}.)) = \emptyset$ for $i \ge 1$ and $\mathrm{supp}(\theta(2^{-j}.)) \cap \mathrm{supp}(\theta(2^{-i}.)) = \emptyset$ when $|i-j| > 1$.
\end{enumerate}
For the existence of $\chi$ and $\theta$ see \cite{BCD-bk}, Proposition 2.10. The Littlewood-Paley blocks are
defined as
$$\Delta_{-1} u = \cF^{-1}(\chi \cF u)\ \mathrm{and\ for}\ j \ge 0, \Delta_j u = \cF^{-1}(\theta(2^{-j}.)\cF u),$$

where $\mathscr F$ denotes the Fourier transform. We define the Besov space of distribution by 
$$
B_{p,q}^{\alpha}=\left\{u\in S'(\mathbb R^d); \quad \|u\|^q_{B_{p,q}^{\alpha}}=\sum_{j\geq-1}2^{jq\alpha}\|\Delta_ju\|^q_{L^p}<+\infty\right\}.
$$
 In the sequel we will deal with the special case of $\mathcal C^{\alpha}:=B_{\infty,\infty}^{\alpha}$ and write $\|u\|_{\alpha}=\|u\|_{B_{\infty,\infty}^{\alpha}}$. We give the following result for the convergence of localized series in Besov spaces, which will prove itself to be useful.
\begin{proposition}
Let $(p,q,s)\in[1,+\infty]^2\times\mathbb R$, $B$ be a ball in $\mathbb R^d$ and $(u_j)_{j\geq-1}$ be a sequence of functions such that $\mathrm{supp}(u_j)$ is contained in $2^jB$, moreover we assume that 
$$
\Xi_{p,q,s}=\left|\left|(2^{js}\|u_j\|_{L^p})_{j\geq-1}\right|\right|_{l^{q}}<+\infty.
$$ 
Then
$u=\sum_{j\geq-1}u_j\in B_{p,q}^s $ and $\|u\|_{B_{p,q}^s}\lesssim\Xi_{p,q,s}$. 
\end{proposition}
 The trick to manipulate  stochastic objects is to deal with Besov spaces with finite integrability exponents and then to go back to the space $\mathcal C^\alpha$. For that we will use the following embedding result:
\begin{proposition}\label{proposition:Bes-emb} 
Let $1\leq p_1\leq p_2\leq +\infty$ and $1\leq q_1\leq q_2\leq +\infty$. For all $s\in \mathbb R$ the space $B_{p_1,q_1}^{s}$ is continuously embedded in $B_{p_2,q_2}^{s-d(\frac{1}{p_1}-\frac{1}{p_2})}$, in particular we have $\|u\|_{\alpha-\frac{d}{p}}\lesssim\|u\|_{B_{p,p}^{\alpha}}$.
 \end{proposition}   
Taking $f\in\mathcal C^{\alpha}$ and $g\in \mathcal C^{\beta}$ we can formally decompose the product as  
$$
fg=f\prec g+f\circ g+f\succ g,
$$
where 
$$
f\prec g=g\succ f=\sum_{j\geq-1}\sum_{i<j-1}\Delta_if\Delta_jg\quad \mathrm{and} \quad f\circ g=\sum_{j\geq-1}\sum_{|i-j|\leq 1}\Delta_if\Delta_jg.
$$
With these notations the following results hold.
\begin{proposition}[Bony estimates]
\label{proposition:Bony-estim}
Let $\alpha,\beta\in\mathbb R$.
\begin{itemize}
\item $\|f\prec g\|_{\beta}\lesssim\|f\|_{\infty}\|g\|_{\beta}$ for $f\in L^\infty$ and $g\in \mathcal C^\beta$.
\item $\|f\succ g\|_{\alpha+\beta}\lesssim\|f\|_{\alpha}\|g\|_{\beta}$  for $\beta<0$, $f\in\mathcal C^\alpha$ and $g\in\mathcal C^\beta$.
\item $\|f\circ g\|_{\alpha+\beta}\lesssim \|f\|_{\alpha}\|g\|_{\beta}$ for $\alpha+\beta>0$ and $f\in\mathcal C^\alpha$ and $g\in\mathcal C^\beta$.
\end{itemize}
A simple consequence of these estimates is that the product $fg$ between two Besov distributions $f\in\mathcal C^\alpha$ and $g\in\mathcal C^\beta$ is well-defined if $\alpha+\beta>0$ moreover it satisfies $\|fg\|_{\min(\alpha,\beta)}\lesssim\|f\|_{\alpha}\|g\|_{\beta}$.
\end{proposition} 
One of the key results of~\cite{gubinelli_paraproducts_2012} is a commutation Lemma for the operator $\prec$ and $\circ$.
\begin{proposition}[Commutator estimate]
\label{proposition:comm-1}
Let $\alpha, \beta,\gamma\in\mathbb R$ be such that $\alpha<1$,  $\alpha+\beta+\gamma>0$ and $\beta+\gamma<0$ then
$$
R(f,x,y)=(f\prec x)\circ y-f(x\circ y)
$$
is well-defined for $f\in\mathcal C^\alpha$, $x\in\mathcal C^\beta$ and $y\in\mathcal C^\gamma$. More precisely 
$$
\|R(f,x,y)\|_{\alpha+\beta+\gamma}\lesssim\|f\|_{\alpha}\|x\|_{\beta}\|y\|_{\gamma}.
$$
\end{proposition}
We end this section by describing the action of the heat flow on  Besov spaces and by giving a commutation property with the paraproduct. See the appendix for a proof.
\begin{lemma}[Heat-flow estimates]
\label{lemma:Heat-flow-smoothing}
Let $\theta\geq0$ and $\alpha\in\mathbb R$. The following inequalities:
$$
\|P_tf\|_{\alpha+2\theta}\lesssim\frac{1}{t^{\theta}}\|f\|_{\alpha},\quad\|(P_{t-s}-1)f\|_{\alpha-2\eps}\lesssim|t-s|^{\eps}\|f\|_{\alpha},
$$
holds for all $f\in\mathcal C^\alpha$. 
Moreover if $\alpha<1$ and $\beta\in\mathbb R$ we have 
$$
\|P_t(f\prec g)-f\prec P_tg\|_{\alpha+\beta+2\theta}\lesssim\frac{1}{t^{\theta}}\|f\|_\alpha\|g\|_\beta,
$$
for all $g\in\mathcal C^\beta$. 
\end{lemma}
Let us now introduce some notations for functional spaces which will be used extensively in the sequel of this paper. 
\begin{notation}
$$
C^{\beta}_T = C([0,T];\mathcal C^\beta).
$$
For $f\in C^\beta_T$ we introduce the norm 
$$
\|f\|_{\beta}=\sup_{t\in[0,T]}\|f_t\|_{\mathcal C^\beta}=\sup_{t\in[0,T]}\|f_t\|_{\beta}
$$
and the space
$$C_T^{\alpha,\beta}:=C^{\alpha}([0,T],\mathcal C^{\beta}(\mathbb T^3)).$$
We denote the space of $\alpha$-H\"older functions in time with value in the Besov space $\mathcal C^\beta$, where $\alpha>0$. Furthermore, we endow this space with the following distance
\[d_{\alpha,\beta}(f,g)=\sup_{t\neq s \in [0,T]}\frac{\|(f-g)_t-(f-g)_s\|_\beta}{|t-s|^\alpha}+\sup_{t\in[0,T]}\|f_t-g_t\|_{\beta}.\]
\end{notation}
Let us end this section by giving a proposition which is a consequence of Lemma~\ref{lemma:Heat-flow-smoothing}, and in which  we describe the action of the operator $I$ on the Besov spaces.  
\begin{proposition}[Schauder estimates]
\label{prop:schauder}
Let $\beta\in\mathbb R$, $f\in C_{T}^{\beta}$ and $I(f)(t)=-\int_0^tP_{t-s}f_s\dd s$, then the following bound holds fo all $\theta<1$
$$
\|I(f)\|_{C_T^{\beta+2\theta}}\lesssim T^{1-\theta}\|f\|_{C_T^{\beta}}.
$$ 
Moreover if $\alpha\in(0,1)$, $\beta>0$,   $f\in C_T^{\alpha,\beta}$ and $g\in C_{T}^{\gamma}$ then the following commutation estimate holds
$$
\|I(f\prec g)-f\prec I(g)\|_{C_T^{\gamma+2\theta}}\lesssim T^{\kappa}\|f\|_{C_T^{\alpha,\beta}}\|g\|_{\mathcal{C}_T^\gamma},
$$ 
 for all $\theta<\min(\alpha,\beta)+1$ with $\kappa=\min(1-\theta+\beta/2,1-\theta+\gamma)$.
\end{proposition}
\begin{proof}
Only the second estimate requires a proof, since the first one is an immediate consequence of the heat-flow estimates. To get the second bound let us observe that 
$$
I(f\prec g)(t)-f\prec I(g)(t)=I_1(t)+I_2(t),
$$
where 
$$
I_1(t)=\int_0^t(f_s\prec P_{t-s}g_s)-P_{t-s}(f_s\prec g_s)\dd s \quad \mathrm{and} \quad I_2(t)=\int_0^t(f_t-f_s)\prec P_{t-s}g_s\dd s.
$$
Now using the heat-flow estimate we have that 
$$
\|I_1(t)\|_{\mathcal C^{\gamma+\beta+2\theta'}}\lesssim \int_0^t\|(f_s\prec P_{t-s}g_s)-P_{t-s}(f_s\prec g_s)\|_{\mathcal C^{\beta+\gamma+2\theta'}}\dd s\lesssim \left(\int_0^t(t-s)^{-\theta'}\dd s\right)\|f\|_{C_T^{\alpha,\beta}}\|g\|_{C_T^\gamma},
$$
for all $\theta'<1$. The second inequality is obtained by using the heat-flow estimates. Taking $2\theta=2\theta'+\beta$ gives the needed bound for $I_1$. The bound of $I_2$ is a consequence of the the Bony estimate for the paraproduct term, the H\"older regularity in time of $f$ and the heat flow estimate. Indeed we have that 
$$
\|I_2(t)\|_{\mathcal C^{\gamma+2\theta}}\lesssim\int_0^t\|(f_t-f_s)\|_{\mathcal C^{\beta}}\|P_{t-s}g_s\|_{\mathcal C^{\gamma+2\theta}}\dd s\lesssim T^{1-(\theta-\gamma)}\|f\|_{C_T^{\alpha,\beta}}\|g\|_{C_T^\gamma}. 
$$
Which ends the proof.
\end{proof}
\subsection{Description of the strategy and renormalized equation}
Let us focus on the mild formulation of the equation \eqref{eq:cauchy problem}
\begin{equation}\label{equation_mild}
u=\Psi+X+I(u^3)=X+\Phi, 
\end{equation}
where we recall the notation  $I(f)(t)=-\int_0^tP_{t-s}f_s\dd s$, $X=-I(\xi)$ and $\Psi_t=P_tu^0$ for $u^0\in\mathcal C^{-z}(\mathbb T^3)$. We can see that a solution $u$ must have the same regularity as $X$. However it is well-known that for all $\varepsilon>0$, we have $X\in C([0,T],\mathcal C^{-1/2-\varepsilon})$ (see  Section~\ref{section:renormalization} for a quick proof of this fact).  But in that case the nonlinear term $u^3$ is not well-defined, as there is no universal notion for the product of distributions.  
A first idea is to proceed by regularization of $X$, such that products of the regularized quantities are well-defined, and then try to pass to the limit. Let us recall that the stationary O.U process $(\hat X_t(k))_{t\in\mathbb R,k\in\mathbb Z^3}$ is a centered Gaussian process with covariance function given by 
$$
\mathbb E\left[\hat X_t(k)\hat X_s(k')\right]=\der_{k+k'=0}\frac{e^{-|k|^2|t-s|}}{|k|^2}
$$ 
and $\hat X_t(0)=0$. Let $X_t^\eps$ be the mollification of $X$ introduced in  Theorem~\ref{theorem:main-theo-2}. Then the following approximated equation
$$\Phi^\eps = \Psi^\eps + I((X^\eps)^3)+3I((\Phi^\eps)^2 X^\eps)+3I(\Phi^\eps (X^\eps)^2) + I((\Phi^\eps)^3),$$
where $\Phi^\eps=I((u^\eps)^3)+\Psi^\eps$, is well-posed. Before proceeding into the analysis of this equation let us observe that a straightforward computation gives
\begin{equation*}
\begin{split}
C_1^\eps:=\mathbb E\left[(X_t^\eps)^2\right]&=\sum_{k\in\mathbb Z^3-\{0\}}\sum_{k_1+k_2=k}f(\eps k_1)f(\eps k_2)\frac{1}{|k_1|^2}\der_{k_1+k_2=0}
\\&=\sum_{k\in\mathbb Z^3-\{0\}}\frac{|f(\eps k)|^2}{|k|^2}\sim_{0}\frac{1}{\eps}\int_\mathbb Rf(x)|x|^{-2}\dd x.
\end{split}
\end{equation*}
Here $A_{\eps}\sim_{0}B_{\eps}$ means that when $\eps$ is close to $0$,  the quantity $A_\eps$ can be be bounded from below and above by a positive constant times $B_\eps$. Then there is no hope of obtaining a finite limit for $(X^\eps)^2$ when $\eps$ goes to zero. This difficulty has to be solved by subtracting from the original equation these problematic contributions. In order to do so consistently, we will introduce a renormalized product. Formally we would like to define
$$
 X^{\diamond 2} =X^2-\mathbb E[X^2]
$$
and show that it is well-defined and that $ X^{\diamond 2} \in \mathcal C_T^{-1-\delta}$ for $\delta >0$. Precisely we will introduce 
\[(X^\varepsilon)^{\diamond 2}=(X^{\eps})^2-\underbrace{\expect[(X^{\eps})^2]}_{=:C^\eps_1}\]
and we will prove that it converges to some finite limit. It would be wise to remark that  many other terms need to be renormalized and subtracting the constant $C_1^\eps$ is not enough to take care of them. Indeed as we will see in  Section~\ref{section:renormalization}, a second renormalization constant $C_2^\eps$ is needed. Including such considerations in the approximated equation gives rise to an algebraic renormalization term which takes the form $-C^\eps I(\Phi^\eps+X^\eps)$ with $C^\eps=3(C_1^\eps-3C_2^\eps)$. More precisely we will study the following equation:     
\begin{align*}
\Phi^\eps &= \Psi^\eps + I((X^\eps)^3)+3I((\Phi^\eps)^2 X^\eps)+3I(\Phi^\eps (X^\eps)^2) + I((\Phi^\eps)^3) - C^\eps I(\Phi^\eps+X^\eps)\\
&=\Psi^\eps + I((X^\eps)^3-3C^\eps_1 X^\eps)+3I((\Phi^\eps)^2 X^\eps)+3I(\Phi^\eps ((X^\eps)^2-C^\eps_1)+9C^\eps_2(\Phi^\eps+X^\eps)) + I((\Phi^\eps)^3),
\end{align*}
or in an other form:  
\begin{equation}\label{eq:renormalized approximated}
\Phi^{\eps} =\Psi^\eps + I((X^{\eps})^{\diamond 3}) + 3I((\Phi^\eps)^2 X^\varepsilon)+3I(\Phi^\varepsilon\diamond(X^\varepsilon)^{\diamond 2}) + I((\Phi^\eps)^3).
\end{equation}
We have adopted the following notation:
$$
I(\Phi^\varepsilon\diamond(X^\varepsilon)^{\diamond 2}):=3I(\Phi^\eps ((X^\eps)^{\diamond2}))+9C_2^\eps I(\Phi^\eps+X^\eps)
$$
and
$$
 I((X^{\eps})^{\diamond 3})= I((X^\eps)^3-3C^\eps_1 X^\eps).
$$
A brief analysis of the wanted regularity for the involved objects shows that even if the terms $I\big((X^\eps)^{\diamond 2}\big)$ and $I\big((X^\eps)^{\diamond 3})$ converge in the suitable spaces, the renormalization introduced before is not enough to define the equation. Indeed, from Section~\ref{section:renormalization}, one can see that for all $\delta>0$, $X^\eps$ converges in probability (and even almost surely) in the space $\mathcal C_T^{-1/2-\delta}$, $(X^\eps)^{\diamond 2}$ converges in the space $\mathcal C_T^{-1-\delta}$ and the term $I\big((X^\eps)^{\diamond 3}\big)$ converges in the space $\mathcal C_T^{1/2-\delta}$. 
So we can expect that the presumed limit $\Phi$ of the solution $\Phi^\eps$ has the same regularity as the worst term in the last equation. Hence, $\Phi\in \mathcal C_T^{1/2-\delta}$ and the estimates of Proposition~\ref{proposition:Bony-estim} are not enough to take care of the terms $X^{\diamond 2}\Phi$ and $X\Phi^2$, since the sum of the regularities are still negative. Nevertheless, it is expected from Section~\ref{section:renormalization} that if those terms are constructed they lie respectively in $\mathcal C_T^{-1-\delta}$ and $\mathcal C_T^{-1/2-\delta}$. One expects from Section~\ref{section:renormalization} that the solution $\Phi^\eps$ may converge as soon as it is expressed as functional of "purely stochastic" terms. In order to have such a decomposition we will use extensively the definition of the paraproduct and the commutator estimates. Proposition~\ref{proposition:Bony-estim} allows us to deal with products of factors as soon as the sum of the spatial regularity is positive. The commutator and Schauder estimates (Propositions~\ref{proposition:comm-1} and~\ref{prop:schauder}) allows us to decompose the analytically ill-defined terms into purely stochastic factors and well-defined terms.
\bigskip

For the sake of  better comprehension, we will consider only the null initial condition. It does not change the algebraic part which is exposed here, but push us  to deal with space of continuous functions for strictly positive time which can blowup at the origin. The equation becomes
\begin{equation}\label{eq:screu}\Phi^\eps = \underbrace{I\big((X^\eps)^{\diamond 3} \big)}_{\mathcal C_T^{1/2 - \delta}} 
+ \underbrace{3I\big((X^\eps)^{\diamond 2}\Phi^\eps\big)}_{\mathcal C^{1-\delta}_T}
+\underbrace{9C_2^\eps I\big(X^\eps+\Phi^\eps\big)}_{\mathcal C^{3/2-\delta}_T}
+ \underbrace{3I\big(X^\eps (\Phi^\eps)^2\big)}_{\mathcal C_T^{3/2-\delta}} 
+ \underbrace{I\big((\Phi^\eps)^3\big)}_{\mathcal C_T^{3/2-\delta}}.\end{equation}
The form of the equation suggests to make the following ansatz about the a priori expression of the solution $\Phi^\eps$:
\begin{ansatz}\label{ansatz:young}
We suppose that there exists $(\Phi^\eps)^\flat$ such that $\Phi^\eps = I\big((X^\eps)^{\diamond 3}\big) + (\Phi^\eps)^{\flat}$ and where for all $\delta>0$ (small enough) the reminder $(\Phi^\eps)^\flat$ is uniformly bounded (in $\eps$) in the space $\mathcal C^{1-\delta}_T$.
\end{ansatz}
If $\Phi^\eps$ fulfills Ansatz~\ref{ansatz:young}, one can develop the fourth term of the right hand side and obtain
\[X^\eps (\Phi^\eps)^2 = 2 X^\eps (\Phi^\eps\prec\Phi^\eps) + X^\eps  (\Phi^\eps \circ \Phi^\eps). \]
As $\Phi^\eps \in \mathcal C_T^{1/2-\delta}$ and $X^\eps\in\mathcal C_T^{-1/2-\delta}$, the term $X^\eps  (\Phi^\eps \circ \Phi^\eps)$ is well-defined thanks to  Proposition~\ref{proposition:Bony-estim}. It is possible to develop the first one a bit further to have that 
\[X^\eps (\Phi^\eps\prec\Phi^\eps) = X^\eps \circ (\Phi^\eps\prec\Phi^\eps)  + \Big(X^\eps \prec(\Phi^\eps\prec\Phi^\eps)+ (\Phi^\eps\prec\Phi^\eps)\prec X^\eps)\Big). \]
Here again, the only ill-defined term may be the first one. Hopefully, the regularities of the objects allows us to use the commutator estimate of Proposition~\ref{proposition:comm-1}. We then use of Ansatz~\ref{ansatz:young} once again to get
\begin{align*}
X^\eps \circ (\Phi^\eps\prec\Phi^\eps) &= \Phi^\eps (\Phi^\eps \circ X^\eps) + R(X^\eps,\Phi^\eps,\Phi^\eps)\\
& = \Phi^\eps \Big(I\big(X^\eps\big)^3\circ X^\eps + (\Phi^\eps)^\flat\circ X^\eps\Big) + R(X^\eps,\Phi^\eps,\Phi^\eps).
\end{align*}
 Hence, Ansatz~\ref{ansatz:young} allows us to see the product $(\Phi^\eps)^2 X^\eps$ as a continuous functional of $(\Phi^\eps)^\flat$ and some stochastic well-defined terms. This fact is summarized in the following proposition:
 \begin{proposition}\label{proposition:young}
 Let $\Phi^\eps$ be as in Ansatz~\ref{ansatz:young}, then $I((\Phi^{\eps})^2 X^\eps)$ is a continuous functional (bounded uniformly in $\eps$) of $(\Phi^\eps)^\flat$, $X^\eps$, $I\big((X^\eps)^{\diamond 3}\big)$, $(X^\eps)^{\diamond 2}$ and $I\big((X^\eps)^{\diamond 3}\big)\circ X^\eps$. Moreover if all these data have a finite limit in the prescribed space then $I((\Phi^{\eps})^2 X^\eps)$ is also convergent. 
 \end{proposition}
The aim of Subsection~\ref{subsection2.4}, and in particular Proposition~\ref{prop:choiceL1} is to specify the dependencies towards the norm of each object. Furthermore, thanks to Section~\ref{section:renormalization}, the term $I((X^\eps)^{\diamond3})\circ X^\eps$ converges in probability in the suitable space.
\begin{notation}
Since the eventual limit of $(\Phi^\eps)^{2}X^\eps$ is not simply a continuous functional of $X$ (but also of the eventual limit of the stochastic terms appearing in Proposition~\ref{proposition:young}) it would be wise to denote the limiting object by $\Phi^2\diamond X$ instead of $\Phi^2X$ to keep this fact in mind. 
\end{notation}
\bigskip

Unfortunately, Ansatz~\ref{ansatz:young} is not enough to handle the product $ \Phi^\eps(X^\eps)^{\diamond 2}$. Indeed, $(X^\eps)^{\diamond 2} \in \mathcal C^{-1-\delta}_T$ and the remainder $(\Phi^\eps)^{\flat} \in \mathcal C^{1-\delta}_T$. Hence, one has to develop Equation~\eqref{eq:screu} a bit further. We still assume that $\Phi^\eps$ complies with Ansatz~\ref{ansatz:young}. From the paraproduct decomposition, we can see that  
\begin{equation}\label{eq:screugneu}
\begin{aligned}
\Phi^\eps 
= &\underbrace{I\big((X^\eps)^{\diamond 3} \big)}_{\mathcal C_T^{1/2 - \delta}} 
+ \underbrace{3I\big(\Phi^\eps\prec(X^\eps)^{\diamond 2}\big)}_{\mathcal C^{1-\delta}_T}\\ 
&+\underbrace{3I\big(\Phi^\eps\circ(X^\eps)^{\diamond 2}\big)}_{\mathcal C^{3/2-\delta}_T}- \underbrace{9C_2^\eps I(\Phi^\eps+X^\eps)}_{\mathcal C^{3/2-\delta}}\\
&+
\underbrace{3I\big((X^\eps)^{\diamond 2}\succ\Phi^\eps\big)}_{\mathcal C^{3/2-\delta}_T} 
+ \underbrace{3I\big(X^\eps (\Phi^\eps)^2\big)}_{\mathcal C_T^{3/2-\delta}}
 + \underbrace{I\big((\Phi^\eps)^3\big)}_{\mathcal C_T^{3/2-\delta}}.
 \end{aligned}
\end{equation}
Let us observe that the only "ill-defined" term in this expansion is $I\big(\Phi^\eps\circ(X^\eps)^{\diamond 2}\big)$. Nevertheless, one can make a second stronger ansatz about the representation of the solutions in terms of functions with increasing regularity:
\begin{ansatz}[Paracontrol Ansatz]\label{ansatz:paracontrol}
We suppose that there exists $(\Phi^\eps)'$ such that:
\[\Phi^\eps = I\big((X^\eps)^{\diamond 3}\big) + 3 I\big((\Phi^\eps)'\prec (X^\eps)^{\diamond 2}\big) + (\Phi^\eps)^\sharp,\]
where for all $\delta>0$ and all $\nu>0$ small enough $(\Phi^\eps)'\in\mathcal C_T^{1/2-\delta}$ and $(\Phi^\eps)^{\sharp}\in\mathcal C_T^{1+\nu}$ uniformly in $\eps$. Moreover as we will see in the sequel some H\"older regularity in time is also needed for the term $(\Phi^\eps)'$ and actually we will assume that this term is uniformly bounded (in $\eps$) in the space $C^{\delta',1/2-\delta}_{T}$ for $\delta'<\delta/2$.
\end{ansatz}
This ansatz is an informal form of the definition of the paracontrolled distributions (Definition~\ref{def:controlled-distrib}). It will allow us to prove an analog of Proposition~\ref{proposition:young} but for $I\big(\Phi^\eps\circ(X^\eps)^{\diamond 2}\big)$. 
First, remind that $I(f) = -\int_0^t P_{t-s}f_s \dd s$. Thanks to the commutator estimate of Proposition~\ref{proposition:comm-1}, the H\"older regularity in time of $(\Phi^\eps)'$ and the Schauder estimate of Proposition~\ref{prop:schauder}, we have that for all $\nu>0$ small enough: 
\[I\big((\Phi^\eps)'\prec (X^\eps)^{\diamond 2}\big)-(\Phi^\eps)'\prec I(X^\eps)^{\diamond 2} \in\cC_T^{1+\nu}.\]
Moreover this quantity is continuous with respect to $(\Phi^\eps)'$ and $(X^\eps)^{\diamond 2}$.
Hence, one can reformulate Ansatz~\ref{ansatz:paracontrol} in the following way: For all $\delta,\nu>0$ small enough, there exists $(\Phi^\eps)' \in\cC_T^{1/2-\delta}$ and $(\Phi^\eps)^\natural$ such that 
\begin{equation}\label{eq:Ansatzrefo}
\Phi^\eps = I\big((X^\eps)^{\diamond 3}\big) + 3 (\Phi^\eps)'\prec I\big((X^{\eps})^{\diamond 2}\big) +  (\Phi^\eps)^\natural.
\end{equation}

Again, the only ill-defined term in $\Phi^\eps(X^\eps)^{\diamond 2}$ is the resonant term $\Phi^\eps\circ(X^\eps)^{\diamond 2}$. Using the reformulation~\eqref{eq:Ansatzrefo} of  Ansatz~\ref{ansatz:paracontrol} we get easily that 
\begin{equation*}
 \begin{aligned}
 (\Phi^\eps)\circ(X^\eps)^{\diamond 2} 
 =& I\big((X^\eps)^{\diamond 3}\big)\circ (X^\eps)^{\diamond 2} + 3 (X^\eps)^{\diamond 2}\circ \Big((\Phi^\eps)'\prec I\big((X^\eps)^{\diamond 2}\big)\Big)+(\Phi^\eps)^\natural\circ (X^\eps)^{\diamond 2}\\
 =& I\big((X^\eps)^{\diamond 3}\big)\circ (X^\eps)^{\diamond 2} + (\Phi^\eps)' \Big(I\big((X^\eps)^{\diamond 2}\big)\circ (X^\eps)^{\diamond 2}\Big)\\
 &+R\Big((X^\eps)^{\diamond 2},(\Phi^\eps)',I\big((X^\eps)^{\diamond 2})\Big)+(\Phi^\eps)^\natural\circ (X^\eps)^{\diamond 2},\\
 \end{aligned}
 \end{equation*}
where we have used the commutator estimate of Proposition~\ref{proposition:comm-1}. All the regularities of the involved objects are enough to take the limit in the product, as soon as $I\big((X^\eps)^{\diamond 2}\big)\circ (X^\eps)^{\diamond 2}$ and $I\big((X^\eps)^{\diamond 3}\big)\circ (X^\eps)^{\diamond 2}$ converge in the prescribed space. Unfortunately as it is shown in  Section~\ref{section:renormalization}, this is not true. However the convergence holds after making a renormalization procedure and this is where the constant $C_2^\eps$ takes its role. More precisely we will consider the following stochastic terms:
$$
\left(I\big((X^\eps)^{\diamond 2}\big)\circ (X^\eps)^{\diamond 2}\right)^{\diamond}=I\big((X^\eps)^{\diamond 2}\big)\circ (X^\eps)^{\diamond 2}-C_2^\eps
$$ 
and
$$
\left(I\big((X^\eps)^{\diamond 3}\big)\circ (X^\eps)^{\diamond 2}\right)^{\diamond}=I\big((X^\eps)^{\diamond 3}\big)\circ (X^\eps)^{\diamond 2}-3C_2^\eps X^\eps,
$$
with 
$$
C_{2}^\eps=\mathbb E\left[I\big((X^\eps)^{\diamond 2}\big)(t)\circ (X^\eps)^{\diamond 2}(t)\right]\big|_{t=0}.
$$
Making such a consideration push us to consider the term $ (\Phi^\eps)\circ(X^\eps)^{\diamond 2} -3C^\eps_2 (X^\eps + \Phi^\eps)$ instead of the original one where we have added the extra counterpart introduced in Equation~\eqref{eq:screu} and, at this point, we can see that this term has the following expansion: 
\begin{equation*}
\begin{aligned}
 (\Phi^\eps)\circ(X^\eps)^{\diamond 2} -3C^\eps_2 (X^\eps + \Phi^\eps)=
    & I\big((X^\eps)^{\diamond 3}\big)\circ (X^\eps)^{\diamond 2} - 3C^\eps_2 X^\eps \\
&     + 3\bigg((\Phi^\eps)' \Big(I\big((X^\eps)^{\diamond 2}\big)\circ (X^\eps)^{\diamond 2}\Big) - C_2^\eps \Phi^\eps\bigg)\\
&+R\Big((X^\eps)^{\diamond 2},(\Phi^\eps)',I\big((X^\eps)^{\diamond 2})\Big)+(\Phi^\eps)^\natural\circ (X^\eps)^{\diamond 2}.
 \end{aligned}
 \end{equation*}
It is important to remember that $\Phi^\eps$ is expected to be a fixed point of the equation. In that setting, one must note that $\Phi^\eps = (\Phi^\eps)'$. The following proposition is a summary of the discussion above (a rigorous proof of it, and more precise estimates can be found in Subsection~\ref{subsection2.5}). 
\begin{proposition}
Suppose that $\Phi^\eps$ fulfills Ansatz~\ref{ansatz:paracontrol}, then $\Phi^\eps (X^\eps)^{\diamond 2}-3C^\eps_2 (X^\eps+\Phi^\eps)$ is a continuous functional, uniformly in $\eps$, of $\Phi^\eps$, $(\Phi^\eps)'$, $X^\eps$, $(X^\eps)^{\diamond 2}$, $I\big((X^\eps)^{\diamond 3}\big)$, $\left(I\big((X^\eps)^{\diamond 2}\big) \circ (X^\eps)^{\diamond 2}\right)^{\diamond}$ and $\left(I\big((X^\eps)^{\diamond 3}\big) \circ (X^\eps)^{\diamond 2}\right)^{\diamond}$.
 \end{proposition}
The following corollary is a byproduct of those two propositions:
\begin{corollary}
\label{corollary:contin}
 Let $\Phi^\eps$ be as in Ansatz~\ref{ansatz:paracontrol} and 
 \[
  \XX^{\eps} = \Big(
  X^\eps,(X^\eps)^{\diamond 2}, I\big((X^\eps)^{\diamond 3}),(I\big((X^\eps)^{\diamond 3})\circ X^\eps,
   I\big((X^\eps)^{\diamond 2})\circ(X^\eps)^{\diamond 2})^{\diamond},\left(I\big((X^\eps)^{\diamond 3})\circ(X^\eps)^{\diamond 2}
  \right)^{\diamond}\Big).\]
  The function $\Gamma$ is continuous toward $\Phi^\eps$, $(\Phi^\eps)'$ and $\XX^\eps$, uniformly in $\eps$, where 
\[  \Gamma(\Phi^\eps) = I\big((X^\eps)^3\big) + 3I\big((X^\eps)^{2}\Phi^\eps\big) + I\big(X^\eps (\Phi^\eps)^2\big) + I\big((\Phi^\eps)^3\big) - 3 (C^\eps_1+C^\eps_2+C_3) I\big(X^\eps + \Phi^\eps).\]
 \end{corollary}
 As mentioned above, the aim of Subsections~\ref{subsection2.4} and~\ref{subsection2.5} is to specify the dependency towards the parameters of the problem. In Section~\ref{section:fixed point}, these estimates will allow us to prove that in a suitable space (the space of paracontrolled distribution of Definition~\ref{def:controlled-distrib}) and for a suitable $\XX$ (lying in the space of rough distributions of Definition~\ref{hyp:renormalization1}) the application $\Gamma$ is a contraction. This will allow us to make a fixed point argument.
 Finally, in Section~\ref{section:renormalization} we apply this analytical theory to the white noise and to $X$. 

\begin{remark}
 Let us remark that this analysis for $\Phi^\eps$ leads to the corresponding problem for $u^\eps$:
 \[\partial_t u^\eps = \Delta u^\eps - \big((u^\eps)^3 - 3(C^\eps_1+3C^\eps_2+C_3) u^\eps\big) + \xi^\eps.\]
 In all the following, we have implicitly chosen to take $C_3=0$. It does not change the resolution of the problem to take $C_3\neq 0$.
 \end{remark}

To end this section let us introduce the space $\mathcal X$ in which the convergence of $\mathbb X^{\eps}$ takes place.  
\begin{definition}\label{hyp:renormalization1}
Let $T>1$, $\nu,\rho>0$. We denote by $\cC_T^{\nu,\delta',\beta}$ the closure of the set of smooth functions $C^{\infty}([0,T],\mathcal C^\beta(\mathbb T^3))$ by the semi-norm:
$$
\|\varphi\|_{\nu,\rho}=\sup_{t\in[0,T]}t^{\nu}|\varphi_t|_{\mathcal C^\beta}+\sup_{t,s\in[0,T];s\ne t}\frac{s^{\nu}|\varphi_t-\varphi_s|_{\mathcal C^\beta}}{|t-s|^{\delta'}}.
$$ 
For $0<4\delta'<\delta$ we define the normed Banach space $\mathcal W_{T,K}$ 
$$
\mathcal W_{T,K}=C_T^{\delta',-1/2-\delta}\times C_{T}^{\delta',-1-\delta}\times C_{T}^{\delta',1/2-\delta}\times C_T^{\delta',-\delta}\times\mathscr C_{T}^{\nu,\delta',-\delta}\times \mathscr C_T^{\nu,\delta',-1/2-\delta},
$$
where $K=(\delta,\delta',\nu,\rho)$. This Banach space is equipped with the product topology. For $X\in C([0,T],C(\mathbb T^3))$, and $(a,b)\in\mathbb R^2$ we define $R_{a,b}X\in\mathcal W_{T,K}$ by
\begin{equation*}
\begin{split}
R_{a,b}X=&(X,X^2-a,I(X^3-3aX),I(X^3-3aX)\circ X,\\&I(X^2-a)\circ(X^2-a)-b,I(X^3-3aX)\circ(X^2-a)-3bX).
\end{split}
\end{equation*}
The space of  rough distributions $\mathcal X_{T,K}$ is defined as the closure in $\mathcal W_{T,K}$ of the set 
$$
\left\{R_{a,b}X,\quad X\in C([0,T],C(\mathbb T^3));(a,b)\in\mathbb R^2\right\}.
$$  
\smallskip
For a generic element $\mathbb X\in\mathcal X$ we denote its components by 
$$
\mathbb X=(X,X^{\diamond2},I(X^{\diamond3}),I(X^{\diamond3})\circ X,(I(X^{\diamond2})\circ X^{\diamond2})^{\diamond},(I(X^{\diamond3})\circ X^{\diamond2})^{\diamond}).
$$ 
We equip the space $\mathcal X_{T,K}$ by the metric $\mathrm d$ induced by the topology of the Banach space $\mathcal W_{T,K}$ and we denote simply by $\|\mathbb X\|_{T,K}$ the norm of $\mathbb X$ in the space $\mathcal W_{T,K}$. For simplicity we will omit in the sequel the dependency in $T$ and $K$ for the space defined above and simply write $\mathcal X$. 
\end{definition}
\begin{remark}
For $\mathbb X\in\mathcal X_{T,K}$ we can obviously construct $\left (I(X^{\diamond2})X^{\diamond2}\right)^{\diamond}$ using the Bony paraproduct decomposition in the following way 
$$
\big(I(X^{\diamond2})X^{\diamond2})^{\diamond}=I(X^{\diamond2})\prec X^{\diamond2}+I(X^{\diamond2})\succ X^{\diamond2}+(I(X^{\diamond2})\circ X^{\diamond2})^{\diamond}.
$$
This could also be done for $\left (I(X^{\diamond3})X^{\diamond2}\right)^{\diamond}$. In the sequel we might abusively denote $\mathbb X$ by $X$ if there is no confusion, and as in the rough paths terminology we denote the other components of $\mathbb X$ by the area components of $\mathbb X$.
\end{remark}
Now let us summarize the discussion and give some pointers for the next sections. Firstly in  Section~\ref{section:para} we will introduce the space of paracontrolled distributions which is formally speaking the space of distributions (actually a couple of distributions) such that Ansatz~\ref{ansatz:paracontrol} holds. In a second step, in  Section~\ref{subsection2.4} (respectively~\ref{subsection2.5}), we will show that given a fixed rough distribution $\mathbb X$  and a fixed paracontrolled distribution $\Phi$ we can construct the product $\Phi^2\diamond X^{\diamond}$ (respectively $\Phi\diamond X^{\diamond2}$) like a continuous functional of the paracontrolled distribution $\Phi$ and the rough distribution $\mathbb X$. Moreover this construction will coincide with the "classical"(classical mean that the products appearing in this expression are understood in the usual sense of pointwise products of functions) definitions when all data are smooth. Finally in  Section~\ref{section:fixed point} we will show that for a small time the map $\Gamma$ is a contraction from the space of paracontrolled distributions to itself, which will allow us to construct immediately the map $\tilde F$ appearing in  Theorem~\ref{Th:Flow-const}. It is wise to remark that all these parts are purely analytic and use simply the fact that $\mathbb X$ is a rough distribution. In order to come back to the original problem, we will prove in Section~\ref{section:renormalization} that if $X^\eps$ is a regularization of the (O.U) then $R_{C_1^\eps,C_2^\eps}X^\eps=\mathbb X^\eps$ converges in the space $\mathcal X_{T,K}$.

\subsection{Paracontrolled distributions and fixed point equation}
\label{section:para}
 The aim of this section is to define a suitable space in which it is possible to formulate a fixed point for the eventual limit of the mollified solution. To be more precise, let $\mathbb X$ be a generic element of the space $\mathcal X$ (not necessarily equal to a fixed trajectory of the O.U.). We know that there exists $X^\eps\in\mathcal C_T^1(\mathbb T^3)$ and $a^\eps,b^\eps\in\mathbb R$ such that $\lim_{\eps\to0}R_{a^\eps,b^\eps}X^\eps=\mathbb X$. Let us focus on the regular equation given by:
\begin{equation*}
\begin{split}
\Phi^\eps&=I((X^\eps)^3-3a^\eps X^\eps)+3\left\{I(\Phi^\eps((X^\eps)^2-a^\eps))-3b^\eps I(X^\eps+\Phi^\eps)\right\}+3I((\Phi^\eps)^2 X^\eps)+I((\Phi^\eps)^3),
\end{split}
\end{equation*}
where we have omitted temporarily the dependence on the initial condition.As pointed previously if we assume simply that $\Phi^\eps$ converges to some $\Phi$ in $\mathcal C^{1/2-\delta}$, we see that the regularity of $\mathbb X$ is not sufficient to define $I(\Phi^2\diamond X):=\lim_{\eps\to0}I((\Phi^\eps)^2 X^{\eps})$ and 
$
I(\Phi\diamond X^{\diamond2}):=\lim_{\eps\to0}I(\Phi^\eps((X^\eps)^2-a^\eps))+3b^\eps I(X^\eps+\Phi^\eps).
$
As it has been remarked in the previous section the solution should satisfy the following decomposition:
$$
\Phi^\eps=I((X^\eps)^3-3a^\eps X^\eps)+3I(\Phi^\eps\prec((X^\eps)^2-a^\eps))+(\Phi^\eps)^\sharp.
$$
Then if we impose the convergence of $(\Phi^\eps)^\sharp$ to some $\Phi^\sharp$ in $\mathcal C_T^{3/2-\delta}$, we see that the limit $\Phi$ should satisfy the following relation 
$$
\Phi^{\sharp}:=\Phi-I(X^{\diamond3})-3I(\Phi\prec X^{\diamond2})\in \mathcal C_T^{3/2-\delta},
$$    
which as pointed in the previous section is the key point to define $I(\Phi^2 X)$, $I(\Phi\diamond X^{\diamond2})$ and to solve the equation
\begin{equation}\label{eq:abstract-renormalized-eq} 
\Phi=I(X^{\diamond3})+3I(\Phi^2\diamond X)+3I(\Phi\diamond X^{\diamond2})+I(\Phi^3).
\end{equation}
\begin{notation}
Let us introduce  some useful notations for the sequel
 $$
 B_{>}(f,g)=I(f\succ g),\quad B_{0}(f,g)=I(f\circ g) \quad\mathrm{and}\quad  B_{<}(f,g)=I(f\prec g).
 $$   
\end{notation}
\begin{remark}
The reader should keep in mind that the paraproduct $B_{<}(f,g)$ is always well-defined for every $f\in\mathcal C^\alpha$ and $g\in\mathcal C^\beta$ for all the value of $\alpha$ and $\beta$ moreover it has regularity $\min(\alpha,\beta)+2-\delta$, for all $\delta>0$.
\end{remark}
Now the following definition gives a precise meaning to the notion of paracontrolled distribution. 
\begin{definition}\label{def:controlled-distrib}
Let $\mathbb X\in\mathcal X$ and $z\in(1/2,2/3)$. We say that a couple  $(\Phi,\Phi')\in (C_T^{-z})^2$ is controlled by $\mathbb X$ if 
$$
\Phi^{\sharp} =\Phi - I(X^{\diamond3}) - 3B_{<}(\Phi' ,X^{\diamond2})
$$
is such that 
\begin{equation*}
\begin{split}
\|\Phi^{\sharp}\|_{\star,1,L,T}=&\sup_{t\in[0,T]}\left(t^{\frac{1+\delta+z}{2}}\|\Phi^\sharp_t\|_{1+\delta}+t^{1/4+\frac{\gamma+z}{2}}\|\Phi^\sharp_t\|_{1/2+\gamma}+t^{\frac{\kappa+z}{2}}\|\Phi^{\sharp}_t\|_{\kappa}\right)
\\&+\sup_{(s,t)\in[0,T]^2}s^{\frac{z+a}{2}}\frac{\|\Phi^\sharp_t-\Phi^{\sharp}_s\|_{a-2b}}{|t-s|^{b}}<+\infty
\end{split}
\end{equation*}
and 
$$
\|\Phi'\|_{\star,2,L,T}=\sup_{(s,t)\in[0,T]^2}s^{\frac{z+c}{2}}\frac{\|\Phi'_t-\Phi'_s\|_{c-2d}}{|t-s|^{d}}+\sup_{t\in[0,T]}t^{\frac{\eta+z}{2}}\|\Phi'_t\|_{\eta}<+\infty,
$$
where $L:=(\delta,\gamma,\kappa,a,b,c,d,\eta)\in[0,1]^{8}$, $z\in(1/2,2/3)$ and $2d\leq c$, $2b\leq a$. Let us denote by $\mathcal D^L_{T,\mathbb X}$ the space of such couples of distributions. For simplicity, in the sequel of the paper, we will sometimes use the abusive notation $\Phi$ instead of $(\Phi,\Phi')$. Moreover we equip this space with the following metric:
$$
d_{L,T}(\Phi_1,\Phi_2)=\|\Phi'_1-\Phi'_2\|_{\star,2,L,T}+\|\Phi^\sharp_1-\Phi^\sharp_2\|_{\star,1,L,T}
$$
for $\Phi_1,\Phi_2\in\mathcal D^{L}_{T,\mathbb X}$ and the quantity 
$$
\|\Phi\|_{\star,T,L}=\|\Phi_1\|_{\mathcal D_{T,X}^L}=d_{L,T}(\Phi_1,I(X^{\diamond3})).
$$
\end{definition}
\begin{remark}
The metric space $\mathcal D^L_{T,\mathbb X}$ is complete.
\end{remark}
In the following we will omit $L$ when its choice is clear.
We notice that the distance and the metric introduced in this last definition do not depend on $\mathbb X$. More generally for $\Phi\in\mathcal D_{T_1,X}^L$ and $\Psi\in\mathcal D_{T_2,Y}^{G}$ we denote by $d_{\min(L,G),\min(T_1,T_2)}(\Phi,\Psi)$ the same quantity. We claim that if $\Phi\in\mathcal D_X^{L}$ for a suitable choice of $L$ then we are able to define $I(\Phi\diamond X^{\diamond2})$ and $I(\Phi^2\diamond X)$ modulo the use of $\mathbb X$. 

Let us decompose the end of this Section into two parts, namely we show that $I(\Phi \diamond X^{\diamond2})$ and $I(\Phi^2  X)$ are well-defined when $\Phi$ is a controlled distribution. We also have to prove that when $\Phi$ is a controlled distribution, $\Psi + I(X^{\diamond 3}) + 3I(\Phi^2\diamond X) + 3I(\Phi\diamond X^{\diamond 2}) + I(\Phi^3)$ is also a controlled distribution. After all those verifications, the only remaining point will be to show that we can apply a fixed point argument to find a solution to the renormalized equation. This is the aim of Section~\ref{section:fixed point}.

\subsection{Decomposition of \texorpdfstring{$I(\Phi^2 \diamond X)$}{I(Phi2 X)}}
\label{subsection2.4}
Let $\mathbb X\in\mathcal X$ and $\Phi\in\mathcal D_{\mathbb X,T}^L$, assuming that all the component of $\mathbb X$ are smooth and using the fact that $\Phi$ is controlled  we get immediately the following expansion:
\begin{equation*}
\begin{split}
I(\Phi^2 X)=I(I(X^{\diamond 3} )^2 X)+I((\theta^{\sharp})^2X)+2I(\theta^{\sharp}I(X^{\diamond 3})X),
\end{split}
\end{equation*}
where 
$$
\theta^{\sharp}=B_{<}(\Phi',X^{\diamond 2})+\Phi^{\sharp}.
$$
 When $\mathbb X$ is no longer smooth and only satisfies  $\mathbb X\in\mathcal X$,  more particularly when $I(X^{\diamond3})\in C_T^{1/2-\delta}$ and $X\in C_{T}^{-1/2-\delta} $,  we can observe that the two terms $I((\theta^\sharp)^2X)$ and $I(\theta^\sharp I(X^{\diamond3})X)$ are well-defined due to the Bony estimates (Proposition~\ref{proposition:Bony-estim}) and the fact that $\Phi$ is a paracontrolled distribution. Let us focus on the term $I(X^{\diamond3})^2X$ which, at this stage, is not well understood. However the Bony paraproduct  decomposition gives: 
\begin{equation*}
\begin{split}
I(X^{\diamond3})^2X&=(I( X^{\diamond 3} )\prec I( X^{\diamond 3} ))\circ X+(I(X^{\diamond3})\circ I(X^{\diamond3}))\circ X
\\&+I(X^{\diamond3})^2\prec X+I(X^{\diamond3})^2\succ X.
\end{split}
\end{equation*}
We remark that only the first term of this expansion is not well-defined. To overcome this problem we use the commutator of  Proposition~\ref{proposition:comm-1}, and we have 
$$
R(I(X^{\diamond3}),I(X^{\diamond3}),X)=(I( X^{\diamond 3} )\prec I( X^{\diamond 3} ))\circ X-I( X^{\diamond 3} )(I( X^{\diamond 3} )\circ X),
$$ 
which is well-defined and  lies in the space $\mathcal C_T^{1/2-3\delta}$,  due to the fact that $\mathbb X\in\mathcal X$.
\begin{remark}
We remark that the "extension" of the term $I(\Phi^2 X)$ is a functional of  $(\Phi,\mathbb X)\in\mathcal D_{\mathbb X,T}^L\times\mathcal X$ and then we  sometimes use the notation $I(\Phi^2\diamond\mathbb X)[\Phi,\mathbb X]$ to underline this fact.
\end{remark}
\begin{proposition}\label{proposition:int-estim-1}\label{prop:choiceL1}
Let $z\in(1/2,2/3)$, $\Phi\in \mathcal D^L_\mathbb X$, and assume that $\mathbb X\in\mathcal X$. The quantity $I(\Phi^2\diamond X)[\Phi,\mathbb X]$ is well-defined via the following expansion 
$$
I(\Phi^2\diamond X)[\Phi,\mathbb X]:=I(I(X^{\diamond3})^2  X)+I((\theta^{\sharp})^2X)+2I(\theta^{\sharp}I(X^{\diamond 3})X),
$$
where 
$$
\theta^{\sharp}=B_{<}(\Phi',X^{\diamond 2})+\Phi^{\sharp}
$$
and 
\begin{equation}
\begin{split}
I( X^{\diamond 3})^2  X&:=I( X^{\diamond 3} )\circ I( X^{\diamond 3})X+2(I( X^{\diamond 3} )\prec I( X^{\diamond 3} ))\prec X
+2(I( X^{\diamond 3} )\prec I( X^{\diamond 3} ))\succ X
\\&+2I( X^{\diamond 3} )I( X^{\diamond 3} )\circ X+2R(I( X^{\diamond 3} ),I( X^{\diamond 3} ),X),
\end{split}
\end{equation}
where 
$$
R(I(X^{\diamond 3} ),I(X^{\diamond 3}),X)=(I( X^{\diamond 3} )\prec I( X^{\diamond 3} ))\circ X-I( X^{\diamond 3} )(I( X^{\diamond 3} )\circ X )
$$
is well-defined by Proposition~\ref{proposition:comm-1}. Moreover there exists a choice of $L$  such that the following bound holds:
$$
\|I(\Phi^2\diamond X)[\Phi,\mathbb X]\|_{\star,1,T}\lesssim T^{\theta}\left(\|\Phi\|_{\mathcal D_X^{L}}+1\right)^2\left(1+\|\mathbb X\|_{T,\nu,\rho,\delta,\delta'}\right)^3,
$$
for $\theta>0$ and $\delta,\delta',\rho,\nu>0$ small enough depending on $L$ and $z$. Moreover if $X\in\mathcal C_T^1(\mathbb T^3)$ then
$$
I(\Phi^2\diamond X)[\Phi,R_{a,b}X]=I(\Phi^2 X).
$$ 
\end{proposition}

\begin{proof}
By a simple computation, we have
\begin{eqnarray*}
\|B_{<}(\Phi', X^{\diamond 2} )(t)\|_{\kappa}&\lesssim&\int_{0}^{t}\dd s(t-s)^{-(\kappa+1+r)/2}\|\Phi_s'\|_{\kappa}\| X^{\diamond 2}_s \|_{-1-r}\\
&\lesssim_{r,\kappa}&T^{1/2-r/2-\kappa/2-z/2}\|\Phi'\|_{\star,2,T}\| X^{\diamond 2} \|_{-1-r},
\end{eqnarray*}
for $r,\kappa>0$  small enough and  $1/2<z<2/3$. A similar computation gives 
\begin{equation*}
\begin{split}
\|B_{<}(\Phi', X^{\diamond 2} )(t)\|_{1/2+\gamma}&\lesssim\int_0^t\dd s(t-s)^{-(3/2+\gamma+r)/2}\|\Phi'_s\|_{\kappa}\| X^{\diamond 2}_s \|_{-1-r}
\\&\lesssim_{\kappa,r,z}\|\Phi'\|_{\star,2,T}\| X^{\diamond 2} \|_{-1-r}\int_{0}^{t}\dd s(t-s)^{-(3/2+\gamma+r)/2}s^{-(\kappa+z)/2}
\\&\lesssim t^{1/4-(\gamma+\kappa+z+r)/2}\|\Phi'\|_{\star,2,L,T}\| X^{\diamond 2} \|_{-1-r},
\end{split}
\end{equation*}
for $\gamma,r,\kappa>0$ small enough. Using this bound we can deduce that 
\begin{equation*}
\begin{split}
\|I((\theta^{\sharp})^2X)(t)\|_{1+\delta}&\lesssim\int_0^t\dd s(t-s)^{-(3/2+\delta+\beta)/2}\|(\theta^{\sharp}_s)^2X_{s}\|_{-1/2-\beta}
\\&\lesssim_{\beta,\delta}\int_{0}^{t}\dd s(t-s)^{-(3/2+\delta+\beta)/2}\|\theta^{\sharp}_s\|_{\kappa}\|\theta^{\sharp}_s\|_{1/2+\gamma}\|X_s\|_{-1/2-\beta}
\\&\lesssim_{L,z}\|\Phi\|^2_{\star,L,T}(\| X^{\diamond 2} \|_{-1-r}+\|X\|_{-1/2-\beta}+1)^2
\\&\times\int_0^t\dd s(t-s)^{-(3/2+\delta+\beta)}s^{-(1/2+\kappa+\gamma+2z)/2}
\\&\lesssim_{L,z}t^{-(\delta+\kappa+\gamma+\beta+2z)}\|\Phi\|^2_{\star,L,T}(\| X^{\diamond 2} \|_{-1-r}+\|X\|_{-1/2-\beta}+1)^2,
\end{split}
\end{equation*}
for $\gamma,\beta,\delta>0$ small enough and $2/3>z>1/2$. Hence we have
$$
\sup_{t\in[0,T]}t^{(1+\delta+z)/2}\|I((\theta^{\sharp})^2X)(t)\|_{1+\delta}\lesssim_{L}T^{\theta_1}\|\Phi\|^2_{\star,L,T}(\| X^{\diamond 2} \|_{-1-r}+\|X\|_{-1/2-\beta}+1)^2,
$$
for some $\theta_1>0$ depending on $L$ and $z$. The same type of computation gives 
$$
\sup_{t\in[0,T]}t^{(\kappa+z)/2}\|\|I((\theta^{\sharp})^2X)(t)\|_{\kappa}\lesssim_{L,z}T^{\theta_2}\|\Phi\|^2_{\star,L,T}(\| X^{\diamond 2} \|_{-1-r}+\|X\|_{-1/2-\beta}+1)^2
$$
and  
$$
\sup_{t\in[0,T]}t^{(1/2+\gamma+z)/2}I((\theta^{\sharp})^2X)(t)\|_{1/2+\gamma}\lesssim_{L,z}T^{\theta_3}\|\Phi\|^2_{\star,T}(\| X^{\diamond 2} \|_{-1-r}+\|X\|_{-1/2-\beta}+1)^2,$$
where $\theta_2$ and $\theta_3$ are two non negative constants depending only on $L$ and $z$. To complete our study for this term, we have also
$$
\|((\theta^{\sharp})^2X)(t)-I((\theta^{\sharp})^2X)(s)\|_{a-2b}\lesssim I^1_{st}+I^2_{st},
$$
where 
$$
I^{1}_{st}=\left|\left|\int_{0}^{s}\dd u(P_{t-u}-P_{s-u})(\theta^{\sharp}_u)^2X_u\right|\right|_{a-2b}\quad \mathrm{and} \quad I^2_{st}=\left|\left|\int_s^t\dd uP_{t-u}(\theta^{\sharp}_u)^2X_u\right|\right|_{a-2b}.
$$
Let us begin by bounding $I^1$:
\begin{equation*}
\begin{split}
I_{st}^{1}&\lesssim(t-s)^b\int_0^s\dd u\|P_{s-u}(\theta^{\sharp}_u)^2X_u\|_{a}
\\&\lesssim(t-s)^b\int_0^t\dd u(s-u)^{-(1/2+a+\beta)}\|(\theta^{\sharp}_u)^2X_u\|_{-1/2-\beta}
\\&\lesssim T^{\theta_4}|t-s|^b\|\Phi\|^2_{\star,T}(\| X^{\diamond 2} \|_{-1-r}+\|X\|_{-1/2-\beta}+1)^2,
\end{split}
\end{equation*}
 where $\theta_4>0$ depends on $L$ and $z$. Let us focus on the bound for $I^2$,
 \begin{equation*}
 \begin{split}
 I^2_{st}&\lesssim\int_{s}^{t}(t-u)^{-(1/2+a-2b+\beta)/2}\|(\theta^{\sharp}_u)^2X_u\|_{-1/2-\beta}
\\&\lesssim_{L,z}\|\Phi\|^2_{\star,T}(\| X^{\diamond 2} \|_{-1-r}+\|X\|_{-1/2-\beta}+1)^2\int_{s}^{t}\dd u(t-u)^{-(1/2+a-2b+\beta)/2}u^{-(1/2+\kappa+\gamma+2z)/2}
 \end{split} 
 \end{equation*}
and 
\begin{equation*}
\begin{split}
&\int_{s}^{t}\dd u(t-u)^{-(1/2+a-2b+\beta)/2}u^{-(1/2+\kappa+\gamma+2z)/2}
\\&=(t-s)^{3/4-(a-2b+\beta)}\int_{0}^{1}\dd x(1-x)^{-(1/2+a-2b+\beta)}(s+x(t-s))^{-(1/2+\kappa+\gamma+2z)/2}
\\&\lesssim_{l,\kappa,\gamma,a,b} (t-s)^{l-(a-2b+\beta)/2}s^{1/2-z+(\kappa+\gamma)/2}\int_0^1\dd x(1-x)^{-(1/2+a-2b+\beta)/2}x^{-3/4+l}.
\end{split}
\end{equation*}
Since $z<1$, we can choose $l,\kappa,\gamma,b>0$ small enough and we have
$$
\int_{s}^{t}\dd u(t-u)^{-(1/2+a-2b+\beta)/2}u^{-(1/2+\kappa+\gamma+2z)/2}\lesssim_{L}T^{\theta_5}(t-s)^bs^{-(z+a)/2},
$$
where $\theta_5>0$. This gives the needed bound for $I_2$. Finally, we have
$$
\sup_{(s,t)\in[0,T]}s^{(z+a)/2}\frac{\|I((\theta^{\sharp})^2X)(t)-I((\theta^{\sharp})^2X)(s)\|_{a-2b}}{|t-s|^b}\lesssim T^{\theta_5}\|\Phi\|^2_{\star,T}(\| X^{\diamond 2} \|_{-1-r}+\|X\|_{-1/2-\beta}+1)^2.
$$ 
Hence
$$
\|I((\theta^{\sharp})^2X)\|_{\star,1,T}\lesssim_{L} T^{\theta}\|\Phi\|^2_{\star,T}(\| X^{\diamond 2} \|_{-1-r}+\|X\|_{-1/2-\beta}+1)^2.
$$ 
The bound  for $\|I(\theta^{\sharp}I(X^3)X)\|_{\star,1,T}$ can be obtained in a similar way and then, according to the hypothesis given on the area term  $I(X^{\diamond 3}) X$ and  the decomposition of $I(I(X^{\diamond 3})^2 X)$, we obtain easily from  Proposition~\ref{proposition:comm-1} and  Proposition~\ref{proposition:Bony-estim} that 
$$
\| I(I(X^{\diamond 3}))^2\diamond X \|_{\star,1,T}\lesssim T^{\theta}(1+\|I(X^{\diamond 3})\circ X)\|_{\delta',-1/2-\rho}+\| I(X^{\diamond 3})\|_{\delta',1/2-\rho}+\|X\|_{\delta',-1/2-\rho})^3,
$$  
for $3\rho<\delta'$ small enough, which gives the wanted result.
\end{proof}

\subsection{Decomposition of \texorpdfstring{$I(\Phi \diamond X^{\diamond 2} )$}{I(Phi * X2)}}
\label{subsection2.5}
First, let $X\in C([0,T],C^{\infty})$, $(a,b)\in\mathbb R^2$, $\mathbb X=R_{a,b}X$ and $\Phi\in\mathcal D_{R_{a,b}X,T}$.  Using the paracontrolled structure of $\Phi$ and the Bony paraproduct decomposition for the term $\Phi(X^2-a)$ we have the following expansion: 
\begin{multline*}
I(\Phi (X^{2}-a ))-3bI(\Phi'+X)  =
 B_{<}(\Phi, X^{2}-a )
+\big(B_{0}(I(X^{3})-3aX,X^{2}-a)-3bI(X)\big)
\\+3\big(B_{0}(B_{<}(\Phi', X^{2}-a ), X^{2}-a )-3bI(\Phi')\big)
+B_{0}(\Phi^{\sharp}, X^{2}-a )
+B_{>}(\Phi, X^{2}-a).
\end{multline*}
Thanks to the Bony estimates (Proposition~\ref{proposition:Bony-estim}), we can see that all the terms appearing in the right hand side, apart from the third one, are well-defined even when $\mathbb X$ is no longer equal to $R_{a,b}X$ but a general rough distribution. The only problem  is to give an expansion of the third term of this equation.
We have to deal with the (ill-defined) diagonal term.
\begin{multline*}
K(\Phi',R_{a,b}X)(t)=B_{0}(B_{<}(\Phi', X^{2}-a ),X^{2}-a)(t)-bI(\Phi')\\ =\int_0^t\dd sP_{t-s}\int_0^s\dd\sigma P_{s-\sigma}(\Phi'_{\sigma}\prec(X^{2}_{\sigma}-a))\circ (X^{2}_s-a)-3bI(\Phi'). 
\end{multline*}
It can be decomposed in the following way 
\begin{equation}
\label{eq:exp-1}
\begin{split}
K(\Phi',R_{a,b}X)(t)
=&
\int_0^t\dd sP_{t-s}\Phi'_s\left(I( X^{2}-a)(s)\circ(X^{2}_s-a )-b\right)\\
&+\int_0^t\dd s P_{t-s}\int_0^s\dd\sigma(\Phi'_\sigma-\Phi'_s) (X^{2}_s-a) \circ P_{s-\sigma}(X_{\sigma}^{2}-a)
\\&+\int_0^t\dd sP_{t-s}\int_0^s\dd\sigma R^1_{s-\sigma}(\Phi'_{\sigma},X^{2}_\sigma-a)\circ (X^{2}_s-a )\\
&+\int_0^t\dd sP_{t-s}\int_0^sR^2(\Phi'_\sigma,P_{s-\sigma}X^{2}_{\sigma}-a, X^{2}_s-a )
\\&\equiv\sum_{i=1}^{4}K_i(\Phi',R_{a,b}X)(t),
\end{split}
\end{equation}
where 
$$
R^1_{s-\sigma}(f,g)=P_{s-\sigma}(f\prec g)-f\prec P_{s-\sigma}g,\quad R^2(f,g,h)=(f\prec g)\circ h-f(g\circ h)
$$
and $f,g,h$ are distributions lying in the suitable Besov spaces in order for $R^1$ and $R^2$ to be defined. Now the point is that the right hand side of this equation allows us to define the operator $K$ (and even each $K_i$) for a general rough distribution $\mathbb X$, as it  will be proved in  Proposition~\ref{prop:choiceL2} below, and for a general $\mathbb X$. Before stating the proposition let us give a useful improvement of the Schauder estimate which will help us to estimate the operator $K$.
\begin{lemma}\label{lemma:Holder-estim}
Let $f$ be a a space time distribution such that  
$
\sup_{t\in[0,T]}t^{(r+z)/2}\|f_t\|_{r}<+\infty
$.
Then the following bound holds
$$
\sup_{s,t\in[0,T]}\frac{\|I(f)(t)-I(f)(s)\|_{a-2b}}{|t-s|^b}\lesssim_{b,a,z,r}T^{\theta}\sup_{t\in[0,T]}t^{(r+z)/2}\|f_t\|_{r},
$$
where $a+z<2$, $z+r<2$, $a-r<2$, $0<a,b<1$ and $\theta>0$ is a constant depending only on $a,r,b,z$.
\end{lemma} 
\begin{proof}
By a simple computation we have 
$$
I(f)(t)-I(f)(s)=I^1_{st}+I^2_{st},
$$
where
$$
I^1_{st}=(P_{t-s}-1)\int_0^s\dd uP_{s-u}f_u\quad \mathrm{and}\quad I^2_{st}=\int_s^t\dd uP_{t-u}f_u.
$$
Using Lemma~\ref{lemma:Heat-flow-smoothing} the following bound holds
$$
\|I^1_{st}\|_{a-2b}\lesssim|t-s|^b\int_0^t\dd u(t-u)^{-(a-r)/2}u^{-(r+z)/2}\sup_{t\in[0,T]}t^{(r+z)/2}\|f_t\|_{r}<+\infty.
$$
To handle the second term we use the  H\"older inequality,
$$
\|I^2_{st}\|_{a-2b}\lesssim|t-s|^b\left(\int_s^t\dd u(t-u)^{\frac{-(a-2b-r)}{2(1-b)}}u^{-\frac{(z+r)}{2(1-b)}}\right)^{1-b}\sup_{t\in[0,T}t^{(r+z)/2}\|f_t\|_{r}<+\infty,
$$
which ends the proof.
\end{proof}

The following proposition gives us the regularity for our terms.
\begin{proposition}\label{proposition:diamond}\label{prop:choiceL2}
Assume that $ X$ is smooth and that $\mathbb X=R_{a,b}X$ then there is a choice of $L$ such that for all $z\in(1/2,2/3)$ the following bound holds
$$
\|K(\Phi',\mathbb X)(t))\|_{\star,1,T}\lesssim T^{\theta}(1+\|\mathbb X\|_{T,K})^2\|\Phi'\|_{\star,2,T},
$$
where $K\in[0,1]^4$ and $\theta>0$ are two small parameters depending only on $L$ and $z$. Thus this bound allows us to extend the operator $K$ the whole space of rough distributions $\mathcal X$, with the same bound. 
\end{proposition}
\begin{proof}
First, let us estimate the first term of the expansion~\eqref{eq:exp-1}.
\begin{equation*}
\begin{split}
\|K_1(\Phi',\mathbb X)(t)\|_{1+\delta}&\lesssim\int_0^t\dd s(t-s)^{-(1+\delta+\eta/2)/2}\|\Phi'_s(I( X^{2}-a )(s)\circ ( X^{\diamond}_s -a)-b\|_{-\eta/2}
\\&\lesssim\|\Phi'\|_{\star,2,T}\|I( X^{2} -a),X^{2}-a)-b\|_{\mathscr C_T^{-\eta/2,\nu}}\times\left(\int_0^t\dd s(t-s)^{-(1+\delta+\eta/2)/2}s^{-(\eta+\nu+z)/2}\right)
\\&\lesssim_{\beta,L}T^{\theta_1}\|\Phi'\|_{\star,2,T}\|I( X^{2} -a),X^{2}-a)-b\|_{\mathscr C_T^{-\eta/2,\nu}},
\end{split}
\end{equation*}
 for $\eta,\delta>0$ small enough and where $\theta_1>0$ depends on $L$. Hence 
$$
\sup_{t\in[0,T]}t^{(1+\delta+z)/2}\||K_1(\Phi',\mathbb X)(t)\|_{1+\delta}\lesssim_{L,z}T^{\theta_1}\|\Phi'\|_{\star,2,T}\|I( X^{2} -a),X^{2}-a)\|_{\mathscr C_T^{-\eta/2,\nu}}.
$$
Let us focus on the second term. We have 
\begin{equation*}
\begin{split}
\|K_2(\Phi',\mathbb X)(t)\|_{1+\delta}
&\lesssim
\int_0^t\dd s (t-s)^{-(1+\delta-\beta)/2}\int_0^s\dd\sigma \|(\Phi'_\sigma-\Phi'_s)(P_{s-\sigma}(X^{\diamond 2}_\sigma-a)\circ (X^{2}_s-a) )\|_{\beta}
\\&\lesssim_{\beta,\rho}\int_0^t\dd s (t-s)^{-(1+\delta-\beta)/2}\int_0^s\dd\sigma(s-\sigma)^{-(2+\rho)/2} \|\Phi'_\sigma-\Phi'_s\|_{c-2d}\|X^{2}_s-a\|^2_{-1-\rho}
\\&\lesssim_{L,\beta,\rho}\|\Phi'\|_{\star,2,T}\| X^{2}-a \|^2_{\mathcal C_T^{-1-\rho}}\int_0^t\dd s (t-s)^{-(1+\delta-\beta)/2}\int_0^s\dd\sigma(s-\sigma)^{-1-\rho/2+d}\sigma^{-(c+z)/2}
\\&\lesssim_{L,\beta,\rho}\|\Phi'\|_{\star,2,T}\| X^{2}-a \|^2_{\mathcal C_T^{-1-\rho}}\int_0^t\dd s (t-s)^{-(1+\delta-\beta)/2}s^{-(\rho+c-2d+z)/2}
\\&\lesssim_{L,\beta,\rho}T^{\theta_2}\|\Phi'\|_{\star,2,T}\| X^{2}-a \|^2_{\mathcal C_T^{-1-\rho}},
\end{split}
\end{equation*}
where $\beta=\min(c-2d,\rho)\geq0$,  $c,d,\rho>0$ are small enough, $z<1$ and $\theta_2>0$ is a constant which  depends only on $L$ and $z$. Using  Lemma~\ref{lemma:Heat-flow-smoothing},  we see that
$$
\|R^1_{s-\sigma}(\Phi'_{\sigma},X^{2}_\sigma-a)\|_{1+2\beta}\lesssim(s-\sigma)^{-(2+3\beta-\eta)/2}\|\Phi'_\sigma\|_{\eta}\|X^{2}_\sigma-a\|_{-1-\beta},
$$
for all $\beta>0$, $\beta<\eta/3$ small enough. By a straightforward computation we have 
\begin{equation*}
\begin{split}
\|K_3(\Phi',\mathbb X)(t))\|_{1+\delta}&\lesssim\int_0^t\dd s(t-s)^{-(1+\delta-\beta)/2}\int_0^s\dd\sigma\|(R^1_{s-\sigma}(\Phi'_\sigma,X_{\sigma}^{ 2}-a)\circ (X_{s}^{2}-a)\|_{\beta}
\\&\lesssim \int_0^t\dd s(t-s)^{-(1+\delta-\beta)/2}\int_0^s\dd \sigma\|R^1_{s-\sigma}(\Phi'_\sigma,X^{2}_\sigma)-a\|_{1+2\beta}\| X^{2}_s-a\|_{-1-\beta}
\\&\lesssim\| X^{2}-a\|_{\mathcal C_T^{-1-\beta}}^2\|\Phi'\|_{\star,2,T}\int_0^t\dd s(t-s)^{-(1+\delta-\beta)/2}\int_0^s\dd\sigma(s-\sigma)^{-(2+3\beta-\kappa)/2}\sigma^{-(\eta+z)/2}
\\&\lesssim \| X^{2}-a\|_{\mathcal C_T^{-1-\beta}}^2\|\Phi'\|_{\star,2,T}\int_0^t\dd s (t-s)^{-(1+\delta-\beta)/2}s^{-(3\beta-\kappa+\eta+z)/2}
\\&\lesssim T^{\theta_3}\| X^{2}-a\|_{\mathcal C_T^{-1-\beta}}^2\|\Phi'\|_{\star,2,T},
\end{split}
\end{equation*}
where $\theta_3>0$ is a constant depending on $L$ and $z$, $0<\beta<\eta/3$ is small enough and $z<1$. 
To treat the last term it is sufficient to use the commutation result given in  Proposition~\ref{proposition:comm-1}. Indeed we have
$$
\|R^2(\Phi'_\sigma,P_{s-\sigma}(X^{2}_{\sigma}-a), X^{2}_s-a )\|_{\eta-3\beta}\lesssim_{\eta,\beta}s^{-(\eta+z)/2}(s-\sigma)^{-(2-\beta)/2}\| X^{2}-a\|^2_{\mathcal C_T^{-1-\beta}}\|\Phi'\|_{\star,2,T},
$$
for $0<\beta<\eta/3$ small enough. Hence
\begin{equation*}
\begin{split}
\||K_4(\Phi',\mathbb X)(t))\|_{1+\delta}&\lesssim_{\eta,\beta}\| X^{2}-a\|^2_{\mathcal C_T^{-1-\beta}}\|\Phi'\|_{\star,2,T}\int_0^t\dd s(t-s)^{-(1+\delta-\eta+3\beta)/2}\int_0^s\dd \sigma s^{-(\eta+z)/2}(s-\sigma)^{-(2-\beta)/2}
\\&\lesssim_{\eta,\beta}\| X^{2}-a\|^2_{\mathcal C_T^{-1-\beta}}\|\Phi'\|_{\star,2,T}\int_0^t\dd s(t-s)^{-(1+\delta-\eta+3\beta)/2}s^{-(\eta+z+\beta)/2}
\\&\lesssim T^{\theta_4}\| X^{2}-a\|^2_{\mathcal C_T^{-1-\beta}}\|\Phi'\|_{\star,2,T},
\end{split}
\end{equation*} 
for $\theta_4>0$ depending on $L$ and $z<1$ and $\beta,\eta,\delta>0$ small enough. Binding all these bounds together, we can conclude that  
\begin{multline*}
\sup_{t\in[0,T]}t^{(1+\delta+z)/2}\|K(\Phi',\mathbb X)(t)\|_{1+\delta}\\
\lesssim_{L,z}T^{\theta}(1+\| X^{2}-a \|_{\mathcal C_T^{-1-\rho}}+\|I( X^{2} -a),X^{2}-a)-b\|_{\mathscr C_T^{-\eta/2,\nu}})^2\|\Phi'\|_{\star,2,T},
\end{multline*}
for $\theta>0$ depending on $L$ and $z$. The same arguments gives
\begin{multline*}
\sup_{t\in[0,T]}t^{(1/2+\gamma+z)/2}\|K(\Phi',\mathbb X)(t)\|_{1/2+\gamma}\\
\lesssim_{L,z}T^{\theta}(1+\| X^{2}-a \|_{\mathcal C_T^{-1-\rho}}+\|I( X^{2} -a),X^{2}-a)-b\|_{\mathscr C_T^{-\eta/2,\nu}})^2\|\Phi'\|_{\star,2,T}
\end{multline*}
and 
\begin{equation*}
\begin{split}
\sup_{t\in[0,T]}t^{(\kappa+z)/2}\|K(\Phi',\mathbb X)(t)\|_{\kappa}\lesssim_{L,z}&T^{\theta}(1+\| X^{\diamond 2} \|_{\mathcal C_T^{-1-\rho}}+\|I( X^{2} -a),X^{2}-a)-b\|_{\mathscr C_T^{-\eta/2,\nu}})^2\|\Phi'\|_{\star,2,T}.
\end{split}
\end{equation*}
To obtain the needed bound we still need to estimate the following quantity 
$$
\sup_{(s,t)\in[0,T]^2}s^{\frac{z+a}{2}}\frac{\|K(\Phi',\mathbb X)(t)-K(\Phi',\mathbb X)(s)\|_{a-2b}}{|t-s|^{b}}.
$$
To handle this   we use the fact that 
$
K_i(\Phi',\mathbb X)(t)=I(f^i)
$
with 
$$
f^1(s)=\Phi'_sI( X^{2}-a )(s)\circ ( X^{\diamond}_s -a)-b),\quad f^2(s)=\int_0^s\dd\sigma(\Phi'_\sigma-\Phi'_s)(( X^{2}_s-a)\circ P_{s-\sigma}X_{\sigma}^{2}-a)
$$
and 
$$
f^3(s)=\int_0^s\dd\sigma(R^1_{s-\sigma}(\Phi'_{\sigma},X^{2}_\sigma-a)\circ (X^{2}_s-a ),\quad f^4(s)=\int_0^sR^2(\Phi'_\sigma,P_{s-\sigma}(X^{2}_{\sigma}-a), X^{2}_s-a ).
$$
By an easy computation we have
$$
\|f^1(t)\|_{\eta/2}\lesssim_{\eta}s^{-(\eta+z)/2}\|\Phi'\|_{\star,2,T}(1+\|I( X^{2} -a),X^{2}-a)-b\|_{\mathscr C_T^{-\eta/4,\nu}})^2,
$$
\begin{multline*}
\|f^2(s)\|_{-d}\lesssim\|\Phi'\|_{\star,2,T}\| X^{2}-a\|^2_{-1-d/4}\\
\times\int_0^s\dd \sigma(s-\sigma)^{-1+d/2}\sigma^{-(c+z)/2}\lesssim_{z,c,d}s^{d/2-(c+z)/2}\|\Phi'\|_{\star,2,T}\| X^{2}-a\|_{-1-d/4},
\end{multline*}
and
\begin{multline*}
\|f^3(s)\|_{2\eta/3}\lesssim\|\Phi'\|_{\star,2,T}\| X^{2}-a\|^2_{-1-\eta/9}\\
\times\int_0^s\dd s(s-\sigma)^{-1+\eta/9}s^{-(\eta+z)/2}\lesssim s^{-(11\eta+9z)/2}\|\Phi'\|_{\star,2,T}\| X^{2}-a\|^2_{-1-\eta/9},
\end{multline*}
where $\nu>0$ depends only on $L$. A similar bound holds for $f^4$, which allows us to conclude by Lemma~\ref{lemma:Holder-estim} that we have
$$
\sup_{(s,t)\in[0,T]^2}s^{\frac{z+a}{2}}\frac{\|K(\Phi',\mathbb X)(t)-K(\Phi',\mathbb X)(s)\|_{a-2b}}{|t-s|^{b}}\lesssim T^\theta\|\Phi'\|_{\star,2,T}\| X^{\diamond 2} \|^2_{-1-\rho},
$$
for some $\rho>0$, $\theta>0$ and $\eta,c,d>0$ small enough and $z\in(1/2,2/3)$.
\end{proof}

We are now able to give the meaning of $I(\Phi \diamond X^{\diamond 2} )$ for  $\Phi\in\mathcal D_\mathbb X^L$.
\begin{corollary}\label{corollary:integ-def-2}
Assume that $\mathbb X\in\mathcal X$ and let $\Phi\in\mathcal D_\mathbb X^L$ then for $z\in(1/2,2/3)$ and for a suitable choice of $L$  the term $I(\Phi \diamond X^{\diamond 2} )[\Phi,\mathbb X]$ is defined via the following expansion:
$$
I(\Phi \diamond X^{\diamond 2} )[\Phi,\mathbb X]:=B_{<}(\Phi, X^{\diamond 2} )+B_{>}(\Phi, X^{\diamond 2} )+I\big(\big(I(X^{\diamond 3})\circ X^{\diamond 2}\big)^{\diamond}\big)+3K(\Phi',\mathbb X)+B_{0}(\Phi^\sharp ,X^{\diamond 2} ).
$$
We have the following bound
$$
\|B_{0}(\Phi^{\sharp}, X^{\diamond 2} )\|_{\star,1,T}+\|B_{>}(\Phi, X^{\diamond 2} )\|_{\star,1,T}\lesssim T^{\theta}\|\Phi\|_{\star,T}\| X^{\diamond 2} \|_{\mathcal C_T^{-1-\rho}},
$$
for some $\theta,\rho>0$ being two non-negative constants depending on $L$ and $z$. Moreover if $a,b\in\mathbb R$, $X\in \mathcal C_T^{1}(\mathbb T^3)$ and $\varphi\in C^{\infty}([0,T])$, we have 
$$
I(\Phi\diamond X^{\diamond 2})[\Phi,R_{a,b}^{\varphi}\mathbf X]=I(\Phi(X^2-a))+3bI(X+\Phi'),
$$
for every $\Phi\in\mathcal D_{R^\varphi_{a,b}\mathbf{X}}$.
\end{corollary}
\begin{proof}
We remark that all the terms in the definition of $I(\Phi \diamond X^{\diamond 2} )$ are well-defined due to  Proposition~\ref{proposition:diamond}  and the definition of the paraproduct. We also notice that
\begin{equation*}
\begin{split}
\|B_{0}(\Phi^\sharp, X^{\diamond 2} )(t)\|_{1+\delta}&\lesssim\int_0^t\dd s(t-s)^{-(1+\delta/2)/2}\|\Phi^\sharp_s\|_{1+\delta}\| X^{\diamond 2} \|_{-1-\delta/2}
\\&\lesssim\|\Phi^\sharp\|_{\star,1,T}\| X^{\diamond 2} \|_{\mathcal C_T^{-1-\delta/2}}\int_0^t\dd s(t-s)^{-(1+\delta/2)/2}s^{-(1+\delta+z)/2}
\\&\lesssim s^{-(3/2\delta+z)/2}\|\Phi^\sharp\|_{\star,1,T}\| X^{\diamond 2} \|_{\mathcal C_T^{-1-\delta/2}},
\end{split}
\end{equation*}
 which gives easily 
$$
\sup_{t\in[0,T]}t^{(1+\delta+z)/2}\|B_{0}(\Phi^\sharp, X^{\diamond 2} )(t)\|_{1+\delta}\lesssim T^{1/2-\delta}\|\Phi^\sharp\|_{\star,1,T}\| X^{\diamond 2} \|_{\mathcal C_T^{-1-\delta/2}},
$$
for $\delta<1/2$. By a similar computation we obtain that there exists $\theta>0$ depending on $L$ and $z$ such that 
\begin{multline*}
\sup_{t\in[0,T]}t^{(1/2+\gamma+z)/2}\|B_{0}(\Phi^\sharp, X^{\diamond 2} )(t)\|_{1/2+\gamma}+\sup_{t\in[0,T]}t^{(\kappa+z)/2}\|B_{0}(\Phi^\sharp, X^{\diamond 2} )(t)\|_{\kappa}\\ \lesssim T^{\theta}\|\Phi^\sharp\|_{\star,1,T}\| X^{\diamond 2} \|_{\mathcal C_T^{-1-\delta/2}}.
\end{multline*}
 To obtain the needed bound for this term we still need  to estimate the H\"older type norm. We remark that 
$$
\|\Phi^\sharp_s\circ X^{\diamond 2}_s\|_{\delta/2}\lesssim s^{-(1+\delta+z)/2}\|\Phi^\sharp_s\|_{1+\delta} \| X^{\diamond 2}_s \|_{-1-\delta/2}
$$
and then as usual we decompose the norm in the following way
$$
B_{0}(\Phi^\sharp_t, X^{\diamond 2} )(t)-B_{0}(\Phi^\sharp_s, X^{\diamond 3}_s )=I^1_{st}+I^2_{st},
$$
where 
$$
I^1_{st}=(P_{t-s}-1)\int_0^t\dd uP_{t-u}(\Phi^\sharp_u\circ X^{\diamond 2}_u)\quad \mathrm{and}\quad I_{st}^2=\int_s^t\dd uP_{t-u}(\Phi^\sharp_u\circ X^{\diamond 2}_u).
$$
A straightforward computation gives 
\begin{equation*}
\begin{split}
\|I^1_{st}\|_{a-2b}&\lesssim\|\Phi^\sharp\|_{\star,1,T}\| X^{\diamond 2} \|_{\mathcal C_T^{1-\delta/2}}|t-s|^b\int_0^t\dd u(t-u)^{-(a-\delta/2)/2}u^{-(1+\delta+z)/2}
\\&\lesssim T^{(1-a-\delta/2-z)/2}|t-s|^b\|\Phi^\sharp\|_{\star,1,T}\| X^{\diamond 2} \|_{\mathcal C_T^{1-\delta/2}}.
\end{split}
\end{equation*}
For $I^2$ we use the H\"older inequality, which gives 
\begin{equation*}
\begin{split}
\|I^2_{st}\|_{a-2b}&\lesssim|t-s|^b\|\Phi^\sharp\|_{\star,1,T}\| X^{\diamond 2} \|_{\mathcal C_T^{1-\delta/2}}
\left(\int_s^t\dd u(t-u)^{-\frac{a-2b-\delta/2}{2(1-b)}}u^{-\frac{1+\delta+z}{2(1-b)}})\right)^{1-b}
\\&\lesssim T^{(1-a-\delta/2-z)/2}|t-s|^b\|\Phi^\sharp\|_{\star,1,T}\| X^{\diamond 2} \|_{\mathcal C_T^{1-\delta/2}},
\end{split}
\end{equation*}
for $a,\delta>0$ small  enough and $z<1$. We have obtained that 
$$
\|B_{0}(\Phi^\sharp, X^{\diamond 2} )\|_{\star,1,T}\lesssim T^{\theta}\|\Phi^\sharp\|_{\star,1,T}\| X^{\diamond 2} \|_{\mathcal C_T^{-1-\delta/2}},
$$
for some $\theta>0$ depending on $L$ and $z$. The bound for the term $B_{>}(\Phi, X^{\diamond 2} )$ is obtained by a similar argument and this ends the proof. 
\end{proof}
\begin{remark}
When there are no ambiguity we use the notation $I(\Phi\diamond X^{\diamond2})$ instead of $I(\Phi\diamond X^{\diamond2})[\Phi,\mathbb X]$.
\end{remark}
\section{Fixed point procedure}\label{section:fixed point}
Using the analysis of $I(\Phi\diamond X^{\diamond 2} )$ and $I(\Phi^2 \diamond X)$ developed in the previous section, we can now show that the equation       
$$
\Phi=I( X^{\diamond 3} )+3I(\Phi \diamond X^{\diamond 2} )+3I(\Phi^2\diamond X)+I(\Phi^3)+\Psi
$$
admits a unique solution $\Phi\in\mathcal D_\mathbb X^L$ for a suitable choice of $L$ and $z\in(1/2,2/3)$, via the fixed point method.  We also show that if $u^\eps$ is the solution of the regularized equation and $\Phi^\eps$ is such that $u^\eps=X^\eps+\Phi^\eps$, then $d(\Phi^\eps,\Phi)$ goes to $0$ as $\eps$. Hence, by the convergence of $X^\eps$ to $X$ we have the convergence of $u^\eps$ to $u=\Phi + X$. Let us begin by giving our fixed point result. 
\begin{theorem}\label{theorem:FP}
Assume that $\mathbb X\in\mathcal X$,  $u^0\in\mathcal C^{-z}(\mathbb T^3)$ with $z\in(1/2,2/3)$ and $L$ is such that the bounds of Propositions~\ref{prop:choiceL1} and~\ref{prop:choiceL2} are satisfied. Let $(\Phi,\Phi')\in\mathcal D_\mathbb X^L$ and $\Psi=Pu^0$. Then we define the application $\Gamma:\mathcal D_{\mathbb X,T}^L\to\mathcal C_T^{-z}(\mathbb T^3)$ by
$$
\Gamma(\Phi,\Phi')=I( X^{\diamond 3} )+3I(\Phi \diamond X^{\diamond 2} )+3I(\Phi^2\diamond X)+I(\Phi^3)+\Psi,
$$
where $I(\Phi \diamond X^{\diamond 2} )$ and $I(\Phi^2X)$ are given by Corollary~\ref{corollary:integ-def-2}  and  Proposition~\ref{proposition:int-estim-1}. Then $(\Gamma(\Phi),\Phi)\in\mathcal D_\mathbb X^L$ for a suitable choice of $L$ and it satisfies the following bound:
\begin{equation}\label{eq:bound}
\begin{split}
\|\Gamma(\Phi)\|_{\star,T}\lesssim &(T^\theta\|\Phi\|_{\star,L,T}+1)^3(1+\|\mathbb X\|_{T,K}+\|u^0\|_{-z})^3.
\end{split}
\end{equation}
Moreover for $\Phi_1,\Phi_2\in\mathcal D_\mathbb X^L$ the following bound holds:
\begin{equation}\label{eq:contarc}
\begin{split}
d_{T,L}\left(\Gamma(\Phi_1),\Gamma(\Phi_2)\right)\lesssim& T^{\theta}d_{T,L}\left(\Phi_1,\Phi_2\right)(\|\Phi_1\|_{\star,L,T}+\|\Phi_2\|_{\star,L,T}+1)^2(1+\|\mathbb X\|_{T,K}+\|u^0\|_{-z})^3,
\end{split}
\end{equation}
for some $\theta>0$ and $K\in[0,1]^8$ depending on $L$ and $z$. We can conclude that for this choice of $L$ there exists $T>0$ and a unique $\Phi\in\mathcal D_{\mathbb X,T}^L$ such that
\begin{equation}\label{eq:local sol}
\Phi=\Gamma(\Phi)=I( X^{\diamond 3} )+3I(\Phi^2 \diamond X^{\diamond 2} )+3I(\Phi^2\diamond X)+I(\Phi^3)+\Psi.
\end{equation}
\end{theorem}
\begin{proof}
By   Corollary~\ref{corollary:integ-def-2} and Proposition~\ref{proposition:int-estim-1} we see that $\Gamma(\Phi)$ has the needed algebraic structure  of the controlled distribution. More precisely,
$$
\Gamma(\Phi)'=\Phi,\quad\Gamma(\Phi)^\sharp=3B_{>}(\Phi, X^{\diamond 2} )+X^{\diamond}(\Phi')+3B_{0}(\Phi^\sharp ,X^{\diamond 2} )+3I(\Phi^2X)+I(\Phi^3)+\Psi
$$
and $\Gamma(\Phi)\in\mathcal C_T^{-z}$. To show that $\Gamma(\Phi)\in\mathcal D_\mathbb X^L$ and to obtain the first bound, it remains to estimate $\|\Phi\|_{\star,2,L,T}$ and $\|\Gamma(\Phi)^\sharp\|_{\star,1,L,T}$. A straightforward computation gives
\begin{equation*}
\begin{split}
\|\Phi_t\|_{\eta}&\lesssim\|I(X^{\diamond3})(t)\|_{\eta}+\|B_{<}(\Phi', X^{\diamond 2} )(t)\|_{\eta}+\|\Phi_t^\sharp\|_{\eta}\
\\&\lesssim\|I(X^{\diamond 3})\|_{\eta}+\|\Phi'\|_{\star,2,T}\| X^{\diamond 2} \|_{-1-\eta}\int_0^t\dd s(t-s)^{-(1+2\eta/2)/2}s^{-(\eta+z)/2}+t^{-(\kappa+z)}\|\Phi^\sharp\|_{\star,1,T}
\\&\lesssim (\|\Phi\|_{\star,L,T}+1)(\| X^{\diamond 2}\|_{-1-\eta}+\|I(X^{\diamond3})\|_{\eta}+1)t^{\min(1/2-(3\eta+z)/2,-(\kappa+z)/2)}.
\end{split}
\end{equation*}   
Hence, for $0<\eta<\kappa$ and $\eta<1/2$ and $z\in(1/2,2/3)$ small enough we see that 
$$
\sup_{t\in[0,T]}t^{(\eta+z)/2}\|\Phi\|_{\eta}\lesssim T^{\kappa-\eta} (\|\Phi\|_{\star,T}+1)(\| X^{\diamond 2} \|_{\mathcal C_T^{-1-\eta}}+\|I(X^{\diamond 3})\|_{\mathcal C_T^{\eta}}+1).
$$
We focus on the explosive H\"older type norm for this term, indeed a quick computation gives
\begin{multline*}
\|\Phi_t-\Phi_s\|_{c-2d}\lesssim\|I( X^{\diamond 3} )(t)-I( X^{\diamond 3} )(s)\|_{c-2d}\\
+\|B_{<}(\Phi', X^{\diamond 2} )(t)-B_{<}(\Phi', X^{\diamond 2} )(s)\|_{c-2d}+\|\Phi^\sharp_t-\Phi^\sharp_s\|_{c-2d}.
\end{multline*}
Let us estimate the first term in the right hand side. Using the regularity for $I( X^{\diamond 3} )$ we obtain that for $d>0$ small enough and $c<1/2$ 
$$
\|I( X^{\diamond 3} )(t)-I( X^{\diamond 3} )(s)\|_{c-2d}\lesssim|t-s|^{d}\|I( X^{\diamond 3} )\|_{d,c-2d}.
$$
We notice that the increment appearing in second term has the following representation:
$$
B_{<}(\Phi', X^{\diamond 2} )=I(f),
$$
where $f=\pi_{<}(\Phi', X^{\diamond 2} )$. To treat this term it is sufficient to notice that 
$$
\|f_t\|_{-1-\delta}\lesssim\|\Phi'_{t}\|_{\eta}\| X^{\diamond 2} \|_{-1-\delta}\lesssim t^{-(\eta+z)/2}\|\Phi\|_{\star,L,T}\| X^{\diamond 2} \|_{-1-\delta}
$$
and then a usual argument gives  
$$
\|B_{<}(\Phi', X^{\diamond 2} )(t)-B_{<}(\Phi', X^{\diamond 2} )(s)\|_{c-2d}\lesssim T^\theta|t-s|^{d}t^{-(c+z)/2}\|\Phi\|_{\star,L,T}\| X^{\diamond 2} \|_{-1-\delta},
$$
for some $\theta>0$ and $c,\delta>0$. For the last term we use  that
$$
\|\Phi^\sharp_t-\Phi^\sharp_s\|_{c-2d}\lesssim|t-s|^bt^{-(a+z)/2}\|\Phi\|_{\star,T}\lesssim  T^{b-d+a-c}|t-s|^dt^{-(c+z)/2}\|\Phi\|_{\star,L,T},
$$
for $c-2d<a-2b$, $d<b$ and then $c<a$ which gives:
$$
\sup_{s,t\in[0,T]}s^{-(c+z)/2}\frac{\|\Phi_t-\Phi_s\|_{c-2d}}{|t-s|^d}\lesssim T^\theta(1+\|I( X^{\diamond 3} )\|_{d,c-2d}+\| X^{\diamond 2} \|_{-1-\delta})(1+\|\Phi\|_{\star,L,T}).
$$
Hence the following bound holds
\begin{equation}\label{eq:estim deri}
\|\Gamma(\Phi)'\|_{\star,2,L,T}\lesssim T^\theta(1+\|I( X^{\diamond 3} )\|_{d,c-2d}+\| X^{\diamond 2} \|_{-1-\delta})(1+\|\Phi\|_{\star,T}).
\end{equation}
We need to estimate the remaining term $\Gamma(\Phi)^\sharp$. Due to  Propositions~\ref{proposition:int-estim-1},~\ref{proposition:diamond} and  Corollary~\ref{corollary:integ-def-2} it only remains to estimate $I(\Phi^3)$ and $\Psi$. In fact an easy computation gives
$$
\|\Psi\|_{\star,1,L,T}\lesssim \|u^0\|_{-z}.
$$
Let us focus to the term $I(\Phi^3)$. We notice that 
$$
\|I(\Phi^3)(t)\|_{1+\delta}\lesssim\int_0^t\dd s(t-s)^{-(1+\delta-\eta)/2}s^{-3/2(\eta+z)}\|\Phi\|_{\star,T}^3(\| X^{\diamond 2} \|_{-1-\rho}+1)^3,
$$
for $\delta,\kappa>0$ small enough and $z<2/3$. Hence we obtain the existence of some $\theta>0$ such that 
$$
\sup_{t\in[0,T]}t^{(1+\delta+z)/2}\|I(\Phi^3)(t)\|_{1+\delta}\lesssim T^\theta\|\Phi\|_{\star,T}.
$$
A similar argument gives 
$$
\sup_{t\in[0,T]}t^{(1/2+\gamma+z)/2}\|I(\Phi^3)(t)\|_{1/2+\gamma}+\sup_{t\in[0,T]}t^{(\kappa+z)/2}\|I(\Phi^3)(t)\|_{\kappa}\lesssim T^\theta\|\Phi\|^3_{\star,L,T}(1+\| X^{\diamond 2} \|_{-1-\rho})^3.
$$
Let us remark that 
 $$
 \|\Phi_t^3\|_{\eta}\lesssim t^{-3(\eta+z)/2}\|\Phi\|^3_{\star,L,T}(1+\| X^{\diamond 2} \|_{-1-\rho})^3.
 $$
As usual, to deal with the H\"older norms, we begin by writing the following decomposition
$$
\|I(\Phi^3)(t)-I(\Phi^3)(s)\|_{c-2d}\lesssim I^1_{st}+I^2_{st},
$$ 
where 
$$
I^1_{st}=(P_{t-s}-1)\int_0^s\dd uP_{s-u}\Phi^3_u \quad \mathrm{and}\quad I^2_{st}=\int_s^t\dd uP_{t-u}\Phi^3_u.
$$
For $I^1$ it is enough to observe that 
\begin{equation*}
\begin{split}
\|I^1_{st}\|_{c-2d}&\lesssim|t-s|^d\int_0^s\dd u(s-u)^{-(c-\eta)/2}u^{-3/2(z+\eta)}\|\Phi\|^3_{\star,T}(1+\| X^{\diamond 2} \|_{-1-\rho})^3
\\&\lesssim T^{1-(c-\eta)-3/2(z+\eta)}|t-s|^d\|\Phi\|^3_{\star,L,T}(1+\| X^{\diamond 2} \|_{-1-\rho})^3,
\end{split}
\end{equation*}
 for $\eta,c>0$ small enough, $z<2/3$. To obtain the bound for the second term, we use the H\"older inequality and we have 
\begin{equation*}
\begin{split}
\|I^2_{st}\|_{c-2d}&\lesssim|t-s|^d\left(\int_s^t\dd u\|P_{t-u}\Phi^3_u\|^{1/(1-d)}_{c-2d}\right)^{1-d}
\\&\lesssim|t-s|^d\left(\int_s^t\dd u(t-u)^{-\frac{c-2d-\eta}{2-2d}}u^{-\frac{3(z+\eta)}{(2-2d)}}\right)^{1-d}\|\Phi\|^3_{\star,T}(1+\| X^{\diamond 2} \|_{-1-\rho})^3
\\&\lesssim |t-s|^d T^{1-(c-2\eta+3z)/2}\|\Phi\|^3_{\star,T}(1+\| X^{\diamond 2} \|_{-1-\rho})^3,
\end{split}
\end{equation*}
for $c,\eta,d>0$ small enough and $z<2/3$. We can conclude that there exists $\theta>0$ such that 
$$
\sup_{s,t}s^{(z+c)/2}\frac{\|I(\Phi^3)(t)-I(\Phi^3)(s)\|_{c-2d}}{|t-s|^d}\lesssim T^\theta\|\Phi\|^3_{\star,T}(1+\| X^{\diamond 2} \|_{-1-\rho})^3
$$
and we obtain all needed bounds for the remaining term. Hence
\begin{equation*}
\begin{split}
\|\Gamma(\Phi)^\sharp\|_{\star,2,L,T}\lesssim &(T^\theta\|\Phi\|_{\star,T}+1)^3(1+\|\mathbb X\|_{T,K}+\|u^0\|_{-z})^3,
\end{split}
\end{equation*}
for some $K\in[0,1]^4$ depending on $L$, which gives the first bound~\eqref{eq:bound}. The second estimate~\eqref{eq:contarc} is obtained in the same manner. 

Due to the bound~\eqref{eq:bound} for $T_1>T>0$ small enough, there exists $R_T>0$ such that the set $B_{R_{T}}:=\left\{\Phi\in\mathcal D_{\mathbb X,T}^L; \|\Phi\|_{\star,T}\leq R_T\right\}$ is invariant by the map $\Gamma$. The bound~\eqref{eq:contarc} tells us that $\Gamma$ is a contraction on $B_{R_{T_2}}$ for $0<T_2<T_1$ small enough. Then by the usual fixed point theorem, there exists $\Phi\in\mathcal D_{\mathbb X,T_2}^L$ such that $\Gamma(\Phi)=\Phi$. The uniqueness is obtained by a standard argument.
 \end{proof}
A quick adaptation of the last proof gives a better result (see for example \cite{gubinelli_controlling_2004} and the continuity result theorem). In fact the flow is continuous with respect to the rough distribution  $\XX$ and with respect to the initial condition $\psi$ (or $u^0$).
\begin{proposition}\label{proposition:cont}
Let $\XX$ and $\mathbb Y$ be two rough distributions such that $\mathbb \|X\|_{T,K},\mathbb \|Y\|_{T,K}\leq R$, $z\in(-2/3,-1/2)$, $u^0_X$ and $u^0_Y$ be two initial conditions and $\Phi^X\in\mathcal D^L_{T^X,X}$ and $\Phi^Y\in\mathcal D^L_{T^Y,Y}$ be the two unique solutions of the equations associating to $\mathbb X$ and $\mathbb Y$, and $T_X$ and $T_Y$ be their respective living times. For $T^\star=\inf\{T_X,T_Y\}$ the following bound holds
\[\|\Phi^X-\Phi^Y\|_{C([0,T],\mathcal C^{-z}(\mathbb T^3))}\lesssim d_{T,L}(\Phi^X,\Phi^Y)\lesssim_{R} \mathbf d_{T,K}(\XX,\mathbb Y) + \|u^0_X-u^0_Y\|_{-z}, \]
for every $T\leq T^\star$, where $d$ is defined in Definition~\ref{def:controlled-distrib} and $\mathbf d$ is defined in Defintion~\ref{hyp:renormalization1}.
\end{proposition}
Hence, using this result and combining it with the convergence Theorem~\ref{theorem:reno-1} , we have this second corollary, where the convergence of the approximated equation is proved.
\begin{corollary}
Let $z\in(1/2,2/3)$, $u^0\in \mathcal C^{-z}$ and let us denote by $u^\eps$ the unique solution (with life time $T^\eps$) of the equation 
$$
\partial_tu^\eps=\Delta u^\eps-(u^\eps)^3+C^{\eps}u^{\eps}+\xi^\eps,
$$
where $\xi^\eps$ is a mollification of the space-time white noise $\xi$ and $C^\eps=3(C_1^\eps-3C_2^\eps)$ where $C_1^\eps$ and $C_2^\eps$ are the constants given in Definition~\ref{def:constant}. Let
us introduce $u=X+\Phi$ where $\Phi$ is the local solution with life-time $T>0$ for the fixed point equation given in the Theorem~\ref{theorem:FP}. Then we have the following convergence result 
$$
\mathbb P(d_{T^\star,L}(\Phi^\eps,\Phi)>\lambda)\longrightarrow_{\eps\to0}0,
$$
for all $\lambda>0$ with $T^\star=\inf(T,T^\eps)$ and $\Phi^\eps=u^\eps-X^\eps\in\mathcal D_{X^\eps,T}^L$. 
\end{corollary}
\section{Renormalization and construction of the rough distribution }\label{section:renormalization}
To end the proof of existence and uniqueness for the renormalized equation, we need to prove that the O.U. process associated to the white noise can be extended to a rough distribution in the space $\mathcal X$ (see Definition~\ref{hyp:renormalization1}). As explained above, to define the appropriate process we proceed by regularization and renormalization. Let us take a \emph{a smooth radial function $f$ with compact support and such that $f(0)=1$}. We regularize $X$ in the following way 
\[
X_t^{\eps}=\sum_{k\ne0}f(\eps k)\hat X_t(k)e_k
\]
and  we show that we can choose two divergent constants $C_1^\eps,C_2^\eps\in\mathbb R^+$ and a smooth function $\varphi^\eps$ such that  $R^{\varphi^\eps}_{C_1^\eps,C_2^\eps}\mathbf{X^\eps}:= X^\eps$  converges in $\mathcal X$. As it has been noticed in the previous sections, without a renormalization procedure there is no finite limit for such a process.
\begin{notation}
Let $k_1,. . .,k_n\in\mathbb Z^3$ we denote by $k_{1,. . .,n}=\sum_{i=1}^nk_i$, and for a function $f$ we denote by $\der f$ the increment of the function given by 
$\der f_{st}=f_t-f_s$. 
\end{notation}
\begin{definition}\label{def:constant}
Let 
\[C_1^\eps=\mathbb E\left[(X^\eps)^2\right]\]
and 
\[C^\varepsilon_2=2\sum_{k_1\neq0,k_2\neq0}\frac{|f(\varepsilon k_1)|^2|f(\varepsilon k_2)|^2}{|k_1|^2|k_2|^2(|k_1|^2+|k_2|^2+|k_{1,2}|^2)}.\]
Let us notice that thanks to the definition of the Littlewood-Paley blocs, we can also choose to write $C^\eps_2$ as
\[C^\varepsilon_2=2\sum_{|i-j|\le1}\sum_{k_1\neq0,k_2\neq0}\theta(2^{-i}|k_{1,2}|)\theta(2^{-j} |k_{1,2}|)\frac{f(\varepsilon k_1)f(\varepsilon k_2)}{|k_1|^2|k_2|^2(|k_1|^2+|k_2|^2+|k_{1,2}|^2)}.\]
Let us define the following renormalized quantities
\begin{align*}
				    (X^\varepsilon)^{\diamond 2} &:= (X^{\varepsilon})^{2} - C^\varepsilon_1,\\
				    I((X^\varepsilon)^{\diamond 3})&:=I( (X^\varepsilon)^{3} - 3C^\varepsilon_1 X^\varepsilon),\\
				    (I((X^\varepsilon)^{\diamond 2})\circ (X^\varepsilon)^{\diamond 2})^{\diamond}&=I((X^\varepsilon)^{\diamond 2})\circ (X^\varepsilon)^{\diamond 2}) -C^\varepsilon_2\\
				   \intertext{and}
				    (I((X^\varepsilon)^{\diamond 3})\circ (X^\varepsilon)^{\diamond 2})^{\diamond}
				  &=I((X^\varepsilon)^{\diamond 3})\circ (X^\varepsilon)^{\diamond 2}-3C^\eps_2 X^{\eps}.
\end{align*}
\end{definition}
Then the following theorem holds:
\begin{theorem}\label{theorem:reno-1}
For $T>0$, there exists a deterministic sequence of functions $\varphi^\eps : [0,T]\to\RR$, a deterministic distribution $\varphi : [0,T]\to \RR$ such that for all $\delta,\delta',\nu, \rho>0$ small enough with $\nu>\rho$ we have
$$\|\varphi\|_{\nu,\rho,T}=\sup_t t^{\nu}|\varphi_t|+\sup_{0\le s<t\le T} s^{\nu}\frac{|\varphi_t-\varphi_s|}{|t-s|^{\rho}}<+\infty$$
and the sequence $\varphi^\eps$ converges to $\varphi$ according to that norm, that is
$$\|\varphi^\eps-\varphi\|_{1,\star,T}\to 0.$$
Furthermore there exists some stochastic processes
\[X^{\diamond 2} \in\mathcal C([0,T],\mathcal C^{-1-\delta}),\] 
\[I(X^{\diamond3})\in \mathcal C^{\delta'}([0,T],\mathcal C^{1/2-\delta-2\delta'}),\]
\[I( X^{\diamond 3} )\circ X\in\mathcal  C^{\delta'}([0,T],\mathcal C^{-\delta-2\delta'}),\]
\[(I( X^{\diamond 2} )\circ X^{\diamond 2})^{\diamond}-\varphi \in \mathcal C^{\delta'}([0,T],\mathcal C^{-\delta-2\delta'}),\]
and
\[(I(X^{\diamond 3})\circ X^{\diamond 2})^{\diamond}-3\varphi X \in \mathcal C^{\delta'}([0,T],\mathcal C^{-1/2-\delta-2\delta'}),\]
such that each component of the sequence $\XX^\eps$ converges respectively to  the corresponding component of the rough distribution  $\XX$ in the good topology, that is for all $\delta, \delta' >0$ small enough, and all $p>1$,
\begin{equation}\label{eq:cv_X}
X^\eps\to X\mathrm{\ in\ } L^p(\Omega,\mathcal C^{\delta'}([0,T],\mathcal C^{-1-\delta-3\delta'-3/2p})),
\end{equation}
\begin{equation}\label{eq:cv_X2}
(X^\eps)^{\diamond 2}\to X^{\diamond 2} \mathrm{\ in \ } L^p(\Omega,\mathcal C^{\delta'}([0,T],\mathcal C^{-1-\delta-3\delta'-3/{2p}})),
\end{equation}
\begin{equation}\label{eq:cv_X3}
I((X^\eps)^{\diamond 3})\to I(X^{\diamond 3}) \mathrm{\ in \ } L^p(\Omega,\mathcal C^{\delta'}([0,T],\mathcal C^{1/2-\delta-3\delta'-3/{2p}})),
\end{equation}
\begin{equation}\label{eq:cv_I(X3)X}
I((X^\eps)^{\diamond 3})\circ X^\eps \to I(X^{\diamond 3})\circ X \mathrm{\ in \ } L^p(\Omega,\mathcal C^{\delta'}([0,T],\mathcal C^{-\delta-3\delta'-3/{2p}})),
\end{equation}
\begin{equation}\label{eq:cv_I(X2)X2}
(I((X^\eps)^{\diamond 2})\circ(X^\eps)^{\diamond 2})^{\diamond}-\varphi^\eps \to (I(X^{\diamond 2})\circ X^{\diamond 2})^{\diamond} -\varphi \mathrm{\ in \ } L^p(\Omega,\mathcal C^{\delta'}([0,T],\mathcal C^{-\delta-3\delta'-3/{2p}}))
\end{equation}
and
\begin{equation}\label{eq:cv_I(X3)X2}
(I((X^\eps)^{\diamond 3})\circ(X^\eps)^{\diamond 2})^{\diamond} - 3 \varphi^\eps X^{\eps} \to(I(X^{\diamond 3})\circ X^{\diamond 2})^{\diamond}-3\varphi X \mathrm{\ in \ } L^p(\Omega,\mathcal C^{\delta'}([0,T],\mathcal C^{-1/2-\delta-3\delta'-3/{2p}})).
\end{equation}
\end{theorem}
\begin{remark}
Thanks to the proof below (especially  Subsections~\ref{subsec:Renor-I(X2)X2} and~\ref{subsec:Renor-I(X3)X2}) we have the following expressions for $\varphi^\eps$ and $\varphi$.
\[\varphi^\eps_t =-\sum_{|i-j|\leq1}\sum_{k_1\neq0 , k_2\neq 0}\frac{|\theta(2^{-i}|k_{12}|)\|\theta(2^{-j}|k_{12}|)\|f(\eps k_1)f(\eps k_2)|}{|k_1|^2|k_2|^2(|k_1|^2 + |k_2|^2 + |k_1+k_2|^2)}\exp\left(-t(|k_1|^2 + |k_2|^2 + |k_1+k_2|^2)\right)\]
and
\[\varphi_t =-\sum_{|i-j|\leq1}\sum_{k_1\neq0 , k_2\neq 0}\frac{|\theta(2^{-i}|k_{12}|)\|\theta(2^{-j}|k_{12}|)|}{|k_1|^2|k_2|^2(|k_1|^2 + |k_2|^2 + |k_1+k_2|^2)}\exp\left(-t(|k_1|^2 + |k_2|^2 + |k_1+k_2|^2)\right).\]
\end{remark}
We split the proof of this theorem according to the various components. We start by the convergence of $X^\eps$ to $X$. We  also give a full proof for $X^{\diamond 2}$. For the other components we only give the crucial estimates.
\subsection{Convergence to \texorpdfstring{$X$}{X}}
We start by an easy computation for the convergence of $X^\eps$
\begin{proof}[Proof of~\eqref{eq:cv_X}]
By a quick computation we have that 
$$
\der (X-X^\eps)_{st}=\sum_{k}(f(\eps k)-1)\der \hat X_{st}(k)e_k.
$$
Then 
$$
\mathbb E\left[|\Delta_{q}\der(X-X^\eps)_{st}|^2\right]=2\sum_{k\ne0;|k|\sim 2^q}|f(\eps k)-1|^2\frac{1-e^{-|k|^2|t-s|}}{|k|^2}\lesssim_{h,\rho}c(\eps)2^{q(1+2h+\rho)}|t-s|^h,
$$
for $h,\rho>0$ small enough, and $c(\eps)=\sum_{k\ne0}|k|^{-3-\rho}|f(\eps k)-1|^2$. The Gaussian Hypercontractivity gives  for $p>2$,
$$
\mathbb E\left[\|\Delta_q\der(X-X^\eps)_{st}\|^p_{L^p}\right]\lesssim_p\int_{\mathbb T^3}\mathbb E\left[|\Delta_q\der (X-X^\eps)_{st}(x)|^2\right]^{p/2}\dd x\lesssim_{\rho,h}c(\eps)^p|t-s|^{hp/2}2^{qp/2(2h+\rho+1)}.
$$
We obtain that 
$$
\mathbb E\left[\|\der (X-X^\eps)_{st}\|^{p}_{B^{-1/2-\rho-h}_{p,p}}\right]\lesssim c(\eps)^{p/2}|t-s|^{hp/2}.
$$
Using the Besov embedding (Proposition~\ref{proposition:Bes-emb}) we have 
$$
\mathbb E\left[\|\der (X-X^\eps)_{st}\|^{p}_{\mathcal C^{-1/2-\rho-h-3/p}}\right]\lesssim c(\eps)^{p/2}|t-s|^{hp/2}
$$
and by the standard Garsia-Rodemich-Rumsey Lemma (see~\cite{grr}) we finally obtain 
$$
\mathbb E\left[\|X-X^\eps\|_{\mathcal C^{h-\theta}([0,T],\mathcal C^{-1/2-h-\rho-3/p})}\right]\lesssim c(\eps)^p,
$$
for all $h>\theta>0$, $\rho>0$ small enough and $p>2$. Moreover we have $X_0=X^\eps_0=0$ and by using the fact that $c(\eps)\to^{\eps\to0}0$, we obtain that 
$$
\lim_{\eps\to0}\|X^\eps-X\|_{L^p(\Omega,\mathcal C_T^{\delta',-1/2-\delta-3/p})}=0,
$$ 
for all $0<\delta'<\delta/3$ and $T>0$.
\end{proof}
 \subsection{Renormalization for \texorpdfstring{$X^2$}{X2}}
To prove the theorem for $X^{\diamond 2}$ we first prove the following estimate, and we use the Garsia-Rodemich-Rumsey lemma to conclude.
\begin{proposition} \label{prop:holder estimate X2}
Let $p>1$, $\theta>0$ be small enough, then the following bounds hold 
$$
\sup_{\eps}\mathbb E\left[\|\Delta_q\der (X^\eps)^{\diamond 2}_{st}\|^{2p}_{L^{2p}}\right]\lesssim_{p,\theta}|t-s|^{p\theta}2^{2qp(1+2\theta)}
$$
and 
$$
\mathbb E\left[\|\Delta_q\left(\der (X^\eps)^{\diamond 2}_{st}-\der (X^{\eps'})^{\diamond 2}_{st}\right))\|^{2p}_{L^{2p}}\right]\lesssim_{p,\theta}C(\eps,\eps')^p|t-s|^{2p\theta}2^{2qp(1+\theta)},
$$
where $C(\eps,\eps')\to0$ when $|\eps-\eps'|\to0$.
\end{proposition}
\begin{proof}
By a straighforward computation we have 
\begin{equation}\label{eq:var1}
\begin{split}
 \var(\Delta_q((X^\eps_t-X^\eps_s)X^\eps_s))=&\sum_{k,k'\in\mathbb Z^3}\theta(2^{-q}k)\theta(2^{-q}k')\sum_{k_{12}=k;k'_{12}=k'}f(\eps k_1)f(\eps k_2)f(\eps k'_1)f(\eps k'_2)
 \\&\times(I^1_{st}+I^2_{st})e_ke_{-k'},
 \end{split}
\end{equation}
where $(e_k)$ denotes the Fourier basis of $L^2(\mathbb T^3)$ and 
$$
I^{1}_{st}=\mathbb E\left[(\hat X_t(k_1)-\hat X_s(k_1))(\overline{\hat X_t(k'_1)-\hat X_s(k'_1)})\right] \mathbb E\left[\hat X_s(k_2)\overline X_s(k'_2)\right]=2\der_{k_1=k'_1}\der_{k_2=k'_2}\frac{1-e^{-|k_1|^2|t-s|}}{|k_1|^2|k_2|^2}
$$ 
and
\begin{equation*}
\begin{split}
I^{2}_{st}&=\mathbb E\left[(\hat X_t(k_1)-\hat X_s(k_1))\overline{\hat X_{s}(k'_2)}\right]\mathbb E\left[(\overline{\hat X_t(k'_1)-\hat X_s(k'_1)}) \hat X_{s}(k_2)\right]
\\&=\der_{k_1=k_2'}\der_{k'_1=k_2}\frac{(1-e^{-|k_1|^2|t-s|})(1-e^{-|k_2|^2|t-s|})}{|k_1|^2|k_2|^2}.
\end{split}
\end{equation*}
Injecting these two identities in Equation~\eqref{eq:var1} we obtain  
\begin{equation}
\begin{split}
 \var(\Delta_q((X^\eps_t-X^\eps_s)X^\eps_s))&\lesssim\sum_{\substack{|k|\sim 2^q\\k_{12}=k}}\frac{1-e^{-|k_1|^2|t-s|}}{|k_1|^2|k_2|^2}+\sum_{\substack{|k|\sim2^q\\k_{12}=k}} \frac{(1-e^{-|k_1|^2|t-s|})(1-e^{-|k_2|^2|t-s|})}{|k_1|^2|k_2|^2}
 \\&\lesssim \sum_{\substack{|k|\sim 2^q\\k_{12}=k}}\frac{1-e^{-|k_1|^2|t-s|}}{|k_1|^2|k_2|^2}.
 \end{split}
\end{equation}
We have 
\begin{equation*}
\begin{split}
\sum_{\substack{|k|\sim 2^q\\k_{12}=k,|k_1|\leq|k_2|}}\frac{1-e^{-|k_1|^2|t-s|}}{|k_1|^2|k_2|^2}\lesssim&|t-s|^\theta\sum_{k\in\mathbb Z^3;|k|\sim2^q,k_{12}=k}|k_1|^{-2+2\theta}|k_2|^{-2}
\\\lesssim & |t-s|^\theta\left\{\sum_{\substack{|k|\sim2^q,k_{12}=k,\\|k_1|\leq|k_2|}}|k_1|^{-2+2\theta}|k_2|^{-2}
+\sum_{\substack{|k|\sim2^q,k_{12}=k\\|k_1|\geq|k_2|}}|k_1|^{-2+2\theta}|k_2|^{-2}\right\}
\\\lesssim& |t-s|^{\theta}2^{2q(1+2\theta)}\left(\sum_{k_1}|k_1|^{-3-2\theta}+\sum_{k_1}|k_2|^{-3-4\theta}\right)<+\infty
\end{split}
\end{equation*}
and by the Gaussian hypercontractivity we finally have
$$
\mathbb E\left[\|\Delta_q\der (X^\eps)^{\diamond 2}_{st}\|_{L^{2p}}^{2p}\right]=\int_{\mathbb T^{3}}(Var\left(\der(X^\eps)^{\diamond 2}_{st}\right)(\xi))^p\dd\xi\lesssim|t-s|^{p\theta}2^{2qp(1+2\theta)}.
$$ 
For the second assertion we see that the computation of the beginning gives
$$
\var ((\Delta_q((X^\eps_t-X^\eps_s)X^\eps_s)-(X^\eps_t-X^\eps_s)X^\eps_s))\lesssim|t-s|^{\theta}2^{2q(1+3\theta)}C(\eps,\eps'),
$$
where
$$
C(\eps,\eps')=\sum_{k_{12}=k}(|f(\eps k_1)|^2|f(\eps k_2)|^2-|f(\eps' k_1)|^2|f(\eps' k_2)|^2)|k|^{-3-\theta}|k_1|^{-3-2\theta}\to^{|\eps-\eps'|\to0}0
$$
thanks to the dominated convergence theorem. Once again the Gaussian hypercontractivity gives us the needed bounds.  
\end{proof}
Using the Besov embedding (Proposition~\ref{proposition:Bes-emb})  combined with the standard Garsia-Rodemich-Rumsey lemma (see~\cite{grr}) the following convergence result holds.
\begin{proposition}\label{propsoition:renorm-1}
Let $\theta,\delta,\rho>0$ be small enough such that $\rho<\theta/2$ and $p>1$ then the following bound holds:
$$
\mathbb E\left[\|(X^\eps)^{\diamond 2}-(X^{\eps'})^{\diamond 2}\|^{2p}_{\mathcal C^{\theta/2-\rho}([0,T],\mathcal C^{-1-3/(2p)-\delta-2\theta}})\right]\lesssim_{\theta,p,\delta} C(\eps,\eps')^p.
$$
Since $(X_0^\eps)^{\diamond2}=0$ and $(X^{\diamond2})_0=0$,  the sequence $ (X^\eps)^{\diamond 2}$ converges in $L^{2p}(\Omega,C^{\theta/2-\rho}([0,T],\mathcal C^{-1-3/(2p)-\delta-3\theta}))$ to a random field denoted by $ X^{\diamond 2} $. 
\end{proposition}

\subsection{Renormalization for \texorpdfstring{$I(X^3)$}{I(X3)}}
As the computations are quite similar, we only prove the equivalent of the $L^2$ estimates in Proposition~\ref{prop:holder estimate X2}. Furthermore we only prove it for a fixed time $t$ and not for an increment.

\begin{proof}[Proof of~\eqref{eq:cv_X3}]
By a simple computation we have  
$$
I((X_t^\eps)^{\diamond 3})=\sum_{k\in\mathbb Z^3}\left(\int_0^t\mathcal F((X_s^\eps)^{\diamond 3})(k)e^{-|k|^2|t-s|}\dd s\right)e_k.
$$
Hence 
\begin{equation*}
\begin{split}
\mathbb E\left[\left|\Delta_qI((X_t^\eps)^{\diamond 3})\right|^2\right]&=6\sum_{\substack{k\in\mathbb Z^3\\k_{123}=k}}|\theta(2^{-q}k)|^2\prod_{i=1,..,3}\frac{|f(\eps k_i)|^2}{|k_i|^2}\int_0^t\dd s\\
&\qquad \qquad \times \int_0^s\dd \sigma e^{-(|k_1|^2+|k_2|^2+|k_3|^2)|s-\sigma|-|k|^2(|t-s|+|t-\sigma|)}
\\&=\sum_{k}|\theta(2^{-q}k)|^2\Xi^{\eps,1}(k),
\end{split}
\end{equation*}
where  
\begin{equation*}
\begin{split}
\Xi^{\eps,1}(k)&=\sum_{k_{123}=k,k_i\ne0}\prod_{i=1,..,3}\frac{|f(\eps k_i)|^2}{|k_i|^2}\int_0^t\dd s\int_0^s\dd \sigma e^{-(|k_1|^2+|k_2|^2+|k_3|^2)|s-\sigma|-|k|^2(|t-s|+|t-\sigma)}
\\&\lesssim\sum_{\substack{k_{123}=k\\\max_{i=1,..3}|k_i|=|k_1|}}\frac{1}{|k_1|^2|k_2|^2|k_3|^2}\int_0^t\dd s\int_0^s\dd \sigma e^{-(|k_1|^2+|k_2|^2+|k_3|^2)|s-\sigma|-|k|^2|t-s|}
\\&\lesssim_T\frac{1}{|k|^{2-\rho}}\sum_{\substack{k_{123}=k\\\max_{i=1,..3}|k_i|=|k_1|}}\frac{1}{|k_1|^{4-\rho}|k_2|^2|k_3|^2}\lesssim_T\frac{1}{|k|^{4-4\rho}}(\sum_{k_2}|k_2|^{-3-\rho})^2.
\end{split}
\end{equation*}
We have used that
$$
\int_0^t\dd s\int_0^s\dd \sigma e^{-(|k_1|^2+|k_2|^2+|k_3|^2)|s-\sigma|-|k|^2(|t-s|+|t-\sigma)}\lesssim_{T}\frac{1}{|k_1|^{2-\rho}|k|^{2-\rho}}\int_0^t\dd s\int_0^s\dd \sigma |t-s|^{-1+\rho/2}|s-\sigma|^{-1+\rho/2},
$$
for $\rho>0$ small enough. Using again the Gaussian hypercontractivity we have
$$
\mathbb E\left[\left|\left|\Delta_qI((X_t^\eps)^{\diamond 3})\right|\right|^{2p}_{L^{2p}}\right]\lesssim 2^{-2pq(1/2-\rho)}
$$
and  the Besov embedding gives
$$
\sup_{t\in[0,T],\eps}\mathbb E\left[\|I((X_t^\eps)^{\diamond 3})\|_{1/2-\rho-3/p}\right]<+\infty.
$$
The same computation gives 
$$
\sup_{t\in[0,T]}\mathbb E\left[\| I((X_t^{\eps})^{\diamond 3})-I((X_t^{\eps'})^{\diamond 3} \|^{2p}_{1/2-\rho-3/p}\right]\to^{|\eps'-\eps|\to0}0,
$$
which gives the needed convergence.
\end{proof}

\subsection{Renormalization for \texorpdfstring{$I( X^{\diamond 3} )\circ X$}{I(X3,X)}}
Here, we only prove the $L^2$ estimate for the term $I(X^{\diamond3})X$ instead of $I(X^{\diamond3})\circ X$ since the computations in the two cases are essentially similar. We remark that fo this term we do not need a renormalization.  
\begin{proof}[Proof of~\eqref{eq:cv_I(X3)X}]
We have the following representation formula 
$$
\mathbb E\left[\left|\Delta_q \left(I((X^\eps)^{\diamond 3} X^\eps \right)(t)\right|^2\right]=\sum_{k}|\theta(2^{-q}k)|^2(6I^\eps_1(t)(k)+18I_2^\eps(t)(k)+18I_3^\eps(t)(k)),
$$
where
\[I^\eps_1(t)(k)=2\sum_{k_{1234}=k}\prod_{i=1,..,4}\frac{|f(\eps k_i)|^2}{|k_i|^2}\int_0^t\dd s\int_0^s\dd\sigma e^{-|k_{123}|^2(|t-s|+|t-\sigma|)-(|k_1|^2+|k_2|^2+|k_3|^2)|s-\sigma|},\]
\[I_{12}^\eps(t)(k)=\sum_{\substack {k_{1234}=k\\\max_{i=1,2,3}|k_i|=|k_1|\\|k_{123}|\geq|k_4|}}\frac{1}{|k_1|^2|k_2|^2|k_3|^2|k_4|^2}\int_0^t\dd s\int_0^s\dd\sigma e^{-|k_{123}|^2(|t-s|+|t-\sigma|)-(|k_1|^2+|k_2|^2+|k_3|^2)|s-\sigma|}\]
and
\[I_3^\eps(t)(k)=\sum_{\substack{k_{12}=k\\k_3,k_4}}\prod_{i=1}^{4}\frac{|f(\eps k_i)|^2}{|k_i|^2}\int_0^t\dd s\int_0^t\dd \sigma e^{-(|k_1|^2+|k_2|^2)|s-\sigma|-(|k+k_3|^2+|k_3|^2)|t-s|-(|k_4|^2+|k+k_4|^2)|t-\sigma|}.\]
We have 
\begin{equation*}
\begin{split}
I^\eps_1(t)(k)
&
\lesssim
	\sum_{\substack {k_{1234}=k\\\max_{i=1,2,3}|k_i|=|k_1|}}
		\frac{1}{|k_1|^2|k_2|^2|k_3|^2|k_4|^2}
		\int_0^t\dd s\int_0^s\dd\sigma 
			e^{-|k_{123}|^2(|t-s|+|t-\sigma|)-(|k_1|^2+|k_2|^2+|k_3|^2)|s-\sigma|} 
\\&
\lesssim I_{11}^{\eps}(t)(k)+I_{12}^\eps(t)(k),
\end{split}
\end{equation*}
where
\[I_{11}^\eps(t)(k) = \sum_{\substack {k_{1234}=k\\\max_{i=1,2,3}|k_i|=|k_1|\\|k_{123}|\leq|k_4|}}\frac{1}{|k_1|^2|k_2|^2|k_3|^2|k_4|^2}\int_0^t\dd s\int_0^s\dd\sigma e^{-|k_{123}|^2(|t-s|+|t-\sigma|)-(|k_1|^2+|k_2|^2+|k_3|^2)|s-\sigma|}\]
and
\[I_{12}^\eps(t)(k)\lesssim\sum_{\substack {k_{1234}=k\\\max_{i=1,2,3}|k_i|=|k_1|\\|k_{123}|\geq|k_4|}}\frac{1}{|k_1|^2|k_2|^2|k_3|^2|k_4|^2}\int_0^t\dd s\int_0^s\dd\sigma e^{-|k_{123}|^2(|t-s|+|t-\sigma|)-(|k_1|^2+|k_2|^2+|k_3|^2)|s-\sigma|}.\]
Hence
\begin{equation*}
\begin{split}
I_{11}^\eps(t)(k)
&\lesssim\frac{1}{|k|^2}\sum_{k_2,k_3,k_1,\max|k_i|=|k_1|}\frac{1}{|k_1|^{4-\rho}|k_2|^2|k_3|^2|k_{123}|^{2-\rho}} 
\\&
\lesssim\frac{1}{|k|^2}\sum_{k_1,k_2,k_3}\frac{1}{|k_2|^{3+\rho}|k_3|^{3+\rho}|k_{123}|^{3+\rho}}\lesssim|k|^{-2},
\end{split}
\end{equation*}
for $\rho>0$ small enough, which is the needed result for $I_{1,1}^\eps$. We can treat the second term by a similar computation, indeed 
\begin{equation*}
I_{12}^\eps(t)(k)\lesssim|k|^{-2+\rho}\sum_{k_2,k_3,k_4}|k_2|^{-3-\rho}|k_2|^{-3-\rho}|k_3|^{-3-\rho}\lesssim|k|^{-2+\rho},
\end{equation*}
for $\rho>0$ small enough. This gives the bound for $I_{1,2}^\eps$ and $I_1^\eps$. More precisely we have $I_1^\eps(t)(k)\lesssim|k|^{-2+\rho}$ for $\rho>0$ small enough.  Let us focus on the second term $I_2^\eps(t)(k)$. We have
\begin{equation*}
\begin{split}
I_2^\eps(t)(k)
&\lesssim 
\sum_{\substack{k_{12}=k\\ \max_{i=1,2}|k_i|=|k_1|\\k_3,k_4}}\frac{1}{|k_1|^{1-\rho}|k_2|^{3+\rho}|k_3|^2|k_4|^2}\int_0^t\dd s\int_0^t\dd\sigma e^{-(|k-k_3|^2+|k_3|^2)|t-s|-(|k_4|^2+|k-k_4|^2)|t-\sigma|}
\\&\lesssim_{\rho}\frac{1}{|k|^{1-\rho}}\left(\sum_{k_3}\frac{1}{|k_3|^2}\int_0^t\dd se^{-(|k_3|^2+|k-k_3|^2)|t-s|}\right)^2
\\&\lesssim_{\rho,T}\frac{1}{|k|^{1-\rho}}\left(\sum_{k_3\ne0,k}\frac{1}{|k_3|^2|k-k_3|^{2-\rho}}\right)^2\lesssim_{T,\rho}\frac{1}{|k|^{3-3\rho}}
\end{split}
\end{equation*}
and we obtain the bound for $I_2^\eps$. We notice that 
\begin{align*}
I_3^\eps(t)(k)&=\sum_{\substack{k_{12}=k\\k_3,k_4}}\prod_{i=1}^{4}\frac{|f(\eps k_i)|^2}{|k_i|^2}\int_0^t\dd s\int_0^t\dd \sigma e^{-(|k_1|^2+|k_2|^2)|s-\sigma|-(|k+k_3|^2+|k_3|^2)|t-s|-(|k_4|^2+|k+k_4|^2)|t-\sigma|}\\
&=I_2^{\eps}(t)(k).
\end{align*}
Finally we have
$$
\sup_{t\in[0,T],\eps}\mathbb E\left[\left|\Delta_q \left(I((X^\eps)^{\diamond 3} X^\eps\right)(t)\right|^2\right]\lesssim_{\rho,T}2^{q(1+\rho)},
$$
which is the wanted bound.
\end{proof}
 \subsection{Renormalization for \texorpdfstring{$I( X^{\diamond 2} )\circ X^{\diamond 2}$}{I(X2)X2}}\label{subsec:Renor-I(X2)X2}
We only prove the crucial estimate for a renormalization of $I((X^\varepsilon))^{\diamond 2}\circ (X^\varepsilon)^{\diamond 2})$. We recall that this is enough since all the other terms of the product $\left(I(X^\varepsilon)^{\diamond 2}(X^\varepsilon)^{\diamond 2}\right)^{\diamond}$ are well-defined and converge to a limit with a good regularity.
\begin{proof}[Proof of~\eqref{eq:cv_I(X2)X2}]
A chaos decomposition gives  
{\small
\begin{eqnarray*}
\lefteqn{-(I((X^\eps)^{\diamond 2})(t)\circ(X_t^\eps)^{\diamond 2})=}\\
&&\sum_{k\in\mathbb Z^3}\sum_{|i-j|\leq1}\sum_{k_{1234}=k}\theta(2^{-i}|k_{12}|)\theta(2^{-j}|k_{34}|)
\\&&\times\int_0^t\dd se^{-|k_{12}|^2|t-s|}:\hat X^\eps_s(k_1)\hat X^\eps_s(k_2)\hat X^\eps_t(k_3)\hat X^\eps_t(k_4):e_k
\\&&+4\sum_{k\in\mathbb Z^3}\sum_{|i-j|\leq1}\sum_{k_{13}=k,k_2}\theta(2^{-i}(|k_{12}|)\theta(2^{-j}|k_{2(-3)}|)|f(\eps k_2)|^2\int_0^t\dd se^{-(|k_{12}|^2+|k_2|^2)|t-s|}|k_2|^{-2}:\hat X^\eps_{s}(k_1)\hat X_t^\eps(k_3):e_k
\\&&+2\sum_{|i-j|\leq1}\sum_{k_1,k_2}\theta(2^{-i}|k_{12}|)\theta(2^{-j}|k_{12}|)|f(\eps k_1)|^2|f(\eps k_2)|^2\frac{1-e^{-(|k_1|^2+|k_2|^2+|k_{12}|^2)t}}{|k_1|^2|k_2|^2(|k_1|^2+|k_2|^2+|k_{12}|^2)},
\end{eqnarray*}}
where $:\, :$ denotes the usual Gaussian Wick product.
Let us focus on the last term
$$
A^\eps(t)=\sum_{|i-j|\leq1}\sum_{k_1,k_2}|\theta(2^{-i}k_{12})\|\theta(2^{-j}k_{12})\|f(\eps k_1)|^2\|f(\eps k_2)|^2\frac{1-e^{-(|k_1|^2+|k_2|^2+|k_{12}|^2)t}}{|k_1|^2|k_2|^2(|k_1|^2+|k_2|^2+|k_{12}|^2)}
= C_2^\eps+I_3^\eps(t),
$$ 
where $I_3^\eps$ is defined below. Moreover it is not difficult to see that 
$$
\lim_{\eps\to0}C_2^\eps=\sum_{|i-j|\leq1}\sum_{k_1,k_2}\frac{\theta(2^{-i}|k_{12}|)\theta(2^{-j}|k_{12}|)}{|k_1|^2|k_2|^2(|k_1|^2+|k_2|^2+|k_{12}|^2)}=+\infty.
$$
To obtain the needed convergence we have to estimate the following terms: 
$$
I_1^\eps(t)=\sum_{k\in\mathbb Z^3}\sum_{|i-j|\leq1}\sum_{k_{1234}=k}\theta(2^{-i}|k_{12}|)\theta(2^{-j}|k_{34}|)\int_0^t\dd se^{-|k_{12}|^2|t-s|}:\hat X^\eps_s(k_1)\hat X^\eps_s(k_2)\hat X^\eps_t(k_3)\hat X_t(k_4):e_k,
$$
\begin{multline*}
I^\eps_2(t)=\sum_{k\in\mathbb Z^3}\sum_{|i-j|\leq1}\sum_{k_{13}=k,k_2}\theta(2^{-i}(|k_{12}|)\theta(2^{-j}|k_{2(-3)}|)|f(\eps k_2)|^2\\
\times\int_0^t\dd se^{-(|k_{12}|^2+|k_2|^2)|t-s|}|k_2|^{-2}\hat X^\eps_{s}(k_1)\hat X_t^\eps(k_3)e_k
\end{multline*}
and 
$$
I_3^\eps(t)=\sum_{|i-j|\leq1}\sum_{k_1,k_2}\theta(2^{-i}|k_{12}|)\theta(2^{-j}|k_{12}|)\frac{|f(\eps k_1)|^2|f(\eps k_2)|^2e^{-(|k_1|^2+|k_2|^2+|k_{12}|^2)t}}{|k_1|^2|k_2|^2(|k_1|^2+|k_2|^2+|k_{12}|^2)}.
$$
We notice  that for the deterministic part we have the following bound 
$$
I_3^\eps(t)\lesssim t^{-\rho}\sum_{k_1,k_2,|k_1|\leq|k_2|}\frac{1}{|k_1|^2|k_2|^2(|k_1|^2+|k_2|^2+|k_{12}|^2)^{1+\rho}}\lesssim t^{-\rho}\sum_{k_1,k_2,|k_1|\leq|k_2|}|k_2|^{-4-2\rho}|k_1|^{-2}\lesssim_{\rho}t^{-\rho}
$$
and  the dominated convergence gives  for $\rho>0$ 
$$
\sup_{t\in[0,T]}t^\rho|I_3^\eps(t)-I_3(t)|\to^{\eps\to0}0,
$$
where 
$$
I_3(t)=\sum_{|i-j|\leq1}\sum_{k_1,k_2}\theta(2^{-i}|k_{12}|)\theta(2^{-j}|k_{12}|)\frac{e^{-(|k_1|^2+|k_2|^2+|k_{12}|^2)t}}{|k_1|^2|k_2|^2(|k_1|^2+|k_2|^2+|k_{12}|^2)}.
$$
 Let us focus on $I_1^\eps(t)$ and $I^\eps_2(t)$. A straightforward computation gives 

{\small
\begin{equation*}
\begin{split}
\mathbb E\left[\Delta_q|I^\eps_1(t)|^2\right]&=2\sum_{k\in\mathbb Z^3}\sum_{i\sim j\sim i'\sim j'}\sum_{k_{1234}=k}\theta(2^{-i}|k_{12}|)\theta(2^{-j}|k_{34}|)\theta(2^{-i'}|k_{12}|)\theta(2^{-j'}|k_{34}|)\theta(2^{-q}|k|)^2\prod_{l=1}^{4}\frac{|f(\eps k_l)|^2}{|k_l|^2}
\\&\times\int_0^t\int_0^t\dd s\dd \sigma e^{-|k_{12}|^2(|t-s|+|t-\sigma|)-(|k_1|^2+|k_2|^2)|s-\sigma|}
\\&+2\sum_{k\in\mathbb Z^3}\sum_{i\sim j\sim i'\sim j'}\sum_{k_{1234}=k}\theta(2^{-i}|k_{12}|)\theta(2^{-j}|k_{34}|)\theta(2^{-j'}|k_{12}|)\theta(2^{-i'}|k_{34}|)\theta(2^{-q}|k|)^2\prod_{l=1}^{4}\frac{|f(\eps k_l)|^2}{|k_l|^2}
\\&\times\int_0^t\int_0^t\dd s\dd\sigma e^{-(|k_{12}|^2+|k_1|^2+|k_2|^2)|t-s|-(|k_{34}|^2+|k_3|^2+|k_4|^2)|t-\sigma|}
\\&+2\sum_{k\in\mathbb Z^3}\sum_{i\sim j; i'\sim j'}\sum_{k_{1234}=k}\theta(2^{-i}|k_{12}|)\theta(2^{-j}|k_{34}|)\theta(2^{-i'}|k_{14}|)\theta(2^{-j'}|k_{23}|)\theta(2^{-q}|k|)^2\prod_{l=1}^{4}\frac{|f(\eps k_l)|^2}{|k_l|^2}
\\&\times\int_0^t\int_0^t\dd s\dd\sigma e^{-(|k_{12}|^2+|k_2|^2)|t-s|-(|k_{14}|^2+|k_4|^2)|t-\sigma|-|k_1|^2|s-\sigma|} 
\\&\equiv\sum_{j=1}^{3} I_{1,j}^\eps(t).
\end{split}
\end{equation*}}
Let us begin by treating the first term. As usual by symmetry we have 
\begin{equation*}
\begin{split}
I_{1,1}^{\eps}(t)&\lesssim\sum_{k\in\mathbb Z^3}\sum_{q\lesssim i}\sum_{\substack{k_{1234}=k\\|k_1|\leq|k_2|,|k_3|\leq|k_4|\\\max_{l=1,\dots,4}|k_l|=|k_2|}}\theta(2^{-q}|k|)\theta(2^{-i}|k_{12}|)\prod_{i=l}^{4}|k_l|^{-2}\int_0^t\int_0^t\dd s\dd \sigma e^{-|k_{12}|^2(|t-s|+|t-\sigma|)-|k_2|^2|s-\sigma|}
\\&+\sum_{k\in\mathbb Z^3}\sum_{q\lesssim i}\sum_{\substack{k_{1234}=k\\|k_1|\leq|k_2|,|k_3|\leq|k_4|\\\max_{l=1,\dots,4}|k_l|=|k_4|}}\theta(2^{-q}|k|)\theta(2^{-i}|k_{12}|)\prod_{i=l}^{4}|k_l|^{-2}\int_0^t\int_0^t\dd s\dd \sigma e^{-|k_{12}|^2(|t-s|+|t-\sigma|)-|k_2|^2|s-\sigma|}
\\&\equiv A_{1}^\eps(t)+A_{2}^\eps(t).
\end{split}
\end{equation*}
We notice that if $\max_{l=1,. . .,4}|k_l|=|k_1|$ then $|k|\lesssim|k_1|$, and 
$$
A_1^\eps(t)\lesssim\sum_{k\in\mathbb Z^3}|k|^{-1+2\eta}\theta(2^{-q}|k|)\sum_{\substack{k_{1234}=k\\|k_1|\leq|k_2|,|k_3|\leq|k_4|\\\max_{l=1,\dots,4}|k_l|=|k_2|}}|k_1|^{-3-\eta/3}|k_3|^{-3-\eta/3}|k_4|^{-3-\eta/3}\sum_{q\lesssim i}2^{-i(2-\eta)}\lesssim  t^\eta2^{3q\eta},
$$
where we have used that
$$
\int_0^t\int_0^t\dd s\dd \sigma e^{-|k_{12}|^2(|t-s|+|s-\sigma|)-|k_2|^2|s-\sigma|}\lesssim t^\eta\frac{1}{|k_2|^{2-\eta}|k_{12}|^{2-\eta}}.
$$
By a similar argument we have
$$
A_2^\eps(t)\lesssim\sum_{k\in\mathbb Z^3}|k|^{-1+4\eta}\theta(2^{-q}|k|)\sum_{\substack{k_{1234}=k\\|k_1|\leq|k_2|,|k_3|\leq|k_4|\\\max_{l=1,\dots,4}|k_l|=|k_4|}}|k_1|^{-3-\eta}|k_2|^{-3-\eta}|k_3|^{-3-\eta}\sum_{q\lesssim i}2^{-i(2-\eta)}\lesssim t^\eta 2^{5q\eta} 
$$
and then $\sup_{\eps}I_{1,1}(t)\lesssim t^{\eta}2^{5q\eta}$. Let us treat the second term $I_{1,2}^\eps(t)$. We have 
{\small
\begin{equation*}
\begin{split}
I_{1,2}^{\eps}(t)\lesssim& \sum_{k\in\mathbb Z^3}\sum_{q\lesssim i\sim j}\sum_{\substack{k_{1234}=k\\|k_1|\leq|k_2|,|k_3|\leq|k_4|\\\max_{l=1,\dots,4}|k_l|=|k_2|}}\theta(2^{-q}|k|)\theta(2^{-i}|k_{12}|)\theta(2^{-j}|k_{34}|)\prod_{i=l}^{4}|k_l|^{-2}
\\&\times\int_0^t\int_0^t\dd s\dd\sigma e^{-(|k_{12}|^2+|k_1|^2+|k_2|^2)|t-s|-(|k_{34}|^2+|k_3|^2+|k_4|^2)|s-\sigma|}
\\ \lesssim&\sum_{k\in\mathbb Z^3}|k|^{-1+4\eta}\sum_{q\lesssim i\sim j}\sum_{\substack{k_{1234}=k\\|k_1|\leq|k_2|,|k_3|\leq|k_4|\\\max_{l=1,. . .4}|k_l|=|k_2|}}\theta(2^{-q}|k|)\theta(2^{-i}|k_{12}|)\theta(2^{-j}|k_{34}|)\frac{1}{|k_1|^2|k_2|^{3+3\eta}|k_3|^2|k_4|^2|k_{34}|^{2-\eta}}
\\ \lesssim& t^\eta 2^{q(2+4\eta)}\sum_{q\lesssim j}2^{-j(2-\eta)}\sum_{l}|l|^{-3-\eta}\lesssim t^\eta 2^{5q\eta}.
\end{split}
\end{equation*}
}
We have to treat the last term in the fourth chaos. A similar computation gives 
\small{ \begin{equation*}
\begin{split}
I_{1,3}^{\eps}(t)
\lesssim&
	\sum_{k\in\mathbb Z^3}\sum_{q\lesssim i\sim j;q\lesssim i'\sim j'}\sum_{k_{1234}=k}\theta(2^{-i}|k_{12}|)\theta(2^{-j}|k_{34}|)\theta(2^{-i'}|k_{14}|)\theta(2^{-j'}|k_{23}|)\theta(2^{-q}|k|)^2
\\
&
\times\int_0^t\int_0^t\dd s\dd\sigma e^{-(|k_{12}|^2+|k_2|^2)|t-s|-(|k_{14}|^2+|k_4|^2)|t-\sigma|}
\\
\lesssim&
	\sum_{k\in\mathbb Z^3}\sum_{q\lesssim i'\sim j'}\sum_{\substack{k_{1234}=k\\|k_4|\leq|k_2|,|k_1|\leq|k_3|}}\theta(2^{-i}|k_{12}|)\theta(2^{-j}|k_{34}|)\theta(2^{-i'}|k_{14}|\theta(2^{-q}|k|)^2\prod_{l=1}^{4}\frac{1}{|k_l|^2}
\\
&
\times\int_0^t\int_0^t\dd s\dd\sigma e^{-|k_2|^2|t-s|-|k_{14}|^2|t-\sigma|}
\\
\lesssim&
	t^\eta 2^{-q(2-\eta)} \sum_{k\in\mathbb Z^3}\theta(2^{-q}|k|)\sum_{\substack{k_{1234}=k\\|k_4|\leq|k_2|,|k_1|\leq|k_3|}}\frac{1}{|k_1|^2|k_2|^{4-\eta}|k_3|^2|k_4|^2}.
\end{split}
\end{equation*}}
We still need to bound  the sum 
$$
\sum_{\substack{k_{1234}=k\\|k_4|\leq|k_2|,|k_1|\leq|k_3|}}\frac{1}{|k_1|^2|k_2|^{4-\eta}|k_3|^2|k_4|^2}.
$$
For that, we notice that when $|k_3|\leq|k_2|$ we can use the bound 
$$
\frac{1}{|k_1|^2|k_2|^{4-\eta}|k_3|^2|k_4|^2}\lesssim|k|^{-1+4\eta}|k_1|^{-3-\eta}|k_3|^{-3-\eta}|k_4|^{-3-\eta}
$$ 
and when $|k_2|\leq|k_3|$   we can use that 
$$
\frac{1}{|k_1|^2|k_2|^{4-\eta}|k_3|^2|k_4|^2}\lesssim|k|^{-1+4\eta}|k_1|^{-2}|k_2|^{-4+\eta}|k_3|^{-1+4\eta}|k_4|^{-2}\lesssim|k|^{-1+4\eta}|k_1|^{-3-\eta}|k_2|^{-3-\eta}|k_4|^{-3-\eta},
$$
where we have used that $|k_4|\leq|k_2|$. We can conclude that  $\sup_{\eps}I_{1,3}^{\eps}(t)\lesssim t^\eta2^{5q\eta}$. This gives the needed bound for the term lying in the chaos of order four. In fact, we have 
$$
\sup_{\eps}\mathbb E\left[\Delta_q|I^\eps_1(t)|^2\right]\lesssim  t^\eta2^{5q\eta}.
$$
Let us focus on the term lying in the second chaos.
{\small \begin{equation*}
\begin{split}
\mathbb E\left[|\Delta_qI_2^\eps(t)|^2\right]&=\sum_{k\in\mathbb Z^3}\sum_{q\lesssim i\sim j,q\lesssim i'\sim j'}\sum_{k_{13}=k,k_2,k_4}\theta(2^{-i}|k_{12}|)\theta(2^{j}|k_{2(-3)}|)\theta(2^{-i'}|k_{14}|)\theta(2^{-j'}|k_{4(-3)}|)
\\&\times\prod_{i=1}^{4}\frac{|f(\eps k_l)|^2}{|k_l|^2}\int_0^t\int_0^t\dd s\dd \sigma e^{-(|k_{12}|^2+|k_2|^2)|t-s|-(|k_{14}|^2+|k_4|^2)|t-\sigma|-|k_1|^2|s-\sigma|} 
\\&+\sum_{k\in\mathbb Z^3}\sum_{q\lesssim i\sim j,q\lesssim i'\sim j'}\sum_{k_{13}=k,k_2,k_4}\theta(2^{-i}|k_{12}|)\theta(2^{j}|k_{2(-3)}|)\theta(2^{-i'}|k_{34}|)\theta(2^{-j'}|k_{4(-3)}|)
\\&\times\prod_{i=1}^{4}\frac{|f(\eps k_l)|^2}{|k_l|^2}\int_0^t\int_0^t\dd s\dd \sigma e^{-(|k_{12}|^2+|k_2|^2+|k_1|^2)|t-s|-(|k_{34}|^2+|k_4|^2+|k_3|^2)|t-\sigma|}
\\&\equiv I_{2,1}^\eps(t)+I_{2,2}^\eps(t).
\end{split}
\end{equation*}}
We treat these two terms separately.  In fact, by symmetry, we have
{\small \begin{equation*}
\begin{split}
I_{2,1}^{\eps}(t)&\lesssim\sum_{k\in\mathbb Z^3}\sum_{q\lesssim i\sim j;q\lesssim i'\sim j'}\sum_{\substack{k_{13}=k\\k_2,k_4,|k_1|\leq|k_3|}}\theta(2^{-q}|k|)^2\theta(2^{-i}|k_{12}|)\theta(2^{j}|k_{2(-3)}|)\theta(2^{-i'}|k_{14}|)\theta(2^{-j'}|k_{4(-3)}|)
\\&\times\prod_{i=1}^{4}\frac{|f(\eps k_l)|^2}{|k_l|^2}\int_0^t\int_0^t\dd s\dd \sigma e^{-(|k_{12}|^2+|k_2|^2)|t-s|-(|k_{14}|^2+|k_4|^2)|t-\sigma|} 
\\&\lesssim t^{\eta}\sum_{|k|\sim q}|k|^{-1+\eta}\sum_{q\lesssim i,i'}\theta(2^{-i}|k_{12}|)\theta(2^{-i'}|k_{14}|)\sum_{k_1,k_2,k_4}|k_1|^{-3-\eta}|k_2|^{-3-\eta}|k_3|^{-3-\eta}|k_{12}|^{-1+2\eta}|k_{14}|^{-1+2\eta}
\\&\lesssim2^{q(2+\eta)}\sum_{q\lesssim i,i'}2^{-(i+i')(1-2\eta)}\lesssim t^\eta2^{3q\eta},
\end{split}
\end{equation*}}
which gives the first bound. The second term has a  similar bound, indeed 
\begin{equation*}
\begin{split}
I_{2,2}^{\eps}(t)&\lesssim \sum_{k\in\mathbb Z^3}\sum_{q\lesssim i,q\lesssim i'}\sum_{k_{13}=k,k_2,k_4,|k_1|\leq|k_3|}\theta(2^{-i}|k_{12}|)\theta(2^{-i'}|k_{34}|)\prod_{i=1}^{4}\frac{|f(\eps k_l)|^2}{|k_l|^2}
\\&\times\int_0^t\int_0^t\dd s\dd \sigma e^{-(|k_{12}|^2+|k_2|^2)|t-s|-(|k_{34}|^2+|k_4|^2)|t-\sigma|}\lesssim t^\eta 2^{3q\eta},
\end{split}
\end{equation*}
which ends the proof.
\end{proof} 

\subsection{Renormalization for \texorpdfstring{$I( X^{\diamond 3} )\circ X^{\diamond 2} $}{I(X3)X2}}
\label{subsec:Renor-I(X3)X2}
Here again we only give the crucial bound, but for $I(X^{\diamond3})\diamond X^{\diamond2}$ instead of $(I(X^{\diamond3})\circ_{\diamond} X^{\diamond2})$.
\begin{proposition}
  For all $T>0$, $t\in[0,T]$, $ \delta, \delta' > 0$ and all $ 1\gg \nu > 0$ small enough, there exists two constants and $C >
  0$ depending on $T$, $\delta, \delta'$ and $\nu$ such that for all $q \ge - 1$,
  \[ \mathbb E [ t^{\delta'+\delta} | \Delta_q ( I (  ( X^{\varepsilon}_t)^{ \diamond 3} )  (
     X_t^{\varepsilon})^{\diamond 2}  - 3 C^{\varepsilon}_2 X_t^{\varepsilon}) |^2]
     \le C t^{\delta} 2^{q ( 1 + \nu)}. \]
\end{proposition}

\begin{proof}
Thanks to a straightforward computation we have
 $$
   -I ( ( X^{\varepsilon}_t)^{\diamond 3}) ( X^{\varepsilon}_t)^{\diamond 2} = I^{( 1)}_t +
     I^{( 2)}_t + I^{( 3)}_t, 
$$
  where
$$
   I^{(1)}_t = \sum_{k \neq 0} e_k \sum_{\tmscript{ \begin{array}{c}
       {}k_{12345} = k\\
       k_i \neq 0
     \end{array}}} \int_0^t \mathd s e^{- | k_1 + k_2 + k_3 |^2 | t - s |} :
     \hat{X}^{\varepsilon}_s ( k_1) \hat{X}^{\varepsilon}_s ( k_2)
     \hat{X}^{\varepsilon}_s ( k_3) \hat{X}^{\varepsilon}_t ( k_4)
     \hat{X}^{\varepsilon}_t ( k_5) :, 
     $$
     $$
   I^{(2)}_t = 6 \sum_{k \neq 0} e_k \sum_{\tmscript{\begin{array}{c}
       k_3, k_{124} = k\\
       k_i \neq 0
     \end{array}}} \int_0^t \mathd s e^{- | k_1 + k_2 + k_3 |^2 | t - s |}
     \frac{e^{- | k_3 |^2 | t - s |}}{| k_3 |^2} f ( \varepsilon k_3)^2 :
     \hat{X}^{\varepsilon}_s ( k_1) \hat{X}^{\varepsilon}_s ( k_2)
     \hat{X}^{\varepsilon}_t ( k_4) : 
  $$
  and
   $$ 
 I^{(3)}_t = 6 \sum_{k \neq 0} e_k \int_0^t \mathd s \sum_{k_1 \neq 0, k_2
     \neq 0} \frac{f ( \varepsilon k_1)^2 f ( \varepsilon k_2)^2}{| k_1 |^2 |
     k_2 |^2} e^{- ( | k + k_1 + k_2 |^2 + | k_1 |^2 + | k_2 |^2) | t - s |}
     \hat{X}_s^{\varepsilon} ( k). 
$$
  Hence, 
  \begin{multline*}
   -\left(I ( ( X^{\varepsilon}_t)^{\diamond 3}) ( X^{\varepsilon}_t)^{\diamond 2} - 3
     C^{\varepsilon}_2 X_t^{\varepsilon} \right)= I ( ( X^{\varepsilon}_t)^{\diamond 3})  (
     X_t^{\varepsilon})^{\diamond 2} I^{(3)}_t\\ + ( I^{(3)}_t - \tilde{I}_t^{(3)}) 
     + (
     \tilde{I}_t^{(3)} - 3 \tilde{C}^{\varepsilon}_2(t) X_t^{\varepsilon}) + 3(C^\eps_2-\tilde{C}_2^\eps(t))X_t^\eps,
   \end{multline*}
  where we remind that
  $$
  C^\eps_2 = \sum_{k_1\neq0,k_2\neq0}\frac{f(\eps k_1)f(\eps k_2)}{|k_1|^2|k_2|^2(|k_1|^2+|k_2|^2+|k_1+k_2|^2)}
  $$
  and where we have defined
  $$
   \tilde{I}_t^{(3)} = 6 \sum_{k \neq 0} e_k \int_0^t \mathd s \sum_{k_1 \neq 0,
     k_2 \neq 0} \frac{f ( \varepsilon k_1)^2 f ( \varepsilon k_2)^2}{| k_1
     |^2 | k_2 |^2} e^{- ( | k + k_1 + k_2 |^2 + | k_1 |^2 + | k_2 |^2) | t -
     s |} \hat{X}_t^{\varepsilon} ( k) 
  $$
  and
   $$
   \tilde {C}^{\varepsilon}_2 = 2 \int_0^t \mathd s \sum_{k_1 \neq 0, k_2 \neq 0}
     \frac{f ( \varepsilon k_1)^2 f ( \varepsilon k_2)^2}{| k_1 |^2 | k_2 |^2}
     e^{- ( | k_1 + k_2 |^2 + | k_1 |^2 + | k_2 |^2) | t - s |}.
 $$
  Hence for $q \ge - 1$,
  \begin{multline*}
    \mathbb E [ | \Delta_q ( I ( ( X^{\varepsilon}_t)^{\diamond 3})  (
    X_t^{\varepsilon})^{\diamond 2}  - 3C^{\varepsilon}_2 X_t^{\varepsilon}) |^2]
    \lesssim 
    		\mathbb E [ | \Delta_q ( I^{(1)}_t) |^2]
    + \mathbb E [ | \Delta_q ( I^{(2)}_t) |^2]
      + \mathbb E [ | \Delta_q ( I^{(3)}_t - \tilde{I}_t^{(3)}) |^2]\\
     +\mathbb E [ | \Delta_q ( \tilde{I}_t^{(3)} - \tilde {C}^{\varepsilon}_2(t)
    X_t^{\varepsilon}) |^2]
     +|C^\eps_2-\tilde{C}^\eps_2(t)|^2\expect[|\Delta_q X^\eps_t|^2].
  \end{multline*}
  \paragraph{Terms in the first chaos.}
  
  Let us first deal with the "deterministic" part, here $C^\eps_2 -\tilde{C}^\eps_2(t)$. An obvious computation gives for all $\delta'>0$,
  $|C^\eps_2 -\tilde{C}^\eps_2(t)|^2\lesssim_{\delta'}1/t^{\delta'}$.
Furthermore,
$\expect[|\Delta_q X^\eps_t|2]\lesssim 2^q$,
hence for all $\delta'>0$,
\[ |C^\eps_2-\tilde{C}^\eps_2(t)|^2\expect[|\Delta_q X^\eps_t|^2] \lesssim 2^q /t^{\delta'}. \]
  Let us deal with $\mathbb E [ | \Delta_q ( I^{(3)}_t - \tilde{I}_t^{(3)})
  |^2]$. For $k \neq 0$ we define
  $$
   a_k ( t - s) = \sum_{k_1 \neq 0, k_2 \neq 0} \frac{f ( \varepsilon
     k_1)^2 f ( \varepsilon k_2)^2}{| k_1 |^2 | k_2 |^2} e^{- ( | k + k_1 +
     k_2 |^2 + | k_1 |^2 + | k_2 |^2) | t - s |}, 
  $$
 such that
 \begin{eqnarray*}
\lefteqn{    \mathbb E [ | \Delta_q ( I_t^{( 3)} - \tilde{I}_t^{( 3)}) |^2]}\\
 & = &
    \mathbb E \left[ \left| \int_0^t \sum_k \theta ( 2^{- q} k) e_k a_k ( t
    - s) ( \hat{X}_s^{\varepsilon} ( k) - \hat{X}_t^{\varepsilon} ( k))
    \right|^2 \right]\\
    & = & \int_{[ 0, t]^2} \mathd \overline{s} \mathd s
    \sum_{\tmscript{\begin{array}{c}
      k \neq 0\\
      \overline{k} \neq 0
    \end{array}}} e_k e_{\overline{k}} \theta ( 2^{- q} k) \theta ( 2^{- q}
    \overline{k}) \\
    && \qquad \qquad\qquad\times a_k ( t - s) a_{\overline{k}} ( t - \overline{s})
    \mathbb E [ ( \hat{X}_s^{\varepsilon} ( k) - \hat{X}_t^{\varepsilon} (
    k)) ( \hat{X}_{\overline{s}}^{\varepsilon} ( \overline{k}) -
    \hat{X}_t^{\varepsilon} ( \overline{k}))].
  \end{eqnarray*}
  But
  \begin{eqnarray*}
\mathbb E [ ( \hat{X}_s^{\varepsilon} ( k) - \hat{X}_t^{\varepsilon} (
    k)) ( \hat{X}_{\overline{s}}^{\varepsilon} ( \overline{k}) -
    \hat{X}_t^{\varepsilon} ( \overline{k}))] 
    & = & \delta_{k = -
    \overline{k}} \frac{f ( \varepsilon k)^2}{| k |^2} ( e^{- | s -
    \overline{s} |  | k |^2} - e^{- | t - \overline{s} |  | k |^2} - e^{- | t
    - s |  | k |^2} + 1)\\
     & \lesssim & \delta_{k = -
    \overline{k}} \frac{f ( \varepsilon k)^2}{| k |^2}| k |^{2 \eta} | t - s |^{\eta / 2} | t - \overline{s}
    |^{\eta / 2}.
  \end{eqnarray*}
    Hence
  \begin{eqnarray*}
    \mathbb E [ | \Delta_q ( I_t^{( 3)} - \tilde{I}_t^{( 3)}) |^2] &
    \lesssim & \sum_{k \neq 0} \theta ( 2^{- q} k)^2 \frac{f ( \varepsilon
    k)^2}{| k |^{2 ( 1 - \eta)}} \left( \int_0^t \mathd s | t - s |^{\eta / 2}
    a_k ( | t - s |) \right)^2
  \end{eqnarray*}
  and
  \begin{eqnarray*}
    \int_0^t \mathd s | t - s |^{\eta / 2} a_k ( | t - s |) & = &
    \sum_{\tmscript{\begin{array}{c}
      k_1 \neq 0\\
      k_2 \neq 0
    \end{array}}} \int_0^t \mathd s | t - s |^{\eta / 2} e^{- ( | k + k_1 +
    k_2 | ^2 + | k_1 |^2 + | k_2 |^2) | t - s |} \frac{f ( \varepsilon k_1)^2
    f ( \varepsilon k_1)^2}{| k_1 |^2 | k_2 |^2}\\
    & \lesssim & \sum_{\tmscript{\begin{array}{c}
      k_1 \neq 0\\
      k_2 \neq 0
    \end{array}}} | k_1 |^{- 3 - \eta'} | k_2 |^{- 3 - \eta'} \int_0^t \mathd
    s | t - s |^{- 1 + ( \eta / 2 - \eta')}  \lesssim  t^{\eta / 2 - \eta'},
  \end{eqnarray*}
  for $\eta / 2 - \eta' > 0$. Hence we have
  $$ 
  \mathbb E [ | \Delta_q ( I_t^{( 3)} - \tilde{I}_t^{( 3)}) |^2] \lesssim
     2^{q ( 1 + 2 \eta)} t^{\eta - 2 \eta'} . 
   $$
 We have furthermore
  $$
   \mathbb E [ | \Delta_q ( \tilde{I}_t^{( 3)} - C^{\varepsilon}_2
     X_t^{\varepsilon}) |^2] = \sum_{k \neq 0} \frac{f ( \varepsilon k)^2}{| k
     |^2}\theta(2^{-q}k)^2 b_k ( t)^2, 
  $$
  with
  $$ b_k ( t) = \int_0^t \sum_{\tmscript{\begin{array}{c}
       k_1 \neq 0\\
       k_2 \neq 0
     \end{array}}} \frac{f ( \varepsilon k_1)^2 f ( \varepsilon k_2)^2}{| k_1
     |^2 | k_2 |^2} e^{- ( | k_1 |^2 + | k_2 |^2) | t - s |} \{ e^{- | k_1 +
     k_2 |^2 | t - s |} - e^{- | k_1 + k_2 + k |^2 | t - s |} \} .
  $$
Using that
 $$
  | e^{- | k_1 + k_2 + k |^2 | t - s |} - e^{- | k_1 + k_2 |^2 | t - s |}
     | \lesssim | t - s |^{\eta} | k |^{\eta} ( | k | + \max \{ | k_1 |, | k_2
     | \})^{\eta} 
 $$
  we have the following bound
 $$
  b_k ( t) \lesssim \int_0^t \sum_{\tmscript{\begin{array}{c}
       k_1 \neq 0\\
       k_2 \neq 0
     \end{array}}} | k_1 |^{- 3 - \eta'} | k_2 |^{- 3 - \eta''} | k |^{\eta} (
     | k | + \max \{ | k_1 |, | k_2 | \})^{\eta} | t - s |^{- 1 + ( \eta -
     \eta' / 2 - \eta'' / 2)} .
 $$
  We can suppose that $\max \{ | k_1 |, | k_2 | \} = | k_1 |$ as the
  expression is symmetric in $k_1, k_2$, then if $| k | > | k_1 |$,
  $$
  b_k ( t) \lesssim t^{( \eta - \eta' / 2 - \eta'' / 2)} | k |^{2 \eta}, 
  $$
  for $\eta - \eta' / 2 - \eta'' / 2 > 0$. Furthermore if $| k_1 | > | k |$,
  and $\eta' > \eta$ then
  $$
   b_k ( t) \lesssim t^{( \eta - \eta' / 2 - \eta'' / 2)} | k |^{\eta}
     \sum_{\tmscript{\begin{array}{c}
       k_1 \neq 0\\
       k_2 \neq 0
     \end{array}}} | k_1 |^{- 3 - ( \eta' - \eta)} | k_2 |^{- 3 - \eta''}
     \lesssim t^{( \eta - \eta' / 2 - \eta'' / 2)} | k |^{\eta}. 
  $$
  Hence, there exists $\delta > 0$ and $\nu > 0$ such that
 $$
  \mathbb E [ | \Delta_q ( \tilde{I}_t^{( 3)} - 3 C^{\varepsilon}_2
     X_t^{\varepsilon}) |^2] \lesssim t^{\delta} 2^{( 1 + \nu) q} . 
  $$  
  
  \paragraph{Terms in the third chaos.}
  Let us define
  $ 
  c_{k_1, k_2} ( t - s) = \sum_{k_3 \neq 0} \frac{f ( \varepsilon
     k_3)^2}{| k_3 |^2} e^{- ( | k_1 + k_2 + k_3 |^2 + | k_3 |^2) | t - s |}
  $
  such that
  $$
   I^{( 2)}_t = 6 \sum_{\tmscript{\begin{array}{c}
       k \neq 0, k_i \neq 0\\
       k_{124} = k
     \end{array}}} e_k \int_0^t \mathd s c_{k_1, k_2} ( t - s) : \hat{X}_s (
     k_1) \hat{X}_s ( k_2) \hat{X}_t^{\varepsilon} ( k_4) : \, .
  $$
  But for all suitable variables we have 
  \begin{eqnarray*}
    \lefteqn{\mathbb E [ : \hat{X}^{\varepsilon}_s ( k_1) \hat{X}^{\varepsilon}_s (
    k_2) \hat{X}^{\varepsilon}_t ( k_4) : :
    \hat{X}^{\varepsilon}_{\overline{s}} ( \overline{k}_1)
    \hat{X}^{\varepsilon}_{\overline{s}} ( \overline{k}_2)
    \hat{X}^{\varepsilon}_t ( \overline{k}_4) :]}\\
     & = & 2 \delta_{k_1 = -
    \overline{k}_1} \frac{f ( \varepsilon k_1)^2}{| k_1 |^2} \delta_{k_2 = -
    \overline{k}_3} \frac{f ( \varepsilon k_2)^2}{| k_2 |^2} \delta_{k_3 = -
    \overline{k}_3} \frac{f ( \varepsilon k_3)^2}{| k_3 |^2} e^{- ( | k_1 |^2
    + | k_2 |^2) | s - \overline{s} |}\\
    &  & + 2 \delta_{k_1 = - \overline{k}_1} \frac{f ( \varepsilon k_1)^2}{|
    k_1 |^2} \delta_{k_2 = - \overline{k}_3} \frac{f ( \varepsilon k_2)^2}{|
    k_2 |^2} \delta_{k_3 = - \overline{k}_2} \frac{f ( \varepsilon k_3)^2}{|
    k_3 |^2} e^{- | k_1 |^2 | s - \overline{s} |} e^{- ( | k_3 |^2) | t -
    \overline{s} |} e^{- ( | k_2 |^2) | t - s |}\\
    &  &\ \ \times e^{- | k_1 |^2 | s - \overline{s} |} e^{- ( | k_3 |^2) | t -
    \overline{s} |}.
  \end{eqnarray*}
  By another easy computation the following holds
  $
  \mathbb E [ | \Delta_q ( I^{( 2)}_t) |^2] = E^{2, 1}_t + E^{2, 2}_t 
  $
  where
  $$ 
  E^{2, 1}_t = 2 \int_0^t \mathd s \int_0^s \mathd \overline{s}
     \sum_{\tmscript{\begin{array}{c}
       k, k_i \neq 0 \overline{}\\
       {}k_{124} = k
     \end{array}}} \theta ( 2^{- q} k)^2 \prod_i \frac{f ( \varepsilon
     k_i)^2}{| k_i |^2} c_{k_1, k_2} ( t - s) c_{k_1, k_2} ( t - \overline{s})
     e^{- ( | k_1 |^2 + | k_2 |^2) | s - \overline{s} |} 
  $$
  and
   \begin{multline*}E^{2, 1}_t = 2 \int_0^t \mathd s \int_0^s \mathd \overline{s}
     \sum_{\tmscript{\begin{array}{c}
       k \neq 0 \overline{}, k_i \neq 0,\\
       {}k_{124} = k
     \end{array}}} \theta ( 2^{- q} k)^2 \prod_i \frac{f ( \varepsilon
     k_i)^2}{| k_i |^2}\\ 
     \times c_{k_1, k_2} ( t - s) c_{k_1, k_4} ( t - \overline{s})
     e^{- | k_1 |^2 | s - \overline{s} |} e^{- | k_4 |^2 | t - \overline{s} |}
     e^{- | k_2 |^2 | t - s |} . \end{multline*}

  In $E^{2, 1}_t$, we have a symmetry in $k_1, k_2$, hence we can assume that
  $| k_1 | \ge | k_2 |$. Furthermore, we have
  $
  c_{k_1, k_2} ( t - s) \lesssim | t - s |^{- \frac{1 + \eta}{2}}
  $ 
  and
  $
   c_{k_1, k_2} ( t - \overline{s}) \lesssim | s - \overline{s} |^{- \frac{1
     + \eta}{2}} 
 $.
If we assume that $| k_1 | \ge | k_4 |$ and that $\eta' / 2 -
  \eta > 0$, we have
  \begin{eqnarray*}
    E^{2, 1}_t & \lesssim & \int_0^t \mathd s \int_0^s \mathd \overline{s} | t
    - s |^{- \frac{1 + \eta}{2}} | s - \overline{s} |^{- 1 + ( \eta' / 2 -
    \eta)} \sum_{\tmscript{\begin{array}{c}
      k \neq 0 \overline{}, k_i \neq 0,\\
      {}k_{124} = k
    \end{array}}} \theta ( 2^{- q} k)^2 \frac{1}{| k_1 |^{3 - \eta'} | k_2 |^2
    | k_4 |^2}\\
    & \lesssim & t^{\delta} \sum_{k \neq 0} \frac{\theta ( 2^{- q} k)^2}{| k
    |^{1 - \eta''}} \sum_{k_2, k_3} | k_2 |^{- 3 - \frac{\eta'' - \eta'}{2}} |
    k_4 |^{- 3 - \frac{\eta'' - \eta'}{2}}
     \lesssim  t^{\delta} 2^{q ( 2 + \eta'')},
  \end{eqnarray*}
  for $\eta'' > \eta'$. When $| k_4 | \ge | k_1 |$ it is pretty much the same
  computation.
  
  In $E^{2, 2}_t$, we can assume that $| k_2 | \ge | k_4 |$, hence
  \begin{eqnarray*}
    E^{2, 2}_t & \lesssim & \int_0^t \mathd s \int_0^t \mathd \overline{s}
    \sum_{\tmscript{\begin{array}{c}
      k \neq 0 \overline{}, k_i \neq 0,\\
      {}k_{124} = k\\
      | k_2 | \lesssim | k_4 |
    \end{array}}} \theta ( 2^{- q} k)^2 | k_1 |^{- 3 + \eta'} | k_2 |^{- 3 +
    \eta'} | k_4 |^2 | t - s |^{- 1 + \frac{\eta' - \eta}{2}} | s -
    \overline{s} |^{- 1 + \frac{\eta' - \eta}{2}}\\
    & \lesssim & t^{\delta} \sum_{\tmscript{\begin{array}{c}
      k \neq 0 \overline{}, k_i \neq 0,\\
      {}k_{124} = k
    \end{array}}} \theta ( 2^{- q} k)^2 | k |^{- 1 + \eta''} | k_1 |^{- 3 +
    \eta'} | k_2 |^{- 3 + \eta'} | k_4 |^2 \max ( | k_i |)^{1 - \eta''}
     \lesssim  t^{\delta} 2^{q ( 1 + \eta'')}.
  \end{eqnarray*}
  Finally by decomposing the previous expression depending on $|k_1|\ge |k_4|$ or $|k_4|\ge |k_1|$ we have the wanted bound.
  
  \paragraph{Terms in the fifth chaos.}
  
  For all suitable variables, we have
  {\small \begin{eqnarray*}
    \lefteqn{\mathbb E [ : \hat{X}^{\varepsilon}_s ( k_1) \hat{X}^{\varepsilon}_s (
    k_2) \hat{X}^{\varepsilon}_s ( k_3) \hat{X}^{\varepsilon}_t ( k_4)
    \hat{X}^{\varepsilon}_t ( k_5) : : \hat{X}^{\varepsilon}_{\overline{s}} (
    \overline{k}_1) \hat{X}^{\varepsilon}_{\overline{s}} ( \overline{k}_2)
    \hat{X}^{\varepsilon}_{\overline{s}} ( \overline{k}_3)
    \hat{X}^{\varepsilon}_t ( \overline{k}_4) \hat{X}^{\varepsilon}_t (
    \overline{k}_5) :]}\\
    & = & 12 \prod_{i = 1}^5 \frac{f ( \varepsilon
    k_i)^2}{| k_i |^2} \delta_{k_i = - \overline{k}_i} e^{- | s - \overline{s}
    | ( | k_1 |^2 + | k_2 |^2 + | k_3 |^2)}\\
    &  & + 72 \prod_{i = 1}^5 \frac{f ( \varepsilon k_i)^2}{| k_i |^2}
    \delta_{k_1 = - \overline{k}_1} \delta_{k_2 = - \overline{k}_2}
    \delta_{k_3 = - \overline{k}_4} \delta_{k_4 = - \overline{k}_3}
    \delta_{k_5 = - \overline{k}_5} e^{- | s - \overline{s} | ( | k_1 |^2 + |
    k_2 |^2) - | t - s |  | k_3 |^2 - | t - \overline{s} |  | k_4 |^2} +\\
    &  & + 36 \prod_{i = 1}^5 \frac{f ( \varepsilon k_i)^2}{| k_i |^2}
    \delta_{k_1 = - \overline{k}_1} \delta_{k_2 = - \overline{k}_4}
    \delta_{k_3 = - \overline{k}_5} \delta_{k_4 = - \overline{k}_3}
    \delta_{k_5 = - \overline{k}_2} e^{- | s - \overline{s} | | k_1 |^2 - | t
    - s |  ( | k_2 |^2 + | k_3 |^2) - | t - \overline{s} |  ( | k_4 |^2 + |
    k_5 |^2)}.
  \end{eqnarray*}
  }
  Then
   $$
   \mathbb E [ | \Delta_q I^1_t |^2] = E^{1, 1}_t + E^{1, 2}_t + E^{1,
     3}_t, 
   $$
  where
  $$
  E^{1, 1}_t = 12 \int_{[ 0, t]^2} \mathd s \mathd \overline{s} \theta (
     2^{- q} k)^2 \sum_{\tmscript{\begin{array}{c}
       k\\
       \begin{array}{c}
         {}k_{12345} = k
       \end{array}
     \end{array}}} \prod_{i = 1}^5 \frac{f ( \varepsilon k_i)^2}{| k_i |^2}
     e^{- | k_{123} |^2 | t - s |} e^{- ( | k_1 |^2 + | k_2 |^2 + | k_3
     |^2) | s - \overline{s} |},
  $$
  \begin{multline*} E^{1, 2}_t = 72 \int_{[ 0, t]^2} \mathd s \mathd \overline{s}
     \sum_{\tmscript{\begin{array}{c}
       k\\
       {}k_{12345} = k
     \end{array}}} \theta ( 2^{- q} k)^2 \prod_{i = 1}^5 \frac{f ( \varepsilon
     k_i)^2}{| k_i |^2}\\ 
     \times e^{- ( | k_{123} |^2 + | k_3 |^2) | t - s |} e^{- (
     | k_{124} |^2 + | k_4 |^2) | t - \overline{s} |} e^{- | s -
     \overline{s} | ( | k_1 |^2 + | k_2 |^2)} \end{multline*}
  and
  \begin{multline*}
    E^{1, 3}_t  =  36 \int_0^t \mathd s \int_0^t \mathd \overline{s}
    \sum_{\tmscript{\begin{array}{c}
      k \neq 0, k_i \neq 0\\
      \begin{array}{c}
        {}k_{12345} = k
      \end{array}
    \end{array}}} \theta ( 2^{- q} k)^2 \prod_{i = 1}^5 \frac{f ( \varepsilon
    k_i)^2}{| k_i |^2}\\
     \times e^{- ( | k_{123} |^2 + | k_2 |^2 + | k_3 |^2) | t - s |}
     e^{- ( | k_{145} |^2 + | k_5 |^2 + | k_4 |^2) | t - \overline{s}
    |} e^{- | s - \overline{s} | | k_1 |^2}.
  \end{multline*}

  \subparagraph{Estimation of $E^{1, 1}_t$.}
  
  Let us rewrite it in a form easier to handle:
 $$
   E^{1, 1}_t = 12 \int_{[ 0, t]^2} \mathd s \mathd \overline{s}
     \sum_{\tmscript{\begin{array}{c}
       k, k \neq 0\\
       k_1 + k_2 + l = k\\
       l_1 + l_2 + l_3 = l\\
       k_i \neq 0, l_i \neq 0
     \end{array}}} \theta ( 2^{- q} k)^2 \prod_{i = 1}^2 \frac{f ( \varepsilon
     k_i)^2}{| k_i |^2} \prod_{i = 1}^3 \frac{f ( \varepsilon l_i)^2}{| l_i
     |^2} e^{- | l |^2 | t - s |} e^{- ( | l_1 |^2 + | l_2 |^2 + | l_3 |^2) |
     s - \overline{s} |}. 
 $$   
  Thanks to the symmetries of this term, we can always assume that $| k_1 | =
  \max ( | k_i |)$ and $l_1 = \max ( | l_i |)$.
  
  For $l = 0$, we have
  \begin{align*}
&  \int_{[ 0, t]^2} \mathd s \mathd \overline{s}
     \sum_{\tmscript{\begin{array}{c}
       k, k \neq 0\\
       k_1 + k_2 = k\\
       l_1 + l_2 + l_3 = 0\\
       k_i \neq 0, l_i \neq 0
     \end{array}}} \theta ( 2^{- q} k)^2 \prod_{i = 1}^2 \frac{f ( \varepsilon
     k_i)^2}{| k_i |^2} \prod_{i = 1}^3 \frac{f ( \varepsilon l_i)^2}{| l_i
     |^2} e^{- ( | l_1 |^2 + | l_2 |^2 + | l_3 |^2) | s - \overline{s} |} 
 \\
   \lesssim & \int_{[ 0, t]^2} \mathd s \mathd \overline{s}
     \sum_{\tmscript{\begin{array}{c}
       k \neq 0
     \end{array}}} \theta ( 2^{- q} k)^2 | k |^{- 1 + \eta} \sum_{k_2 \neq 0}
     | k_{} |^{- 3 - \eta} \sum_{l_2 \neq 0, l_3 \neq 0} | l_2 |^{- 4 + \eta}
     | l_3 |^{- 4 + \eta} | s - \overline{s} |^{- 1 + \eta} 
\\
 \lesssim & 2^{q ( 2 + \eta)} t .
\end{align*}

  Let us assume that $| l | = \max ( | l |, | k_1 |)$. As we have the
  following estimate $| l_1 |^{- 1} \lesssim | l |^{- 1}$, the following bound
  holds:
  $$
  \int_{[ 0, t]^2} \mathd s \mathd \overline{s}
     \sum_{\tmscript{\begin{array}{c}
       k \neq 0
     \end{array}}} \theta ( 2^{- q} k)^2 | k |^{- 1 + \eta}
     \sum_{\tmscript{\begin{array}{c}
       k_1 k_2 \neq 0\\
       l_2, l_3 \neq 0
     \end{array}}} (| k_1 | |k_2|)^{- 4 + 9 \eta / 2} ( | t
     - s | | s - \overline{s} |)^{- 1 + \eta} | l_2 |^{- 3 - \eta} | l_3 |^{-
     3 - \eta} 
   $$  
   $$
  \lesssim t^{\eta} 2^{q ( 2 + \eta)}. 
  $$
  The case where $| k_1 | = \max ( | l |, | k_1 |)$ is similar, and the
  conclusion holds for $E^{1, 1}_t$.

  \subparagraph{Estimation of $E^{1, 2}_t$.}
  
  This term is symmetric in $k_1, k_2$ and in $k_3, k_4$. Hence, we can assume
  that $| k_1 | \ge | k_2 |$ and $| k_3 | \ge | k_4 |$ First let
  us assume that $| k_5 | = \max \{ | k_i | \}$. Then

  \begin{multline*} E^{1, 2}_t \lesssim \sum_{\tmscript{\begin{array}{c}
       k\\
       \begin{array}{c}
         {}k_{12345} = k
       \end{array}
     \end{array}}} \theta ( 2^{- q} k)^2 \int_0^t \mathd s \int_0^s \mathd
     \overline{s} ( | t - s | | s - \overline{s} |)^{- 1 + \eta} \\
     \times  |k_1 |^{- 4
     + 2 \eta} | k_2 |^{- 2} | k_3 |^{- 4 + 2 \eta} | k_4 |^{- 2} | k_5 |^{- (
     1 + \eta')} | k |^{- ( 1 - \eta')} 
     \end{multline*}
     \begin{align*}
  &\lesssim t^{\eta} \sum_k \theta ( 2^{- q} k)^2 | k |^{- ( 1 - \eta')}
     \sum_{\tmscript{\begin{array}{c}
       {}k_{12345} = k
     \end{array}}} | k_1 |^{- 7 / 2 + 2 \eta} | k_2 |^{- 3 - \eta' / 2} | k_3
     |^{- 7 / 2 + 2 \eta} | k_4 |^{- 3 - \eta' / 2} \\
  & \lesssim t^{\eta} 2^{( 2 + \eta') q},
\end{align*}
  for $\eta$ small enough.
  
  Then let us assume that $\max \{ | k_i | \} = | k_1 |$
 \begin{align*}
   E^{1, 2}_t 
   &\lesssim t^{\delta} \sum_{\tmscript{\begin{array}{c}
       k\\
       {}k_{12345} = k
     \end{array}}} \theta ( 2^{- q} k)^2 | k_1 |^{- 4 + 2 \eta} | k_2 |^{- 2}
     | k_3 |^{- 3 + \eta'} | k_4 |^{- 3 + \eta'} | k_5 |^{- 2} \\
     &\qquad \qquad \times \int_0^t \mathd
     s \int_0^s \mathd \overline{s} | t - s |^{- 1 + \eta'} | s - \overline{s}
     |^{- 1 + \eta} \\
& \lesssim t^{\eta'} \sum_{\tmscript{\begin{array}{c}
       k\\
       {}k_{12345} = k
     \end{array}}} \theta ( 2^{- q} k)^2 | k_{} |^{- 1 + \eta''} | k_2 |^{- 3
     - \eta''} \\
     & \qquad \qquad \qquad\times | k_3 |^{- 7 / 2 + ( 2 \eta + \eta'' + \eta') / 2} | k_4 |^{- 7
     / 2 + ( 2 \eta + \eta'' + \eta') / 2} | k_5 |^{- 3 - \eta''}  \\
  & \lesssim t^{\delta} 2^{( 2 + \eta') q}. 
\end{align*}
  For $\max \{ | k_i | \} = | k_3 |$ we have:
  \begin{align*}
 E^{1,2}_t  & \lesssim t^{\delta} \sum_{\tmscript{\begin{array}{c}
       k\\
       {}k_{12345} = k
     \end{array}}} \theta ( 2^{- q} k)^2 | k_1 |^{- 4 + \eta} | k_2 |^{- 2} |
     k_3 |^{- 4 + \eta'} | k_4 |^{- 2} | k_5 |^{- 2} 
  \\   
&  \lesssim t^{\delta} \sum_{\tmscript{\begin{array}{c}
       k\\
       {}k_{12345} = k
     \end{array}}} \theta ( 2^{- q} k)^2 | k_1 |^{- 3 + \eta + 1 / 4} | k_2
     |^{- 3 + 1 / 4} | k |^{- 1 + \eta'} | k_4 |^{- 3 + 1 / 4} | k_5 |^{- 3 +
     1 / 4} 
    \lesssim t^{\delta} 2^{( 2 + \eta') q}.
  \end{align*}
  Hence there exists $\delta > 0$ and $\nu > 0$ such that
  \[ E^{1, 2}_t \lesssim t^{\delta} 2^{( 2 + \nu) q}. \]

  \subparagraph{Estimation of $E^{1, 3}_t$.}
  
  Let us deal with this last term. Here the symmetries are in $k_2, k_3$ and
  $k_4, k_5$. Then we can suppose that $| k_2 | \ge | k_3 | \ge$
  and $| k_4 | \ge | k_5 |$. Furthemore, the role of $k_2, k_3$ and
  $k_4, k_5$ are symmetrical, then we can assume that $| k_1 | \ge | k_4
  |$, and we have
 \begin{multline*} E^{1, 3}_t = \int_{[ 0, t]^2} \mathd s \mathd \overline{s}
     \sum_{\tmscript{\begin{array}{c}
       k \neq 0, k_i \neq 0\\
       \begin{array}{c}
         {}k_{12345} = k
       \end{array}
     \end{array}}} \theta ( 2^{- q} k)^2 \prod_{i = 1}^5 \frac{f ( \varepsilon
     k_i)^2}{| k_i |^2} \\
     \times e^{- ( | k_{123} |^2 + | k_2 |^2 + | k_3 |^2) | t -
     s |} e^{- ( | k_{145} |^2 + | k_5 |^2 + | k_4 |^2) | t - \overline{s}
     |} e^{- | s - \overline{s} | | k_1 |^2}. 
     \end{multline*}
  If $| k_1 | = \max ( | k_i |)$ then
\begin{align*}
&  \lesssim \int_{[ 0, t]^2} \mathd s \mathd \overline{s} ( | t - s | | t -
     \overline{s} |)^{- 1 + \eta} \sum_{\tmscript{\begin{array}{c}
       k \neq 0, k_i \neq 0\\
       \begin{array}{c}
         {}k_{12345} = k
       \end{array}
     \end{array}}} \theta ( 2^{- q} k)^2 | k |^{- 1 + \eta} ( | k_2 | | k_3 |
     | k_3 | | k_4 |)^{- 7 / 4 + 3 \eta / 4} \\ 
 & \lesssim 2^{q ( 2 + \eta)} t^{\eta}. 
\end{align*}
  If $| k_2 | = \max ( | k_i |)$ then
\begin{align*}
 &\lesssim 2 \int_0^t \int_0^s \mathd s \mathd \overline{s} ( | t - s | | s
     - \overline{s} |)^{- 1 + \eta} \sum_{\tmscript{\begin{array}{c}
       k \neq 0, k_i \neq 0\\
       \begin{array}{c}
         {}k_{12345} = k
       \end{array}
     \end{array}}} \theta ( 2^{- q} k)^2 | k |^{- 1 + \eta} ( | k_1 | | k_3 |
     | k_3 | | k_4 |)^{- 7 / 4 + 3 \eta / 4} \\ 
 &\lesssim t^{\eta} 2^{q ( 2 + \eta)}. 
\end{align*}
  \end{proof}
  
  \appendix
\section{A commutation lemma}

We give the proof of  Lemma~\ref{lemma:Heat-flow-smoothing}. This proof is from \cite{these_perkowski} Lemmas 5.3.20 and 5.5.7. In fact we give a stronger result, and apply it with $\varphi(k) = \exp(-|k|^2/2)$.
\begin{lemma}\label{l:multiplier paraproduct commutator}
   Let $\alpha < 1$ and $\beta \in \R$. Let $\varphi \in \mathcal{S}$, let $u \in C^{\alpha}$, and $v \in C^{\beta}$. Then for every $\varepsilon > 0$ and every $\delta \ge - 1$ we have
   \begin{align*}
      \Vert  \varphi (\varepsilon \dD) \pi_< ( u, v ) - \pi_< (u, \varphi(\varepsilon \dD) v)\Vert_{\alpha + \beta + \delta} \lesssim \varepsilon^{- \delta} \Vert  u \Vert_{\alpha} \Vert  v \Vert_{\beta}, 	
   \end{align*}
   where
         \[\varphi(\cD)u=\cF^{-1}(\varphi \cF u).\]
\end{lemma}

\begin{proof}
We define for $j\ge -1$, 
\[S_{j-1}u=\sum_{i=-1}^{j-2}\Delta_i u\]
and
   \begin{align*}
      \varphi (\varepsilon \dD) \pi_< (u, v) - \pi_<(u, \varphi(\varepsilon \dD) v ) = \sum_{j\ge - 1} (\varphi(\varepsilon \dD)(S_{j-1} u \Delta_j v ) - S_{j-1} u \Delta_j \varphi(\varepsilon \dD)v),
   \end{align*}
   where every term of this series has a Fourier transform with support in an annulus of the form $2^j \mathcal{A}$. Lemma~2.69 in \cite{BCD-bk} implies that it is enough to control the $L^{\infty}$ norm of each term. Let $\psi \in \mathcal{D}$ with support in an annulus be such that $\psi \equiv 1$ on $\mathcal{A}$. We have
   \begin{align*}
      &\varphi(\varepsilon \dD)(S_{j-1} u \Delta_j v ) - S_{j-1} u \Delta_j \varphi(\varepsilon \dD)v \\
      &\hspace{40pt} = (\psi(2^{- j} \cdot) \varphi(\varepsilon \cdot))(\dD) (S_{j-1} u \Delta_j v) - S_{j-1} u(\psi(2^{- j} \cdot) \varphi(\varepsilon \cdot)) (\dD) \Delta_j v \\
      &\hspace{40pt} = [(\psi(2^{- j} \cdot) \varphi(\varepsilon \cdot))(\dD), S_{j-1} u ]\Delta_j v,
   \end{align*}
   where $[(\psi(2^{- j} \cdot) \varphi(\varepsilon \cdot))(\dD), S_{j-1} u ]$ denotes the commutator. In the proof of Lemma 2.97 in \cite{BCD-bk}, it is shown that writing the Fourier multiplier as a convolution operator and applying a first order Taylor expansion and then Young inequality yields to
   \begin{align}\label{e:multiplier paraproduct commutator pr1}\nonumber
      &\Vert [(\psi(2^{- j} \cdot) \varphi(\varepsilon \cdot))(\dD), S_{j-1} u ] \Delta_j v \Vert_{L^{\infty}}\\
      &\hspace{40pt} \lesssim \sum_{\eta \in \N^d: |\eta| = 1} \Vert  x^{\eta} \cF^{- 1} ( \psi( 2^{- j} \cdot ) \varphi (\varepsilon \cdot)) \Vert_{L^1} \Vert  \partial^{\eta} S_{j - 1} u \Vert_{L^{\infty}} \Vert  \Delta_j v \Vert_{L^{\infty}} .
   \end{align}
   Now $\cF^{- 1}(f(2^{- j} \cdot) g(\varepsilon \cdot)) = 2^{j d} \cF^{- 1} ( f g( \varepsilon 2^j \cdot ) ) ( 2^j \cdot )$ for every $f, g$, and thus we have for every multi-index $\eta$ of order one
   \begin{align}\label{e:multiplier paraproduct commutator pr4} \nonumber
      &\Vert  x^{\eta} \cF^{- 1} ( \psi ( 2^{- j} \cdot) \varphi (\varepsilon \cdot)) \Vert_{L^1}\\ \nonumber
      &\hspace{40pt} \le 2^{- j} \Vert  \cF^{-1}((\partial^{\eta} \psi) (2^{- j} \cdot) \varphi ( \varepsilon\cdot ) ) \Vert_{L^1} + \varepsilon \Vert \cF^{- 1} ( \psi ( 2^{- j} \cdot) \partial^{\eta} \varphi(\varepsilon \cdot))\Vert_{L^1} \\ \nonumber
      &\hspace{40pt} = 2^{- j} \Vert  \cF^{- 1}((\partial^{\eta}\psi ) \varphi ( \varepsilon 2^j \cdot)) \Vert_{L^1} + \varepsilon \Vert  \cF^{- 1} ( \psi \partial^{\eta} \varphi ( \varepsilon 2^j \cdot ) )\Vert_{L^1} \\ \nonumber
      &\hspace{40pt} \lesssim 2^{- j} \Vert  ( 1 + | \cdot | )^{2 d} \cF^{- 1}((\partial^{\eta} \psi) \varphi(\varepsilon 2^j \cdot)) \Vert_{L^{\infty}} + \varepsilon \Vert  ( 1 + | \cdot | )^{2 d} \cF^{- 1} ( \psi \partial^{\eta} \varphi(\varepsilon 2^j \cdot)) \Vert_{L^{\infty}}\\ \nonumber
      &\hspace{40pt} = 2^{- j} \Vert  \cF^{- 1} ( ( 1 - \Delta )^d((\partial^{\eta} \psi) \varphi(\varepsilon 2^j \cdot)))\Vert_{L^{\infty}} + \varepsilon \Vert \cF^{- 1}((1 - \Delta)^d(\psi \partial^{\eta} \varphi ( \varepsilon 2^j \cdot)) \Vert_{L^{\infty}}\\
      &\hspace{40pt} \lesssim 2^{- j} \Vert (1 - \Delta)^d((\partial^{\eta} \psi) \varphi(\varepsilon 2^j \cdot))\Vert_{L^{\infty}} + \varepsilon \Vert(1 - \Delta)^d (\psi \partial^{\eta} \varphi(\varepsilon 2^j \cdot)) \Vert_{L^{\infty}},
   \end{align}
   where the last step follows because $\psi$ has compact support. For $j$ satisfying $\varepsilon 2^j \ge 1$ we obtain
   \begin{align}\label{e:multiplier paraproduct commutator pr2}
      \Vert  x^{\eta} \cF^{- 1} (\varphi (\varepsilon \cdot) \psi( 2^{- j} \cdot ) ) \Vert_{L^1} \lesssim (\varepsilon + 2^{- j}) (\varepsilon 2^j )^{2 d} \sum_{\eta: |\eta| \le 2 d + 1} \Vert \partial^{\eta} \varphi(\varepsilon 2^j \cdot) \Vert_{L^{\infty}(\mathrm{supp}(\psi))},
   \end{align}
   where we used that $\psi$ and all its partial derivatives are bounded, and where $L^{\infty}(\mathrm{supp} (\psi))$ means that the supremum is taken over the values of $\partial^{\eta}\varphi(\varepsilon 2^j \cdot)$ restricted to $\mathrm{supp}(\psi)$. Now $\varphi$ is a Schwartz function, and therefore it decays faster than any polynomial. Hence, there exists a ball $\mathcal{B}_{\delta}$ such that for all $x \notin \mathcal{B}_{\delta}$ and all $|\eta|\le 2 d + 1$ we have
   \begin{align}\label{e:multiplier paraproduct commutator pr3}
      |\partial^{\eta}\varphi ( x ) | \le | x |^{-2d-1-\delta}.
   \end{align}
   Let $j_0 \in \N$ be minimal such that $2^{j_0} \varepsilon \mathcal{A} \cap \mathcal{B}_{\delta} = \emptyset$ and $\varepsilon 2^{j_0} \ge 1$. Then the combination of \eqref{e:multiplier paraproduct commutator pr1}, \eqref{e:multiplier paraproduct commutator pr2}, and \eqref{e:multiplier paraproduct commutator pr3} shows that for all $j \ge j_0$,
   \begin{align*}
      & \Vert [(\psi(2^{- j} \cdot) \varphi(\varepsilon \cdot))(\dD), S_{j-1} u ] \Delta_j v \Vert_{L^{\infty}} \\
      &\hspace{40pt} \lesssim ( \varepsilon + 2^{- j} ) ( \varepsilon 2^j)^{2 d} \sum_{\eta : |\eta| \le 2 d + 1} \Vert (\partial^{\eta} \varphi) ( \varepsilon 2^j \cdot) \Vert_{L^{\infty}(\mathrm{supp}(\psi))} 2^{j ( 1 - \alpha )} \Vert  u \Vert_{\alpha} 2^{-j \beta} \Vert  v \Vert_{\beta} \\
      &\hspace{40pt} \lesssim (\varepsilon + 2^{-j})(\varepsilon 2^j)^{2d}(\varepsilon 2^j)^{-2d -1 -\delta} 2^{j(1 -\alpha -\beta)} \Vert u \Vert_{\alpha} \Vert v \Vert_{\beta}\\
      &\hspace{40pt} \lesssim ( 1 + ( \varepsilon 2^j )^{- 1}) \varepsilon^{- \delta} 2^{- j(\alpha + \beta + \delta)} \Vert u \Vert_{\alpha} \Vert  v \Vert_{\beta} .
   \end{align*}
   Here we used that $\alpha < 1$ in order to obtain $\Vert  \partial^{\eta} S_{j-1} u \Vert_{L^{\infty}} \lesssim 2^{j(1-\alpha)} \Vert u \Vert_{L^{\infty}}$. Since $\varepsilon 2^j \ge 1$, we have shown the desired estimate for $j \ge j_0$. On the other side Lemma~2.97 in \cite{BCD-bk} implies that for every $j \ge -1$
   \begin{align*}
      \Vert [\varphi ( \varepsilon \dD), S_{j-1}u] \Delta_j v \Vert_{L^{\infty}} \lesssim \varepsilon \max_{\eta \in \N^d: |\eta|=1} \Vert \partial^\eta S_{j - 1} u    \Vert_{L^{\infty}} \Vert  \Delta_j v \Vert_{L^{\infty}} \lesssim \varepsilon 2^{j ( 1 - \alpha - \beta )} \Vert u \Vert_{\alpha} \Vert  v \Vert_{\beta}.
   \end{align*}
   Hence, we obtain the bound for $j < j_0$, i.e. for $j$ satisfying $2^j \varepsilon \lesssim 1$,
   \begin{align*}
      \Vert [\varphi ( \varepsilon \dD), S_{j - 1} u ] \Delta_j v \Vert_{L^{\infty}} \lesssim (\varepsilon 2^j)^{1 + \delta} \varepsilon^{- \delta} 2^{- j(\alpha + \beta + \delta )} \Vert u \Vert_{\alpha} \Vert  v \Vert_{\beta}\lesssim \varepsilon^{- \delta} 2^{- j ( \alpha + \beta + \delta)} \Vert  u \Vert_{\alpha} \Vert  v \Vert_{\beta},
   \end{align*}
   where we used that $\delta \ge - 1$. This completes the proof.
\end{proof}


\begin{bibdiv}
    \begin{biblist}

\bib{BCD-bk}{book}{
author={H. Bahouri},
author={J-Y. Chemin},
author={R. Danchin},
title={\it Fourier analysis and nonlinear partial differential equations.},
 series={Grundlehren der Mathematischen Wissenschaften [Fundamental Principles of Mathematical Sciences].},
 volume={343},
 publisher={Springer},
 place={Heidelberg},
 date={2011}}
  
  
\bib{bertini_stochastic_1993}{article}{
	title = {Stochastic Quantization, Stochastic Calculus and Path Integrals: Selected Topics},
	volume = {111},
	journal = {Progress of Theoretical Physics Supplement},
	author = {Bertini, L.},
	author = {Jona-Lasinio, G.},
	author = {Parrinello, C.},
	year = {1993},
	pages = {83--113},
}




\bib{da_prato_strong_2003}{article}{
	title = {Strong solutions to the stochastic quantization equations},
	volume = {31},
	number = {4},
	journal = {The Annals of Probability},
	author = {Da Prato, G.},
	author = {Debussche, A.},
	year = {2003},
	pages = {1900--1916},
}



\bib{grr}{article}{
author={A. M. Garsia}, 
author={E. Rodemich}, 
author={H. Rumsey},
title={A real variable lemma and the continuity of paths of some Gaussian processes}, 
journal={Indiana Univiversity Mathematical Journal},
volume={20},
year={1970},
number={6}, 
pages={565--578},
}



\bib{gub}{article}{
  author={Gubinelli, M.},
  author={Perkowski, N.},
 title={KPZ reloaded},
 url={http://arxiv.org/abs/1508.03877},
 journal = {{arXiv:1508.03877}},
 year={2015},
 note={preprint}
}
 

\bib{gubinelli_controlling_2004}{article}{
	title = {Controlling rough paths},
	volume = {216},
	number = {1},
	journal = {Journal of Functional Analysis},
	author = {Gubinelli, M.},
	year = {2004},
	pages = {86--140},
	}

\bib{gubinelli_paraproducts_2012}{article}{
	title = {Paracontrolled distributions and singular PDEs},
	journal = {Forum of Mathematics, Pi},
	author = {Gubinelli, M.},
	author={Imkeller, P.},
	author={Perkowski, N.},
	year = {2015},
	volume={3},
	number={6}
}

\bib{hairer_theory_2013}{article}{	
year={2014},
journal={Inventiones Mathematicae},
title={A theory of regularity structures},
author={Hairer, M.},
pages={269--504},
volume={198},
number={2}

}


\bib{HW}{article}{
	title = {Large deviations for white-noise driven, nonlinear stochastic PDEs in two and three dimensions},
	journal={Annales de la Faculté des Sciences de Toulouse},
	volume={24},
	number={1},
	year={2015},
	pages={55-92},
	author = {Hairer, M.},
	author = {Weber, H.},
	year = {2014},
	}

\bib{hairer_solving_2013}{article}{
	title = {Solving the {KPZ} equation},
	volume = {178},
	number = {2},
	Journal= {Annals of Mathematics},
	author = {Hairer, M.},
	year = {2013},
	pages = {559--664}
}
 \bib{janson}{book}{,
	title = {Gaussian Hilbert Spaces},
	publisher = {Cambridge University Press},
	author = {Janson, S.},
	year = {1997},
}



\bib{jona-lasinio_stochastic_1985}{article}{
	title = {On the stochastic quantization of field theory},
	volume = {101},
	issn = {0010-3616, 1432-0916},
	url = {http://projecteuclid.org/euclid.cmp/1104114183},
	number = {3},
	journal = {Communications in Mathematical Physics (1965-1997)},
	author = {Jona-Lasinio, G.},
	author={Mitter, P. K.},
	year = {1985},
	pages = {409--436},
}

\bib{jona-lasinio_large_1990}{article}{
	title = {Large deviation estimates in the stochastic quantization of $\Phi^4_2$},
	volume = {130},
	issn = {0010-3616, 1432-0916},
	url = {http://projecteuclid.org/euclid.cmp/1104187931},
	number = {1},
	journal = {Communications in Mathematical Physics},
	author = {Jona-Lasinio, G.},
	author={Mitter, P. K.},
	year = {1990},
	pages = {111--121},
}

\bib{HM1}{article}{
	title = {Global well-posedness of the dynamic $\Phi^4$ model in the plane},
	journal = {The Annals of Probability},
		author={Mourrat, J.C.},
	author = {Weber, H.},
	year = {2016},
           note={Accepted for publication},
}
\bib{HM2}{article}{
	title = {Global well-posedness of the dynamic $\Phi_4^3$ model on the torus },
	journal = {arXiv:1601.01234},
		author={Mourrat, J.C.},
	author = {Weber, H.},
	year = {2015},
           note={preprint},
}



\bib{lyons_differential_2004}{book}{
	title = {Differential equations driven by rough paths},
	url = {http://sag.maths.ox.ac.uk/tlyons/st_flour/StFlour.pdf},
	serie = {Ecole d'été de Probabilités de Saint-Flour {XXXIV}},
	author = {Lyons, T.},
	author = {Caruana, M.},
	author =  {L\'evy, T.},
	publisher={Springer},
	volume={1908},
	place={Heidelberg},
	year={2004}}

\bib{these_perkowski}{thesis}{
    author={Perkowski, N.},
	title = {Studies of Robustness in Stochastic Analysis and Mathematical Finance},
	publisher={Humboldt-Universit\"at zu Berlin},
	note = {PhD Thesis},
	year={2013}
	}
 
\end{biblist}
\end{bibdiv}
\end{document}